\documentclass[11pt]{article}
\usepackage{mathtools}
\usepackage{amsmath,amsthm}
\usepackage{csquotes}
\usepackage{amssymb}
\usepackage{graphicx}
\usepackage{subfig}
\usepackage{mathpazo}
\usepackage{soul}  
\usepackage{xcolor}
\usepackage{empheq}
\usepackage{lipsum}

\definecolor{myblue}{rgb}{.8, .8, 1}
  \newcommand*\mybluebox[1]{
    \colorbox{myblue}{\hspace{1em}#1\hspace{1em}}}

\usepackage[obeyspaces,hyphens,spaces]{url}
\usepackage{hyperref}
\hypersetup{
    colorlinks=true,
    linkcolor=blue,
    citecolor=blue,
    filecolor=magenta,
    urlcolor=cyan
}

\usepackage[
  open,
  openlevel=2,
  atend,
  numbered
]{bookmark}

\usepackage[capitalize, nameinlink, noabbrev]{cleveref}
\crefname{equation}{}{}
\crefname{chapter}{Chapter}{Chapters}
\crefname{item}{item}{items}
\crefname{figure}{Figure}{Figures}
\crefname{theorem}{Theorem}{Theorems}
\crefname{lemma}{Lemma}{Lemmas}
\crefname{proposition}{Proposition}{Propositions}
\crefname{corollary}{Corollary}{Corollarys}
\crefname{definition}{Definition}{Definitions}
\crefname{fact}{Fact}{Facts}
\crefname{example}{Example}{Examples}
\crefname{algorithm}{Algorithm}{Algorithms}
\crefname{remark}{Remark}{Remarks}
\crefname{note}{Note}{Notes}
\crefname{notation}{Notation}{Notations}
\crefname{case}{Case}{Cases}
\crefname{exercise}{Exercise}{Exercises}
\crefname{question}{Question}{Questions}
\crefname{claim}{Claim}{Claims}
\crefname{enumi}{}{}

\usepackage[top= 2cm, bottom = 2 cm, left = 2.2 cm, right= 2.2 cm]{geometry}  
\numberwithin{equation}{section}

\theoremstyle{plain}
\newtheorem{theorem}{Theorem}[section]

\newtheorem{fact}[theorem]{Fact}
\newtheorem{lemma}[theorem]{Lemma}
\newtheorem{proposition}[theorem]{Proposition}

\theoremstyle{definition}
\newtheorem{definition}[theorem]{Definition}
\newtheorem{example}[theorem]{Example}

\newtheorem{remark}[theorem]{Remark}

\usepackage{enumitem}

\newcommand{\zer}{\ensuremath{\operatorname{zer}}}

\newcommand{\weakly}{\ensuremath{{\;\operatorname{\rightharpoonup}\;}}}

\newcommand{\dom}{\ensuremath{\operatorname{dom}}}
\newcommand{\gra}{\ensuremath{\operatorname{gra}}}
\newcommand{\Fix}{\ensuremath{\operatorname{Fix}}}
\newcommand{\Id}{\ensuremath{\operatorname{Id}}}

\newcommand{\Pro}{\ensuremath{\operatorname{P}}}

\newcommand{\J}{\ensuremath{\operatorname{J}}}

\providecommand{\abs}[1]{\left|#1\right|}
\providecommand{\norm}[1]{\left\lVert#1\right\rVert}
\providecommand{\innp}[1]{\left\langle#1\right\rangle}
\providecommand{\lr}[1]{\left(#1\right)}

\begin{document}

\title{Weak and Strong Convergence of Generalized Proximal Point Algorithms with Relaxed Parameters}

\author{
         Hui\ Ouyang\thanks{
                 Mathematics, University of British Columbia, Kelowna, B.C.\ V1V~1V7, Canada.
                 E-mail: \href{mailto:hui.ouyang@alumni.ubc.ca}{\texttt{hui.ouyang@alumni.ubc.ca}}.}
                 }

\date{March 27, 2022}

\maketitle

\begin{abstract}
	\noindent
In this work, we propose and study a framework of generalized proximal point algorithms associated with a maximally monotone operator. We indicate sufficient conditions on the regularization and relaxation parameters of  generalized proximal point algorithms for the equivalence of the boundedness of the sequence of iterations generated by this algorithm and the non-emptiness of the zero set of the maximally monotone operator, and for the weak and strong convergence of the algorithm. Our results cover or improve many results on generalized proximal point algorithms in our references. Improvements of our results are illustrated by comparing our results with related known ones. 
\end{abstract}

{\small

}

\section{Introduction} \label{sec:Introduction}

Throughout this paper,  
\begin{empheq}[box = \mybluebox]{equation*}
\text{$\mathcal{H}$ is a real Hilbert space},
\end{empheq}
with inner product $\innp{\cdot,\cdot}$ and induced norm $\|\cdot\|$.
Moreover, we assume that $\mathcal{H} \neq \{0\}$ and that $m \in \mathbb{N}
\smallsetminus \{0\}$, where $\mathbb{N}=\{0,1,2,\ldots\}$.
 
In 1976, Rockafellar, in the seminal work \cite{Rockafellar1976}, generalized the proximal point algorithm for minimizing lower semicontinuous proper convex functions by weakening the exact minimization at each iteration and by replacing the subgradient mapping with an arbitrary maximally monotone operator. In particular, Rockafellar's proximal point algorithm solves the fundamental problem:
\begin{empheq}[box = \mybluebox]{equation}\label{eq:problem}
\text{determine an element $x \in \mathcal{H}$ ~s.t.~ $0 \in A(x)$, where $A:\mathcal{H} \to 2^{\mathcal{H}}$ is maximally monotone},
\end{empheq}
 which includes minimization problems subject to implicit constraints, variational inequality problems, and minimax problems as special cases (see \cite{Rockafellar1976} and the references therein for details).  For example, given a proper lower semicontinuous and convex function $f$,   it is well-known that $\partial f $ is maximally monotone (see, e.g., \cite[Theorem~20.25]{BC2017}),   that if there exist a closed convex subset $C$ of $\mathcal{H}$ and $\xi \in \mathbb{R}$ such that 
the set $\{ x \in C ~:~ f(x) \leq \xi \}$
  is nonempty and bounded, then $\zer \partial f \neq \varnothing$ (see, e.g., \cite[Theorems~11.10 and 16.3]{BC2017} for details) and solving \cref{eq:problem} with $A=\partial f$  is equivalent to finding the minimizer of $f$.
 
Synthesizing the work of Rockafellar \cite{Rockafellar1976} with that of Gol'shtein and Tret'yakov \cite{GolcprimeTretcprime1979},  Eckstein and Bertsekas  in \cite{EcksteinBertsekas1992} proposed a generalized form of the proximal point algorithm and elaborated that  the Douglas-Rachford splitting algorithm is a special case of the proximal point algorithm. In addition, because generally  the proximal point algorithm   converges weakly but not strongly   (see, e.g., \cite{Guler1991} for details), various modified  proximal point algorithms were studied in many articles (see, e.g., \cite{BoikanyoMorosanu2010}, \cite{CormanYuan2014}, \cite{EcksteinBertsekas1992}, \cite{MarinoXu2004}, \cite{SolodovSvaiter2000}, \cite{Xu2002}, \cite{Xu2006}, \cite{YaoNoor2008}, and \cite{YaoShahzad2012}) to obtain the strong convergence.
 
 Henceforth, 
 \begin{empheq}[box = \mybluebox]{equation*}
 \text{$A:\mathcal{H} \to 2^{\mathcal{H}}$ is maximally monotone.}
 \end{empheq}
 Then, via \cite[Proposition~20.22]{BC2017}, $(\forall \gamma \in \mathbb{R}_{++})$ $\gamma A$ is maximally monotone.  
 In the whole work, given a point $u \in \mathcal{H}$, we investigate the sequence  $(x_{k})_{k \in \mathbb{N}}$ of iterations  generated by the \emph{generalized proximal point algorithm with relaxed parameters}:
\begin{empheq}[box = \mybluebox]{equation} \label{eq:xk+1}
 (\forall k \in \mathbb{N}) \quad x_{k+1} := \alpha_{k} u +\beta_{k} x_{k} + \gamma_{k} \J_{c_{k} A} (x_{k}) +\delta_{k} e_{k},
 \end{empheq}
 where  $x_{0} \in \mathcal{H}$ is the \emph{initial point} and  $(\forall k \in \mathbb{N})$ $e_{k} \in \mathcal{H}$ is   the \emph{error term}, 
 $c_{k} \in \mathbb{R}_{++}$ is the \emph{stepsize} or \emph{regularization parameter}, 
 and $ \alpha_{k}$,  $\beta_{k}$, $\gamma_{k}$, and $\delta_{k} $ are the  \emph{relaxation parameters} in $  \mathbb{R}$.  For simplicity, in this work, we refer to generalized proximal point algorithms with relaxed parameters as generalized proximal point algorithms.

We compare the scheme \cref{eq:xk+1} with some known proximal point algorithms in the literatures below. 
\begin{enumerate}
	\item Suppose that  $(\forall k \in \mathbb{N})$ $\alpha_{k} =\beta_{k} \equiv 0$ and  $\gamma_{k} =\delta_{k} \equiv 1$.
	 Then \cref{eq:xk+1} reduces to the proximal point algorithm devised by Rockafellar in \cite{Rockafellar1976}.
	\item  Suppose that  $(\forall k \in \mathbb{N})$  $\alpha_{k}  \equiv 0$,    $\gamma_{k} \in \left[0,2\right]$, and  $\beta_{k} =1-\gamma_{k}$.
	Then \cref{eq:xk+1} turns to the generalized proximal point algorithm developed by Eckstein and Bertsekas  in \cite{EcksteinBertsekas1992}.
	\item Suppose that $(\forall k \in \mathbb{N})$ $\alpha_{k} =\delta_{k} \equiv 0$, $\gamma_{k} \equiv  \gamma$, and $\beta_{k} \equiv 1-\gamma$ where $\gamma  \in \mathbb{R}_{++}$. Then \cref{eq:xk+1} reduces to the generalized proximal point algorithm scheme proposed by Corman and Yuan in \cite{CormanYuan2014}. In particular, Corman and Yuan provided examples where the generalized proximal point algorithm scheme with $\gamma >2$ converges faster than that with $\gamma \in \left]0,2\right]$.
	
	\item  Suppose that $u=x_{0}$ and $(\forall k \in \mathbb{N})$ $\alpha_{k}  \in \left[0,1\right]$, $\beta_{k} \equiv 0$, and  $\gamma_{k} =\delta_{k} =1-\alpha_{k}$, or that $u=x_{0}$ and $(\forall k \in \mathbb{N})$ $\alpha_{k} \equiv 0$, $\beta_{k} \in \left[0,1\right]$, and $\gamma_{k} =\delta_{k} =1-\beta_{k}$. Then \cref{eq:xk+1} becomes the modified proximal point algorithms introduced by Xu in \cite{Xu2002}.
	
	\item Suppose that $(\forall k \in \mathbb{N})$ $\alpha_{k}  \in \left]0,1\right]$,  $\beta_{k} \equiv 0$, $\gamma_{k} =1-\alpha_{k}$, and $\delta_{k} \equiv 1$. Then  \cref{eq:xk+1} turns to the contraction-proximal point algorithm introduced by Marino and Xu in 
	\cite{MarinoXu2004}.  Note that by some natural substitution one can easily see that the regularization method for the proximal point algorithm proposed by Xu in \cite{Xu2006} is equivalent to the contraction-proximal point algorithm and hence a special case of the scheme \cref{eq:xk+1} as well. Moreover, as it is verified in  \cite{Xu2006}, the prox-Tikhonov algorithm of Lehdili and Moudafi \cite{LehdiliMoudafi1996} deals essentially with a  special case of the regularization method for the proximal point algorithm of Xu in \cite{Xu2006}, which in turn shows that  \cref{eq:xk+1} also covers the prox-Tikhonov algorithm in \cite{LehdiliMoudafi1996}.
	
	\item Suppose that $(\forall k \in \mathbb{N})$ $\delta_{k} \equiv 1 $ and $\{ \alpha_{k},  \beta_{k} , \gamma_{k} \} \subseteq \left]0,1\right[$ with $\alpha_{k}+ \beta_{k} +\gamma_{k} =1$. Then  \cref{eq:xk+1} deduces  the contraction proximal point algorithm proposed by Yao and Noor in \cite{YaoNoor2008}.
	\item Suppose that $u=0$ and that $(\forall k \in \mathbb{N})$  $ \alpha_{k} \equiv 0$ and $\{    \beta_{k} , \gamma_{k}, \delta_{k} \} \subseteq \left]0,1\right[$ with $ \beta_{k} +\gamma_{k}+\delta_{k} =1$. Then  \cref{eq:xk+1} becomes the proximal point algorithm with general errors constructed by Yao and Shahzad in  \cite{YaoShahzad2012}.
\end{enumerate}

For the  generalized proximal point algorithm conforming the recursion \cref{eq:xk+1}, the advantage  of considering the  range $\mathbb{R}$ of the parameters  $(\alpha_{k})_{k \in \mathbb{N}}$, $(\beta_{k})_{k \in \mathbb{N}}$, $(\gamma_{k})_{k \in \mathbb{N}}$, and $(\delta_{k})_{k \in \mathbb{N}}$ is suggested by \cite{CormanYuan2014}; 
 the necessity of the consideration of the coefficient $(\delta_{k})_{k \in \mathbb{N}}$ preceding the error terms  is illustrated by  \cite{Xu2002} and \cite{YaoShahzad2012}; and the term $u$ is  motivated by \cite{Xu2002} and \cite{MarinoXu2004}.

\emph{The goal of this work is to explore the equivalence of the boundedness of $(x_{k})_{k \in \mathbb{N}}$ generated by \cref{eq:xk+1} and $\zer A \neq \varnothing$  and  to deduce sufficient conditions for the weak and strong convergence of the scheme \cref{eq:xk+1} for solving \cref{eq:problem} when $\zer A \neq \varnothing$.}

Main results of this work are the following. 
\begin{enumerate}
	\item[\textbf{R1:}]  
\Cref{Theorem:ZerBounded,theorem:ZerBoundedIFF} present requirements on the regularization and relaxation parameters  of  \cref{eq:xk+1} for the equivalence of 
 the non-emptiness of $\zer A$ and the boundedness of the sequence of iterations generated by the scheme \cref{eq:xk+1}.  
  
  	\item[\textbf{R2:}] The weak convergence of  the generalized proximal point algorithms is illustrated in \cref{theorem:WeakConverg}.
		
	\item[\textbf{R3:}]    \Cref{theorem:StrongConvergence,theorem:StrongConvergence:YS} exhibit sufficient conditions for the strong convergence of the sequence of iterations conforming the scheme \cref{eq:xk+1}.
\end{enumerate}

In  \Cref{remark:WeakConvergence,reamark:StrongConvergence} below, we shall compare our convergence results with  related known results in   references mentioned above and demonstrate our improvements.

The paper is organized as follows.   In  \cref{sec:Preliminaries}, we provide some fundamental and essential results for proving the convergence of generalized proximal point algorithms.
The boundedness and asymptotic regularity of the sequence of iterations generated by the generalized proximal point algorithm is elaborated in \cref{sec:GeneralizedPPAs}.   The equivalence of the boundedness of this sequence and $\zer A \neq \varnothing$ is also established in \cref{sec:GeneralizedPPAs}.  Convergence results are exhibited in the last section, \cref{sec:Convergence}.

We now turn to the notation used in this work.  $\Id$ stands for the \emph{identity mapping}.  
Denote by  $\mathbb{R}_{+}:=\{\lambda \in \mathbb{R} ~:~ \lambda \geq 0 \}$ and $\mathbb{R}_{++}:=\{\lambda \in \mathbb{R} ~:~ \lambda >0 \}$. 
Let $\bar{x} $ be in $ \mathcal{H}$, let $r \in \mathbb{R}_{+}$, and let $(x_{k})_{k \in \mathbb{N}}$ be a sequence in $\mathcal{H}$. $B(\bar{x};r) := \{ y\in \mathcal{H} ~:~ \norm{y-\bar{x}} <r \}$ and $B[\bar{x};r]:= \{ y\in \mathcal{H} ~:~ \norm{y-\bar{x}} \leq r \}$ are the \emph{open and closed ball centered at $\bar{x}$ with radius $r$}, respectively. 
 If  $(x_{k})_{k \in \mathbb{N}}$ \emph{converges strongly} to $\bar{x}$, then we denote by $x_{k} \to \bar{x}$.  $(x_{k})_{k \in \mathbb{N}}$ \emph{converges weakly} to  $\bar{x} $ if, for every $y \in \mathcal{H}$, $\innp{x_{k},y} \rightarrow \innp{x,y}$; in symbols, $x_{k} \weakly \bar{x}$. 
Let $C$ be a nonempty closed convex subset of $\mathcal{H}$. The \emph{projector} (or \emph{projection operator}) onto $C$ is the operator, denoted by $\Pro_{C}$,  that maps every point in $\mathcal{H}$ to its unique projection onto $C$.  $\iota_{C}$ is the \emph{indicator function of $C$}, that is, $(\forall x \in C)$ $\iota_{C} (x) =0$ and $(\forall x \in \mathcal{H} \smallsetminus C)$ $\iota_{C} (x) = \infty$.  Let $f:\mathcal{H} \to \left]-\infty, \infty\right]$ be  \emph{proper}, i.e., $\dom f \neq \varnothing$.
The \emph{subdifferential of $f$} is the set-valued operator $\partial f :\mathcal{H} \to 2^{\mathcal{H}} : x \mapsto \left\{ z \in \mathcal{H} ~:~ (\forall y \in \mathcal{H}) \innp{z,y-x} \leq f(y) -f(x)  \right\}$.
Let $\mathcal{D} $ be a nonempty subset of $\mathcal{H}$ and let $T: \mathcal{D} \to  \mathcal{H}$. $\Fix T :=\{ x \in \mathcal{D}~:~ x = T(x) \}$ is the \emph{set of fixed points  of $T$}. 
Let	$G: \mathcal{H} \to 2^{\mathcal{H}}$ be a set-valued operator. Then $G$ is characterized by its \emph{graph} $\gra G:= \{ (x,w) \in \mathcal{H} \times \mathcal{H} ~:~ w\in G(x) \}$. The \emph{inverse} of $G$, denoted by $G^{-1}$, is defined through its graph $\gra G^{-1} :=  \{ (w,x) \in \mathcal{H} \times \mathcal{H} ~:~ (x,w) \in \gra G \}$.
The \emph{set of zeros of $G$} is $\zer G: =G^{-1}(0) = \{ x \in \mathcal{H} ~:~ 0 \in G (x) \} $.
$G$ is \emph{monotone} if  $(\forall (x,u) \in \gra G)$ $(\forall (y,v) \in \gra G)$ $\innp{x-y,u-v} \geq 0$. $G$ is \emph{maximally monotone} if there exists no monotone operator $B: \mathcal{H} \to 2^{\mathcal{H}} $ such that $\gra B$ properly contains $\gra G$, i.e., for every $(x,u) \in \mathcal{H} \times \mathcal{H}$,  $(x,u) \in \gra G$ if and only if $(\forall (y,v) \in \gra G) $ $ \innp{ x-y, u-v} \geq 0$.

For other notation not explicitly defined here, we refer the reader to \cite{BC2017}.

\section{Preliminaries} \label{sec:Preliminaries}
In order to facilitate our investigation in the following sections, we gather some auxiliary results in this section. 
The ideas of these  results are frequently used in  proofs of the convergence of   generalized  proximal point algorithms in references of this work. 
Clearly,  results in this section are interesting in their own right and are helpful to study various generalized proximal point algorithms.

\subsection*{Limits of sequences}
\begin{fact} {\rm \cite[Lemma~2.5]{Xu2002}} \label{fact:sequence:sk}
	Let $(s_{k})_{k \in \mathbb{N}}$ be a sequence in $\mathbb{R}_{+}$ satisfying
	\begin{align*}
	(\forall k \in \mathbb{N}) \quad s_{k+1} \leq (1- a_{k}) s_{k} +a_{k} b_{k} +\epsilon_{k},
	\end{align*}
	where $(a_{k})_{k \in \mathbb{N}}$, $(b_{k})_{k \in \mathbb{N}}$, and $(\epsilon_{k})_{k \in \mathbb{N}}$  are sequences in $\mathbb{R}$ satisfying the conditions:
	\begin{enumerate}
		\item  $(a_{k})_{k \in \mathbb{K}}$ is a sequence in $\left[0,1 \right]$ such that $\sum_{k \in \mathbb{N}} a_{k} = \infty $, or equivalently,  $ \prod_{k \in \mathbb{N}} (1- a_{k}) =0$;
		\item $\limsup_{k \to \infty} b_{k} \leq 0$;
		\item $(\forall k \in \mathbb{N}) \epsilon_{k} \in \mathbb{R}_{+}$ and $\sum_{k \in \mathbb{N}} \epsilon_{k} < \infty$.
	\end{enumerate}
Then $\lim_{k \to \infty} s_{k} =0$.
\end{fact}

Inspired by the proof of \cref{fact:sequence:sk}, we obtain the following \cref{prop:tkleq}, which is critical to some results in the next sections.  
It is not difficult to prove that \cref{prop:tkleq}\cref{prop:tkleq:conve0}  is actually equivalent to  \cref{fact:sequence:sk}.  We present \cref{prop:tkleq}\cref{prop:tkleq:conve0} because comparing with  \cref{fact:sequence:sk}, \cref{prop:tkleq}\cref{prop:tkleq:conve0} is more convenient to use.
The following lemma is necessary to prove \cref{prop:tkleq}.  
\begin{lemma} \label{lemma:alphabetaleq}
	Let $(\alpha_{k})_{k \in \mathbb{N}}$ be in $\mathbb{R}_{+}$ and let $(\beta_{k})_{k \in \mathbb{N}}$ be  in $\mathbb{R}$ such that $(\forall k \in \mathbb{N})$ $\alpha_{k} +\beta_{k} \leq 1$. Then 
		\begin{align}\label{eq:lemma:alphabetaleq}
		(\forall  m \in \mathbb{N}) (\forall k \in \mathbb{N}) \quad \sum^{m+k}_{i=m}  \prod^{m+k}_{j=i+1} \alpha_{j} \beta_{i}  \leq 1 -  \prod^{m+k}_{i=m} \alpha_{i}. 
		\end{align}
		Consequently,  $(\forall  m \in \mathbb{N}) (\forall k \in \mathbb{N}) $ $ \sum^{m+k}_{i=m}  \prod^{m+k}_{j=i+1} \alpha_{j} (1-\alpha_{i}) \leq 1 -  \prod^{m+k}_{i=m} \alpha_{i}. $ 
\end{lemma}
\begin{proof}
	 Let $m \in \mathbb{N}$. 
If $k=0$, then \cref{eq:lemma:alphabetaleq} turns to $\beta_{m} \leq 1- \alpha_{m}$, which is true by assumption.\footnotemark
	\footnotetext{As is the custom, in the whole work, we use the empty sum convention and empty product convention, that is, given a sequence $(t_{k})_{ k\in \mathbb{N}}$ in $\mathbb{R}$, for every $m$ and $n$ in $\mathbb{N}$ with $m > n$, we have   $\sum^{n}_{i=m} t_{i} =0$ and  $\prod^{n}_{i=m} t_{i} =1$. }  
	Suppose that \cref{eq:lemma:alphabetaleq} holds for some $k \in \mathbb{N}$.
	 Then apply the induction hypothesis in the  first inequality below to observe that 
	\begin{align*}
	\sum^{m+k+1}_{i=m}  \prod^{m+k+1}_{j=i+1} \alpha_{j} \beta_{i}   = \beta_{m+k+1} +   \alpha_{m+k+1} \sum^{m+k}_{i=m}  \prod^{m+k}_{j=i+1} \alpha_{j} \beta_{i} 
	 \leq  \beta_{m+k+1} +   \alpha_{m+k+1} \left( 1 -  \prod^{m+k}_{i=m} \alpha_{i}.  \right) 
  \leq 1 -  \prod^{m+k+1}_{i=m} \alpha_{i}.
	\end{align*}
	So, we proved \cref{eq:lemma:alphabetaleq} by induction.  The last assertion is clear with $(\forall k\in \mathbb{N})$ $\beta_{k} =1-\alpha_{k}$.
\end{proof}

\begin{proposition}  \label{prop:tkleq}
	Let $(t_{k})_{k \in \mathbb{N}}$ and $(\alpha_{k})_{k \in \mathbb{N}}$ be  sequences in $\mathbb{R}_{+}$, and let $(\beta_{k})_{k \in \mathbb{N}}$,  $(\gamma_{k})_{k \in \mathbb{N}}$, and $(\omega_{k})_{k \in \mathbb{N}}$   be sequences in $\mathbb{R}$ such that 
	\begin{align}  \label{eq:lemma:tkleq:tk}
	(\forall k \in \mathbb{N}) \quad t_{k+1} \leq \alpha_{k} t_{k} + \beta_{k} \omega_{k} +\gamma_{k}.
	\end{align}
	The following statements hold.
 \begin{enumerate}
 		\item  \label{prop:tkleq:supbounded} Suppose that $ \limsup_{k \to \infty} \alpha_{k} <1$ and  $M:=\sup_{k \in \mathbb{N}} \left( \beta_{k} \omega_{k} +\gamma_{k} \right) <\infty$. Then $(t_{k})_{k \in \mathbb{N}}$ is bounded. 
 	\item  \label{prop:tkleq:bounded} Suppose that $(\forall k \in \mathbb{N})$ $\alpha_{k} \in \left[0,1\right]$, $\alpha_{k} +\beta_{k} \leq 1$, and $\omega_{k} \in \mathbb{R}_{+}$,  that $ \hat{\omega} := \sup_{k \in \mathbb{N}} \omega_{k} < \infty  $, and that  $(\forall k \in \mathbb{N}) $ $ \alpha_{k} +\gamma_{k} \leq 1 $ or    $ \sum_{k \in \mathbb{N}}  \abs{\gamma_{k}} <\infty$.	
 	Then $(t_{k})_{k \in \mathbb{N}}$ is bounded. 
 	 
 	 	\item \label{prop:tkleq:conve0} Suppose that $(\forall k \in \mathbb{N})$ $\alpha_{k} \in \left[0,1\right]$ and $\beta_{k} \in \left[0,1\right]$ with $\alpha_{k} +\beta_{k} \leq 1$ and $\sum_{k \in \mathbb{N}} (1-\alpha_{k}) = \infty $, that $\limsup_{k \to \infty} \omega_{k} \leq 0$, and that
 	 $\sum_{k \in \mathbb{N}} | \gamma_{k} |<\infty$. Then $\lim_{k \to \infty} t_{k} =0$.
 	 
 	 	 \item \label{prop:tkleq:betakalphak} Suppose that $(\forall k \in \mathbb{N})$ $\alpha_{k} \in \left[0,1\right[$ and $ \beta_{k}    \in \mathbb{R}_{+}$ with $\sup_{k \in \mathbb{N}} \frac{\beta_{k}   }{1-\alpha_{k}}   <\infty$ and $\sum_{k \in \mathbb{N}} (1-\alpha_{k}) = \infty $, 
 	 	  that $\limsup_{k \to \infty} \omega_{k}\leq 0$, and that $\sum_{k \in \mathbb{N}} \abs{ \gamma_{k} }<\infty$. Then $\lim_{k \to \infty} t_{k} =0$.
 \end{enumerate}
\end{proposition}

\begin{proof}
Based on \cref{eq:lemma:tkleq:tk}, by induction, it is easy to get that  
\begin{align}  \label{eq:prop:tkleq:tk}
(\forall m \in \mathbb{N}) (\forall k \in \mathbb{N}) \quad t_{m+k} \leq  \prod^{m+k-1}_{j=m} \alpha_{j} t_{m} + \sum^{m+k-1}_{i=m}  \prod^{m+k-1}_{j=i+1} \alpha_{j} \beta_{i} \omega_{i}  +  \sum^{m+k-1}_{i=m}  \prod^{m+k-1}_{j=i+1} \alpha_{j} \gamma_{i}.
\end{align}

\cref{prop:tkleq:supbounded}: Because $ \limsup_{k \to \infty} \alpha_{k} <1$, there exists $\hat{\alpha} \in \mathbb{R}_{++}$ and $N \in \mathbb{N}$ such that $ \limsup_{k \to \infty} \alpha_{k}  <\hat{\alpha} <1$ and $(\forall k \geq N)$ $\alpha_{k} \leq \hat{\alpha}$. 
This and the assumption that $(\forall k \in \mathbb{N})$ $\alpha_{k} \in \mathbb{R}_{+}$ ensure that
\begin{align*}
(\forall k \in \mathbb{N}) \quad \sum^{N+k}_{i=N}  \prod^{N+k}_{j=i+1} \alpha_{j} \leq \sum^{N+k}_{i=N}  \prod^{N+k}_{j=i+1} \hat{\alpha} =\sum^{N+k}_{i=N} \hat{\alpha}^{N+k-i} =   \sum^{k}_{j=0} \hat{\alpha}^{j} \leq  (1-\hat{\alpha})^{-1},
\end{align*}
which, combining with \cref{eq:prop:tkleq:tk}, entails that 
\begin{align*}
(\forall k  \in \mathbb{N}) \quad  t_{N+k+1} &\leq \prod^{N+k}_{j=N} \alpha_{j} t_{N} + \sum^{N+k}_{i=N}  \prod^{N+k}_{j=i+1} \alpha_{j} \left( \beta_{i} \omega_{i}  + \gamma_{i} \right)\\
&\leq t_{N} + (1-\hat{\alpha})^{-1} \max \{ M, 0 \} < \infty.
\end{align*}

\cref{prop:tkleq:bounded}: 
In view of  \cref{eq:prop:tkleq:tk}, 
\begin{align}\label{prop:tkleq:bounded:tk}
(\forall k \in \mathbb{N}) \quad t_{k+1} &\leq  \prod^{k}_{j=0} \alpha_{j} t_{0} + \sum^{k}_{i=0}  \prod^{k}_{j=i+1} \alpha_{j} \beta_{i} \omega_{i}  +  \sum^{k}_{i=0}  \prod^{k}_{j=i+1}  \alpha_{j} \gamma_{i}.
\end{align}
  Because  $(\forall k \in \mathbb{N})$ $\alpha_{k} \in \left[0,1\right]$, we know that  $(\forall k \in \mathbb{N})$ $ \prod^{k}_{i=0} \alpha_{i} \in \left[0,1\right]$ and $1 -  \prod^{k}_{i=0} \alpha_{i} \in \left[0,1\right]$. 
Then combine \cref{lemma:alphabetaleq} with the assumption  to get that
\begin{align}\label{prop:tkleq:bounded:alphabetaomega}
(\forall k \in \mathbb{N}) \quad  \sum^{k}_{i=0}  \prod^{k}_{j=i+1} \alpha_{j} \beta_{i} \omega_{i}  \leq    \left( 1 -  \prod^{k}_{i=0} \alpha_{i} \right) \hat{\omega} \leq     \hat{\omega}.
\end{align}

If  $(\forall k \in \mathbb{N}) $ $ \alpha_{k} +\gamma_{k} \leq 1 $, then by  \cref{lemma:alphabetaleq},   $(\forall k \in \mathbb{N}) $ $ \sum^{k}_{i=0}  \prod^{k}_{j=i+1} \alpha_{j} \gamma_{i} \leq 1 -  \prod^{k}_{i=0} \alpha_{i} \leq 1 $, which, combining with \cref{prop:tkleq:bounded:tk} and \cref{prop:tkleq:bounded:alphabetaomega}, forces that $(\forall k \in \mathbb{N})$ $t_{k+1} \leq t_{0} + \hat{\omega} +1 <  \infty$.

On the other hand, if $ \sum_{i \in \mathbb{N}}  \abs{\gamma_{i}} <\infty$, then   $(\forall k \in \mathbb{N}) $ $ \sum^{k}_{i=0}  \prod^{k}_{j=i+1} \alpha_{j} \gamma_{i} \leq \sum^{k}_{i =0}   \abs{\gamma_{i}} \leq \sum_{i \in \mathbb{N}}  \abs{\gamma_{i}} <\infty$. 
Combine this with \cref{prop:tkleq:bounded:tk} and \cref{prop:tkleq:bounded:alphabetaomega} to get that $(\forall k \in \mathbb{N})$ $t_{k+1} \leq t_{0} + \hat{\omega} +  \sum_{i \in \mathbb{N}}   \abs{\gamma_{i} } <  \infty.$ 

Hence, in both cases,  $(t_{k})_{k \in \mathbb{N}}$ is bounded.

\cref{prop:tkleq:conve0}:  
Let $\epsilon \in \mathbb{R}_{++}$. Because $\limsup_{k \to \infty} \omega_{k} \leq 0$ and $\sum_{k \in \mathbb{N}} |\gamma_{k}| <\infty$, there exists $N \in \mathbb{N}$ such that 
\begin{align}\label{eq:prop:tkleq:conve0:epsolon}
(\forall k \geq N) \quad \omega_{k} \leq \epsilon \quad \text{and} \quad \sum^{\infty}_{i =k} |\gamma_{i } |< \epsilon.
\end{align}
Taking \cref{eq:prop:tkleq:tk} and \cref{lemma:alphabetaleq} into account, we  establish that 
\begin{align*}
(\forall k \in \mathbb{N} \setminus \{0\}) \quad t_{N+k} &\leq  \prod^{N+k-1}_{j=N} \alpha_{j} t_{N} + \sum^{N+k-1}_{i=N}  \prod^{N+k-1}_{j=i+1} \alpha_{j} \beta_{i} \omega_{i}  +  \sum^{N+k-1}_{i=N}  \prod^{N+k-1}_{j=i+1} \alpha_{j} \gamma_{i}\\
& \leq \prod^{N+k-1}_{j=N} \alpha_{j} t_{N} + \left( 1 -  \prod^{N+k-1}_{j=N} \alpha_{j} \right) \epsilon +  \sum^{N+k-1}_{i=N} |\gamma_{i}|\\
& \leq  \prod^{N+k-1}_{j=N} \alpha_{j} t_{N} +\epsilon +\epsilon,
\end{align*}
which implies that $\limsup_{k \to \infty} t_{k} \leq 2 \epsilon$, since $(\forall k \in \mathbb{N})$ $\alpha_{k} \in \left[0,1\right]$ and  $\sum_{i \in \mathbb{N}} (1-\alpha_{i}) = \infty $ imply that 
$ \prod_{k \in \mathbb{N}} \alpha_{k} =0$ and that $ \lim_{k \to \infty} \prod^{N+k-1}_{j=N} \alpha_{j}  =0$.
Because $\epsilon \in \mathbb{R}_{++}$ is chosen arbitrarily, and $(t_{k})_{k \in \mathbb{N}}$ is in $\mathbb{R}_{+}$,  we obtain that $\lim_{k \to \infty} t_{k} =0$. 

\cref{prop:tkleq:betakalphak}: Because $(\forall k \in \mathbb{N})$   $\frac{\beta_{k}   }{1-\alpha_{k}}  \in \mathbb{R}_{+}$ with $\sup_{k \in \mathbb{N}} \frac{\beta_{k}   }{1-\alpha_{k}}   <\infty$ and $\limsup_{k \to \infty} \omega_{k}\leq 0$, it is easy to prove that $\limsup_{k \to \infty} \frac{\beta_{k}   }{1-\alpha_{k}}  \omega_{k}\leq 0$. Moreover,  inasmuch as \cref{eq:lemma:tkleq:tk},
\begin{align*}
	(\forall k \in \mathbb{N}) \quad t_{k+1} \leq \alpha_{k} t_{k} + \beta_{k} \omega_{k} +\gamma_{k } \leq \alpha_{k} t_{k} + (1-\alpha_{k}) \frac{ \beta_{k}}{1-\alpha_{k}} \omega_{k} +\abs{\gamma_{k }}.
\end{align*}
So the required result follows easily from \cref{fact:sequence:sk}.
\end{proof}

\begin{fact} {\rm \cite[Lemma~2.2]{Suzuki2005}} \label{fact:sequence:ukvk}
	Let $(u_{k})_{k \in \mathbb{N}}$ and $(v_{k})_{k \in \mathbb{N}}$ be bounded sequences in $\mathcal{H}$  and let $(\alpha_{k})_{k \in \mathbb{K}}$ be a sequence in $\left[0,1 \right]$ with
	$0 < \liminf_{k \to \infty} \alpha_{k} \leq \limsup_{k \to \infty} \alpha_{k} <1$. Suppose that 
	\begin{align*}
	(\forall k \in \mathbb{N}) ~ u_{k+1} =\alpha_{k} v_{k} + (1-\alpha_{k}) u_{k} \quad \text{and} \quad  \limsup_{i \to \infty} \left( \norm{v_{i+1} -v_{i}} -\norm{u_{i+1} -u_{i}}  \right) \leq 0.
	\end{align*} 
	Then $\lim_{k \to \infty} \norm{v_{k} -u_{k}} =0$.
\end{fact}

The existence of the limit $\lim_{k \to \infty} a_{k}$ in the following \cref{fact:SequenceConverg} was directly used in proofs of \cite{EcksteinBertsekas1992}, \cite{MarinoXu2004}, \cite{Rockafellar1976},  \cite{Xu2002} and many other papers on the convergence of proximal point algorithms. 
For completeness, we present a detailed proof below. 
\begin{fact}\label{fact:SequenceConverg}
	Let $(a_{k})_{k \in \mathbb{N}}$ and $(b_{k})_{k \in \mathbb{N}}$ be sequences in $\mathbb{R}_{+}$ such that $\sum_{k \in \mathbb{N}} b_{k} <  \infty$ and 
	\begin{align}\label{eq:fact:SequenceConverg:akbk}
	(\forall k \in \mathbb{N}) \quad a_{k+1} \leq a_{k} +b_{k}.
	\end{align}
	Then   $\lim_{k \to \infty} a_{k} = \liminf_{k \to \infty} a_{k} \in \mathbb{R}_{+}$.  
\end{fact} 
\begin{proof}
	
	Let $\epsilon \in \mathbb{R}_{++}$. Because  $\sum_{k \in \mathbb{N}} b_{k} <\infty$ and $(b_{k})_{k \in \mathbb{N}}$ is in $\mathbb{R}_{+}$ , there exists $K_{1} \in \mathbb{N}$ such that 
	\begin{align}\label{eq:fact:SequenceConverg:sumbi}
	(\forall k \geq K_{1}) \quad \sum_{i=k}^{\infty} b_{i} < \frac{\epsilon}{3}.
	\end{align} 
Denote by $\bar{a}:=\liminf_{k \to \infty} a_{k} $.  By the definition of $\liminf$, there exists a subsequence $(a_{n_{k}})_{k \in \mathbb{N}}$ of $(a_{k})_{k \in \mathbb{N}}$ such that $a_{n_{k}} \to \bar{a} \in \mathbb{R}_{+}$.
	Then, there exists	$K_{2}  \in \mathbb{N}$ such that 
	\begin{align}\label{eq:fact:SequenceConverg:akj}
	(\forall k  \geq K_{2}) \quad \abs{a_{n_{k}} -\bar{a} } < \frac{\epsilon}{3}.
	\end{align}
	Employ the definition of $\liminf$ again to know that there exists $K_{3} \in \mathbb{N}$ such that
	\begin{align}\label{eq:fact:SequenceConverg:>alpha}
	(\forall k \geq K_{3}) \quad a_{k} > \bar{a}  -  \frac{\epsilon}{3}.
	\end{align} 
	Set $K:= \max \{ K_{1}, K_{2}, K_{3} \}$. Then 
	\begin{align*}
	(\forall i >  n_{K}) \quad \bar{a}  -  \frac{\epsilon}{3} \stackrel{\cref{eq:fact:SequenceConverg:>alpha}}{<} a_{i}  \stackrel{\cref{eq:fact:SequenceConverg:akbk}}{\leq} a_{i-1} +b_{i-1} \stackrel{\cref{eq:fact:SequenceConverg:akbk}}{\leq}  \cdots \stackrel{\cref{eq:fact:SequenceConverg:akbk}}{\leq} a_{n_{K}} +\sum_{j=n_{K}}^{i-1}b_{j} \stackrel{\cref{eq:fact:SequenceConverg:sumbi}\cref{eq:fact:SequenceConverg:akj}}{\leq} \bar{a}  +\frac{2\epsilon}{3},
	\end{align*}
	which implies the desired result. 
\end{proof}

\begin{fact} {\rm \cite[Lemma~2.5]{MarinoXu2004}} \label{fact:x+y} 
	Let $x$ and $y$ be in $\mathcal{H}$. Then $\norm{x+y}^{2} \leq \norm{x}^{2} + 2\innp{y,x+y} \leq \norm{x}^{2} + 2\norm{y}\norm{x+y}$ and  $\norm{x+y}^{2} \leq \norm{x}^{2} + \norm{y} ( 2\norm{x} +\norm{y})$.
\end{fact}

\subsection*{Maximally monotone operators}
\begin{definition} {\rm \cite[Definition~23.1]{BC2017}} \label{defn:ResolventApproxi}
	Let $G: \mathcal{H} \to 2^{\mathcal{H}}$ and let $\gamma \in \mathbb{R}_{++}$. The \emph{resolvent of $G$} is 
	\begin{align*}
	\J_{G} = (\Id + G)^{-1}
	\end{align*}
	and the \emph{Yosida approximation of $G$ of index $\gamma$} is
	\begin{align*}
  \prescript{\gamma}{}{G} = \frac{1}{\gamma} (\Id -  	\J_{\gamma G}).
	\end{align*}
\end{definition}
\begin{definition} \label{defn:FirmNonexpansive} {\rm \cite[Definition~4.1]{BC2017}}
	Let $D$ be a nonempty subset of $\mathcal{H}$ and let $T:D \rightarrow \mathcal{H}$. Then $T$ is	\begin{enumerate}
			\item \label{defn:FirmNonexpansive:Firm} \emph{firmly nonexpansive} if  $(\forall x  \in D)$ 	$(\forall y \in D)$ $\norm{Tx -Ty}^{2} +\norm{(\Id -T) x -(\Id -T) y}^{2} \leq \norm{x-y}^{2}$;
		\item \label{defn:FirmNonexpansive:Nonex} \emph{nonexpansive} if 
		 $(\forall x  \in D)$ 	$(\forall y \in D)$ $\norm{Tx -Ty} \leq \norm{x-y}$.
	\end{enumerate}
\end{definition}
Remember that  throughout this work,
\begin{empheq}[box = \mybluebox]{equation*}
\text{$A:\mathcal{H} \to 2^{\mathcal{H}}$ is maximally monotone.}
\end{empheq}

The following properties of the resolvent and Yosida approximation of  maximally monotone operators are fundamental to our analysis later and will be frequently used in the next sections. 
\begin{fact} \label{fact:resolvent}
\begin{enumerate}
		\item  \label{fact:resolvent:demiclosed} {\rm \cite[Proposition~20.38(ii)]{BC2017}} $\gra A$ is sequentially closed in $\mathcal{H}^{\text{weak}} \times \mathcal{H}^{\text{strong}}$, i.e., for every sequence $(x_{k}, u_{k})_{k \in \mathbb{N}}$ in $\gra A$ and every $(x,u) \in \mathcal{H}\times \mathcal{H}$, if $x_{k} \weakly x$ and $u_{k} \to u$, then $(x,u) \in \gra A$.	
	\item \label{fact:resolvent:ingraphA} {\rm \cite[Proposition~23.7(i)]{BC2017}} $(\forall \gamma \in \mathbb{R}_{++})$ $(\forall x \in \mathcal{H})$ $ \frac{1}{\gamma} (x -  	\J_{\gamma A}x)=\prescript{\gamma}{}{A} (x) \in A \left(\J_{\gamma A}x \right)$, that is, $\left( \J_{\gamma A}x ,  \prescript{\gamma}{}{A} (x) \right) \in \gra A$.
	\item \label{fact:resolvent:FN} {\rm \cite[Proposition~23.10]{BC2017}}  $\J_{A} $ is full domain, single-valued, and firmly nonexpansive. 
	\item  \label{fact:resolvent:zer} {\rm \cite[Proposition~23.38]{BC2017}} Let $\gamma \in \mathbb{R}_{++}$. Then $\zer A =\Fix \J_{\gamma A} = \zer  \prescript{\gamma}{}{A} $.
	\item  \label{fact:resolvent:closedconvex} {\rm \cite[Proposition~23.39]{BC2017}} $\zer A$ is closed and convex. 
\end{enumerate}	
\end{fact}

\begin{fact} {\rm \cite[Lemma~2.4]{MarinoXu2004}} \label{fact:ineqJIden}
  Let $\lambda$ and $\mu$ be in $\mathbb{R}_{++}$. Then 
	\begin{align*}
	(\forall x \in \mathcal{H}) \quad \J_{\lambda A}(x ) = \J_{\mu A} \left( \frac{\mu}{\lambda}x + \left( 1- \frac{\mu}{\lambda} \right)  \J_{\lambda A}x  \right). 
	\end{align*}
\end{fact}	
The following \cref{fact:ineqTIden}\cref{fact:ineqTIden:eq} is  used in the proof of \cite[Theorem~3.6]{MarinoXu2004}.
\begin{fact}\label{fact:ineqTIden}
  Let $\lambda$ and $\mu$ be in $\mathbb{R}_{++}$. Set $T_{\lambda}:=2  \J_{\lambda A} -\Id$ and $T_{\mu}:=2  \J_{\mu A} -\Id$. Then the following hold. 
	\begin{enumerate}
		\item \label{fact:ineqTIden:eq} $(\forall x \in \mathcal{H})$ $T_{\lambda}x =T_{\mu} \left(  \frac{\mu}{\lambda}x+\left( 1- \frac{\mu}{\lambda} \right)  \J_{\lambda A}x   \right) + \left( 1-  \frac{\mu}{\lambda} \right) (  \J_{\lambda A}(x)-x )$.
		\item\label{fact:ineqTIden:ineq} $(\forall x \in \mathcal{H})$ $\norm{T_{\lambda}(x) -T_{\mu}(x) } \leq \left|  1-  \frac{\mu}{\lambda} \right| \norm{T_{\lambda}(x) -x}$.
	\end{enumerate}
\end{fact}

\begin{proof}
	\cref{fact:ineqTIden:eq}: The desired result follows immediately  from \cref{fact:ineqJIden} and the definitions of $T_{\lambda}$ and $T_{\mu}$.
	
	\cref{fact:ineqTIden:ineq}: Notice that, via \cref{fact:resolvent}\cref{fact:resolvent:FN} and \cite[Proposition~4.4]{BC2017}, $T_{\lambda}$ and $T_{\mu}$ are nonexpansive. Hence, 
	the required inequality follows easily from \cref{fact:ineqTIden:eq}.
\end{proof}

\begin{fact} {\rm \cite[Lemma~3.3]{MarinoXu2004}} \label{fact:ineqJ}
  Let $c$ and $\bar{c}$ be in $\mathbb{R}_{++}$ with $\bar{c} \leq c$. Then $(\forall x \in \mathcal{H})$ $\norm{\J_{\bar{c} A}(x )-x} \leq 2 \norm{\J_{c A}(x )-x}$.
	In particular, for every sequences $(y_{k})_{k \in \mathbb{N}}$ in $\mathcal{H}$ and $(c_{k})_{k \in \mathbb{N}}$ in $\mathbb{R}_{++}$ such that $(\forall k \in \mathbb{N})$ $\bar{c} \leq c_{k}$, we have 
	\begin{align*}
	(\forall k \in \mathbb{N}) \quad \norm{\J_{\bar{c} A}(y_{k} )-y_{k}} \leq 2 \norm{\J_{c_{k} A}(y_{k} )-y_{k}}.
	\end{align*} 
\end{fact}

\subsection*{Sets of zeroes}

The technique of the following  proof  was used in  \cite[Theorem~1]{Rockafellar1976} 
to prove  the  uniqueness of the weak sequential  cluster point of the sequence of iterations generated by Rockafellar's proximal point algorithm. 
According to Rockafellar's remark, one similar uniqueness argument was used by B.~Martinet in 1970 and it was suggested to Martinet by H. Br\'ezis. 
 
\begin{proposition} \label{prop:less1weakcluster}
	Let $(y_{k})_{k \in \mathbb{N}}$ be a sequence in $\mathcal{H}$.  Set $\Omega $ as the set of all weak sequential cluster points of $ (y_{k} )_{k \in \mathbb{N}}$. Suppose that for every $z \in \Omega$, the limit $\lim_{k \to \infty} \norm{y_{k} -z}$ exists.
	Then there is at most one element in $\Omega$, that is, there cannot be more than one weak sequential cluster point of  $ (y_{k} )_{k \in \mathbb{N}}$.
\end{proposition}

\begin{proof}
	Suppose to the contrary that there exist $z_{1}$ and $z_{2}$ in $\Omega $ with $z_{1} \neq z_{2}$. Then based on the assumption, there exist $q_{1}$ and $q_{2}$ in $ \mathbb{R}_{+}$ such that $(\forall i \in \{1,2\})$ $\lim_{k \to \infty} \norm{y_{k} -z_{i}} =q_{i}$.
	Note that for every $k \in \mathbb{N}$,
	\begin{align*}
	&\norm{y_{k} -z_{2}}^{2} =\norm{y_{k} -z_{1}}^{2} + 2 \innp{ y_{k} -z_{1}, z_{1} -z_{2}} +\norm{z_{1} -z_{2}}^{2}  \text{and} \\
	& \norm{y_{k} -z_{1}}^{2} =\norm{y_{k} -z_{2}}^{2} + 2 \innp{ y_{k} -z_{2}, z_{2} -z_{1}} +\norm{z_{2} -z_{1}}^{2}, 
	\end{align*}
	which imply, respectively,  that 
	\begin{subequations}\label{eq:prop:less1weakcluster:q1q2}
		\begin{align}
		&2 \innp{ y_{k} -z_{1}, z_{1} -z_{2}} =     \norm{y_{k} -z_{2}}^{2}  -\norm{y_{k} -z_{1}}^{2} -\norm{z_{1} -z_{2}}^{2}   \to   q^{2}_{2} -q^{2}_{1} -\norm{z_{1} -z_{2}}^{2}, \text{ and}\\
		& 2	\innp{ y_{k} -z_{2}, z_{2} -z_{1}} =   \norm{y_{k} -z_{1}}^{2} - \norm{y_{k} -z_{2}}^{2} - \norm{z_{2} -z_{1}}^{2} \to  q^{2}_{1} -q^{2}_{2} -\norm{z_{1} -z_{2}}^{2}.
		\end{align}
	\end{subequations}
	On the other hand, $\{z_{1}, z_{2}\} \subseteq \Omega$ forces that once  $\lim_{k \to \infty}  \innp{ y_{k} -z_{1}, z_{1} -z_{2}}$ and $\lim_{k \to \infty} \innp{ y_{k} -z_{2}, z_{2} -z_{1}}$ exist, these two limits must be $0$. Combine this with \cref{eq:prop:less1weakcluster:q1q2} and $z_{1} \neq z_{2}$ to deduce that 
	\begin{align*}
	q^{2}_{2} -q^{2}_{1}= \norm{z_{1} -z_{2}}^{2} > 0  \quad \text{and}  \quad q^{2}_{1} -q^{2}_{2} =\norm{z_{1} -z_{2}}^{2} > 0,
	\end{align*}
	which is absurd. Therefore, the desired result holds. 
\end{proof}

The following result is inspired by the proof of \cite[Theorem~1]{Rockafellar1976}, which shows  the weak convergence of Rockafellar's proximal point algorithm.
\begin{proposition}\label{prop:AAtilde}
	   Let $r \in \mathbb{R}_{++}$. Define $\tilde{A} := A +\partial \iota_{B[0;r]}$. The  following assertions hold.
	\begin{enumerate}
		\item \label{prop:AAtilde:partialiota} $(\forall x \in \mathcal{H})$ $\partial \iota_{B[0;r]}x =
		\begin{cases}
		\{0\}, \quad &\text{if } \norm{x} <r;\\
		\mathbb{R}_{+}x,  \quad &\text{if } \norm{x} =r;\\
		\varnothing, \quad &\text{if } \norm{x} >r,
		\end{cases}$
		$~$ and $~$ 
		$\tilde{A} x  =
		\begin{cases}
		A(x), \quad &\text{if } \norm{x} <r;\\
		A(x)+ \mathbb{R}_{+}x,  \quad &\text{if } \norm{x} =r;\\
		\varnothing, \quad &\text{if } \norm{x} >r.
		\end{cases}$
		\item \label{prop:AAtilde:cap} Suppose that $\dom A \cap B(0;r) \neq \varnothing$. Then the following hold.
		\begin{enumerate}\label{prop:AAtilde:domAB}
			\item   \label{prop:AAtilde:cap:maximonot} $\tilde{A}$ is maximally monotone. Consequently, $(\forall \gamma \in \mathbb{R}_{++})$ $\J_{\gamma \tilde{A}}: \mathcal{H} \to \mathcal{H}$ is full-domain and firmly nonexpansive.  
			\item  \label{prop:AAtilde:cap:zer} $\zer \tilde{A} \neq \varnothing$.
			\item  \label{prop:AAtilde:cap:J} Let $  \gamma  \in  \mathbb{R}_{++}$ and let $x \in \mathcal{H}$. If $ x \in \J_{\gamma A}^{-1} \left( B(0;r) \right)$  or $  x \in \J_{\gamma \tilde{A}}^{-1} \left( B(0;r) \right)$, then $ \J_{\gamma \tilde{A}}  x = \J_{\gamma A} x$. Consequently, $ B(0;r) \cap \zer A =B(0;r) \cap \zer \tilde{A}  $.
			\item   \label{prop:AAtilde:cap:zerT} If $\zer \tilde{A}$ is not a singleton, then $ \zer A\neq \varnothing$.
		\end{enumerate}
	\end{enumerate}
\end{proposition}

\begin{proof}
	\cref{prop:AAtilde:partialiota}: The explicit formula of $\partial \iota_{B[0;r]}$ is a direct result from \cite[Examples~6.39 and 16.13]{BC2017}, which   immediately implies the formula of $\tilde{A}$.
	
	\cref{prop:AAtilde:cap:maximonot}: Clearly, because $ B[0;r] $ is a nonempty closed and convex set, we have that $ \iota_{B[0;r]}$  is a proper lower semicontinuous and convex function. Then the required results are guaranteed by \cite[Theorem~20.25, Corollary~25.5(ii), and Proposition~23.10(iii)]{BC2017}.
	
	\cref{prop:AAtilde:cap:zer}: According to \cref{prop:AAtilde:partialiota}, $\dom \tilde{A} \subseteq B[0;r]$ is bounded. Hence, the desired result is immediate from the maximal monotonicity of $\tilde{A}$ and \cite[Proposition~23.36(iii)]{BC2017}.
	
	\cref{prop:AAtilde:cap:J}: If $x \in  \J_{\gamma A}^{-1} \left( B(0;r) \right)$, i.e., $ \J_{\gamma A}(x) \in B(0;r) $, then, via \cref{defn:ResolventApproxi}, 
	\begin{align*}
	\J_{\gamma A}(x) = (\Id + \gamma A)^{-1} (x) & \Leftrightarrow x \in \J_{\gamma A}(x)  + \gamma A \left( \J_{\gamma A}(x)  \right)\\
	& \Leftrightarrow  x \in \J_{\gamma A}(x)  + \gamma \tilde{A} \left( \J_{\gamma A}(x)  \right) \quad (\text{by $\J_{\gamma A}(x) \in B(0;r)$ and \cref{prop:AAtilde:partialiota}})\\
	& \Leftrightarrow  \J_{\gamma A}(x)  \in (\Id + \gamma \tilde{A})^{-1} (x) =  \J_{\gamma \tilde{A}}(x) \\
	& \Leftrightarrow   \J_{\gamma A}(x) = \J_{\gamma \tilde{A}}(x). \quad (\text{ $\J_{\gamma \tilde{A}}$ is single-valued})
	\end{align*}
	On the other hand, switch $A$ and $\tilde{A}$ in the proof above to obtain that $x \in  \J_{\gamma \tilde{A}}^{-1} \left( B(0;r) \right)$ implies $  \J_{\gamma \tilde{A}}(x)  =   \J_{\gamma A}(x)$. Hence, the first required result is true. 
	
In addition,   for every $ y  \in B(0;r) \cap \zer A  $, by \cref{fact:resolvent}\cref{fact:resolvent:zer}, $	\J_{\gamma A}(y) =y \in B(0;r)$,
	which, combining with the result proved above,  entails that $y= \J_{\gamma A}(y) = \J_{\gamma \tilde{A}}(y) \in \Fix  \J_{\gamma \tilde{A}} =\zer \tilde{A}$. Hence, $B(0;r) \cap \zer A \subseteq B(0;r) \cap \zer \tilde{A} $.	
Moreover, applying the similar technique, we obtain that  $B(0;r) \cap \zer \tilde{A}   \subseteq B(0;r) \cap \zer A $.  Altogether, $B(0;r) \cap \zer A =  B(0;r) \cap \zer \tilde{A} $.

	\cref{prop:AAtilde:cap:zerT}: Suppose that $\{x,y\} \subseteq \zer \tilde{A} $ with $x \neq y$. If $\norm{x} <r$ or $\norm{y} <r$, then,  via \cref{prop:AAtilde:cap:J},  $\varnothing \neq B(0;r) \cap \zer \tilde{A} \subseteq \zer A$. 
	
	Suppose that $\norm{x}=r$ and $\norm{y} =r$. Notice that, due to 	\cref{prop:AAtilde:cap:maximonot} and  \cref{fact:resolvent}\cref{fact:resolvent:closedconvex}, $\zer \tilde{A} $ is closed and convex. Let $\alpha \in \left]0,1\right[\,$. Then based on \cite[Corollary~2.15]{BC2017},
	\begin{align*}
	\norm{\alpha x + (1-\alpha) y }^{2} =\alpha \norm{x}^{2} +(1-\alpha)\norm{y}^{2} - \alpha (1-\alpha) \norm{x -y}^{2} < r^{2}, 
	\end{align*}  
	which leads to $ \alpha x + (1-\alpha) y \in B(0;r)  \cap  \zer \tilde{A}  \subseteq \zer A$ by \cref{prop:AAtilde:cap:J}. 
	
	Altogether, the proof is complete.
\end{proof}

\begin{proposition} \label{prop:tildeAzerA}
	Let $(y_{k})_{k \in \mathbb{N}}$ be a sequence in $\mathcal{H}$ and let $(c_{k})_{k \in \mathbb{N}}$ be in $\mathbb{R}_{++}$.  Suppose that $(y_{k})_{k \in \mathbb{N}}$ and $\left( \J_{c_{k} A}y_{k}  \right)_{k \in \mathbb{N}}$ are bounded. Set $\Omega $ as the set of all weak sequential cluster points of $ (y_{k} )_{k \in \mathbb{N}}$.  Then there exists  $r \in \mathbb{R}_{++}$ such that  $\tilde{A} := A +\partial \iota_{B[0;r]}$  is a maximally monotone operator and that
	\begin{align*}
\zer \tilde{A} \neq   \varnothing,  \quad   \left( \Omega \cap \zer \tilde{A} \right) \subseteq \zer A, \quad  \text{and} \quad  (\forall  k \in \mathbb{N}) ~ \J_{c_{k} A}y_{k} = \J_{c_{k} \tilde{A}}y_{k}.
	\end{align*}
\end{proposition}

\begin{proof}
Because $(y_{k})_{k \in \mathbb{N}}$ and $\left( \J_{c_{k} A}y_{k}  \right)_{k \in \mathbb{N}}$ are bounded,	there exists $r \in \mathbb{R}_{++}$ such that 
	\begin{align}\label{eq:prop:boundedZerA:zerA:<r/2}
	(\forall k \in \mathbb{N}) \quad \norm{ y_{k} } \leq \frac{r}{2} \quad \text{and} \quad \norm{ \J_{c_{k} A}y_{k}   } \leq \frac{r}{2},
	\end{align}
	which, due to \cite[Lemmas~2.42 and 2.45]{BC2017}, implies that 	$  \varnothing   \neq \Omega \subseteq B[0; \frac{r}{2}] \subseteq B(0;r)$.
	
	Set   $\tilde{A} := A +\partial \iota_{B[0;r]}$.  In view of \cref{fact:resolvent}\cref{fact:resolvent:ingraphA}, $(\forall k \in \mathbb{N})$ $ \frac{1}{c_{k}} (y_{k} -  \J_{c_{k} A}y_{k} )=\prescript{c_{k}}{}{A} (y_{k})  \in A \left(\J_{c_{k}  A} y_{k}  \right)$, which, by \cref{eq:prop:boundedZerA:zerA:<r/2},  yields that $(\forall k \in \mathbb{N}) $ $ \J_{c_{k}  A} y_{k} \in B(0;r) \cap \dom A$.
	Combine this with 	\cref{prop:AAtilde}\cref{prop:AAtilde:domAB} to entail that $\tilde{A} $ is maximally monotone and that $ \zer \tilde{A}  \neq \varnothing$, 
	$	\Omega \cap \zer \tilde{A} \subseteq B(0;r) \cap  \zer \tilde{A}  \subseteq \zer A$, and $ (\forall k \in \mathbb{N}) $ $ \J_{c_{k}  A} y_{k} = \J_{c_{k}  \tilde{A}} y_{k}$.
\end{proof}

The result of \cref{prop:weakclusterCBAR} under the condition \cref{prop:weakclusterCBAR:>0} was  also proved in  proofs of \cite[Theorem~1]{Rockafellar1976} and \cite[Theorem~3]{EcksteinBertsekas1992} for   related proximal point algorithms  by applying   \cref{fact:resolvent}\cref{fact:resolvent:ingraphA} and employing the definition of maximal monotonicity.
In addition,  the idea of the proof of \cref{prop:weakclusterCBAR} under the hypothesis \cref{prop:weakclusterCBAR:i} with $t=1$ was adopted in the Step~2 of the proof of \cite[Theorem~5.1]{Xu2002}.

\begin{proposition} \label{prop:weakclusterCBAR}
	Let $(y_{k})_{k \in \mathbb{N}}$ be a sequence in $\mathcal{H}$ and let $(c_{k})_{k \in \mathbb{N}}$ be in $\mathbb{R}_{++}$. 	 Set $\Omega $ as the set of all weak sequential cluster points of $ (y_{k} )_{k \in \mathbb{N}}$. Suppose that one of the following holds.
	\begin{enumerate}
		\item  \label{prop:weakclusterCBAR:>0} $\bar{c}:=\inf_{k \in \mathbb{N}}c_{k} >0$ and   $y_{k } - \J_{c_{k} A}(y_{k} ) \to 0$. 
		\item  \label{prop:weakclusterCBAR:i} $c_{k} \to \infty$,   $(y_{k})_{k \in \mathbb{N}}$ is bounded,  and there exists $t \in \mathbb{N}$ such that $y_{k+t} - \J_{c_{k} A}(y_{k} ) \to 0$.  
	\end{enumerate}
	Then $\Omega \subseteq \zer A$.   
\end{proposition}

\begin{proof}
	If $\Omega  =\varnothing$, then the desired inclusion is trivial. Suppose that $\Omega \neq \varnothing$.  Take $\bar{y} \in \Omega$, that is, there exists a subsequence  $(y_{k_{i}})_{i \in \mathbb{N}}$  of $(y_{k})_{k \in \mathbb{N}}$ such that $y_{k_{i}} \weakly \bar{y} $. 
	
	Assume that	\cref{prop:weakclusterCBAR:>0} holds.  Then  \cref{fact:ineqJ} and the assumption that $y_{k} - \J_{c_{k} A}(y_{k} ) \to 0$ imply that 
	$\J_{\bar{c} A}(y_{k_{i}} )-y_{k_{i}} \to 0$. Therefore, by \cref{fact:resolvent}\cref{fact:resolvent:FN}$\&$\cref{fact:resolvent:zer} and \cite[Corollary~4.28]{BC2017}, we conclude that $\bar{y} \in \Fix \J_{\bar{c} A} =\zer A$.
	
	Assume that \cref{prop:weakclusterCBAR:i} holds. Clearly, the boundedness of $(y_{k})_{k \in \mathbb{N}}$ and the convergence $y_{k+t} - \J_{c_{k} A}(y_{k} ) \to 0$ imply that $(\J_{c_{k} A}(y_{k} ) )_{k \in \mathbb{N}}$ is bounded and that 
	\begin{align}\label{eq:prop:weakclusterCBAR:i:J} 
	\J_{c_{k_{i} -t} A}(y_{k_{i} -t} ) =y_{k_{i}}  - \left( y_{k_{i}}  - \J_{c_{k_{i} -t} A}(y_{k_{i} -t} ) \right) \weakly \bar{y}.
	\end{align}
	Moreover, as a consequence of  \cref{fact:resolvent}\cref{fact:resolvent:ingraphA},
	\begin{align}\label{eq:prop:weakclusterCBAR:i:graA}
	(\forall i \in \mathbb{N}) \quad  \left( \J_{c_{k_{i} -t} A}(y_{k_{i} -t} ) , \frac{1}{ c_{k_{i} -t}} \left( y_{k_{i} -t}  - \J_{c_{k_{i} -t} A}(y_{k_{i} -t} )  \right) \right) \in \gra A.
	\end{align}
	Because $c_{k} \to \infty$ and the boundedness of $(y_{k})_{k \in \mathbb{N}}$ and $(\J_{c_{k} A}(y_{k} ) )_{k \in \mathbb{N}}$ yield $\frac{1}{ c_{k_{i} -t}} \left( y_{k_{i} -t}  - \J_{c_{k_{i} -t} A}(y_{k_{i} -t} )  \right) \to 0$, combine \cref{eq:prop:weakclusterCBAR:i:J}, \cref{eq:prop:weakclusterCBAR:i:graA}, and \cref{fact:resolvent}\cref{fact:resolvent:demiclosed}  to establish that $(\bar{y}, 0) \in \gra A$, i.e., $\bar{y} \in \zer A$.
	
	Altogether, 
	the required result is correct, since $\bar{y} \in \Omega$ is arbitrary. 
\end{proof}

\subsection*{Asymptotic regularity and convergence}
Given a sequence $(y_{k})_{k \in \mathbb{N}}$   in $\mathcal{H}$ and a sequence $(c_{k})_{k \in \mathbb{N}}$  in $\mathbb{R}_{++}$, we say the \emph{asymptotic regularity holds for $(y_{k})_{k \in \mathbb{N}}$  and $(c_{k})_{k \in \mathbb{N}}$}, if $ y_{k} -  \J_{c_{k} A}y_{k} \to 0$. 
We shall see that the asymptotic regularity plays an important role in the proof of the convergence of generalized proximal point algorithms. 

\begin{proposition} \label{prop:ykJckykto0}
	Suppose that $\zer A \neq \varnothing$. Let $p \in \zer A$ and let $(y_{k})_{k \in \mathbb{N}}$ be a  sequence in $\mathcal{H}$.  Suppose that $\lim_{k \to \infty} \norm{ y_{k} -p } $ exists in $ \mathbb{R}_{+}$ and that $y_{k+t} - \J_{c_{k} A}y_{k} \to 0$ for some $t \in \mathbb{N}$. Then $y_{k} -  \J_{c_{k} A}y_{k} \to 0$ and $y_{k}  -y_{k+t} \to 0 $.
\end{proposition}

\begin{proof}
Taking  \cref{fact:resolvent}\cref{fact:resolvent:FN}$\&$\cref{fact:resolvent:zer} into account and employing \cref{defn:FirmNonexpansive}\cref{defn:FirmNonexpansive:Firm}, we observe that 
	\begin{align*}
	(\forall k \in \mathbb{N}) \quad \norm{ \J_{c_{k} A}y_{k} -  p}^{2} + \norm{y_{k} -\J_{c_{k} A}y_{k} }^{2} \leq \norm{y_{k} -p}^{2},
	\end{align*} 
	which yields that for every $k \in \mathbb{N}$, 
	\begin{align*}
	\norm{y_{k} -\J_{c_{k} A}y_{k} }^{2} -\norm{y_{k} -p}^{2} + \norm{y_{k+t}-p}^{2} &\leq  - \norm{ \J_{c_{k} A}y_{k} -  p}^{2} +\norm{y_{k+t}-p}^{2} \\ 
	&= \innp{ y_{k+t}-p -\left(\J_{c_{k} A}y_{k} -  p \right), y_{k+t}-p  +  \left(\J_{c_{k} A}y_{k} -  p  \right)}\\
	&\leq \norm{ y_{k+t} -\J_{c_{k} A}y_{k} } \left(\norm{y_{k+t}-p} +\norm{y_{k}-p} \right),
	\end{align*}
	where in the last inequality we use  the Cauchy-Schwarz inequality, the nonexpansiveness of $\J_{c_{k} A}$, and $p \in \zer A = \Fix \J_{c_{k} A}$.  	Hence, 
	\begin{align*} 
	(\forall k \in \mathbb{N}) \quad 	\norm{y_{k} -\J_{c_{k} A}y_{k} }^{2}  \leq \norm{y_{k} -p}^{2} - \norm{y_{k+t}-p}^{2} +\norm{ y_{k+t} -\J_{c_{k} A}y_{k} } \left(\norm{y_{k+t}-p} +\norm{y_{k}-p}    \right),
	\end{align*}
	which ensures $y_{k} -  \J_{c_{k} A}y_{k} \to 0$, since the existence of $\lim_{k \to \infty} \norm{ y_{k} -p } $  yields  $\norm{y_{k} -p}^{2} - \norm{y_{k+t}-p}^{2}  \to 0$ and the boundedness of $\left(\norm{y_{k+t}-p} +\norm{y_{k}-p}    \right)_{k \in \mathbb{N}}$. 
	
	Moreover, in consideration of $	(\forall k \in \mathbb{N})  $ $ \norm{y_{k}  -y_{k+t} } \leq \norm{ y_{k} -  \J_{c_{k} A}y_{k}} +\norm{  \J_{c_{k} A}y_{k} -y_{k+t} }$, we  reach the last required convergence by using $y_{k} -  \J_{c_{k} A}y_{k} \to 0$ and  $y_{k+t} - \J_{c_{k} A}y_{k} \to 0$.
\end{proof}

The following result will play an essential role to prove the weak convergence of the generalized proximal point algorithm.
\begin{fact} { \rm \cite[Lemma~2.47]{BC2017}} \label{fact:weakconveRockEcks}
		Let $(y_{k})_{k \in \mathbb{N}}$ be a sequence in $\mathcal{H}$ and let $C$ be a nonempty subset of $\mathcal{H}$. 
	Suppose that every weak sequential cluster point of $(y_{k})_{k \in \mathbb{N}}$ belongs to $C$, that is,  $  \Omega \lr{ (y_{k})_{k \in \mathbb{N}} }\subseteq C$, and that   $\lr{ \forall z \in C}$ $\lim_{k \to \infty} \norm{y_{k} -z} $ exists in $\mathbb{R}_{+}$. Then $(y_{k})_{k \in \mathbb{N}}$  converges weakly to a point in $C$. 
\end{fact} 

 
The following \cref{prop:weakcluster:Pu} is inspired by the Step~2 of the proof of \cite[Theorem~5.1]{Xu2002}. The following result is critical to prove the strong convergence of generalized proximal point algorithms. 
\begin{proposition} \label{prop:weakcluster:Pu}
	 Let $(y_{k})_{k \in \mathbb{N}}$ be a bounded sequence in $\mathcal{H}$ and let $u \in \mathcal{H}$.
	Set $\Omega $ as the set of all weak sequential cluster points of $ (y_{k} )_{k \in \mathbb{N}}$. Suppose that $\Omega \subseteq \zer A$.   Then 
	\begin{align*}
	\limsup_{k \to \infty} \innp{u-  \Pro_{\zer A} u, y_{k} - \Pro_{\zer A} u} \leq 0.
	\end{align*}
	\end{proposition}

\begin{proof}
By the definition of $\limsup$, there exists a subsequence $(y_{k_{i}})_{i \in \mathbb{N}}$ of $(y_{k})_{k \in \mathbb{N}}$ 
such that 
\begin{align}\label{eq:prop:weakcluster:Pu:limsup}
\limsup_{k \to \infty} \innp{u-  \Pro_{\zer A} u, y_{k} - \Pro_{\zer A} u} = \lim_{i \to \infty} \innp{u-  \Pro_{\zer A} u, y_{k_{i}} - \Pro_{\zer A} u}.
\end{align}
 Because $(y_{k})_{k \in \mathbb{N}}$ is  bounded, without loss of generality (otherwise take a subsequence of $(y_{k_{i}})_{i \in \mathbb{N}}$),  we assume that $ y_{k_{i}} \weakly \bar{y}$ for some $\bar{y} \in   \Omega \subseteq \zer A$. Hence, due to \cite[Proposition~3.16]{BC2017} and \cref{fact:resolvent}\cref{fact:resolvent:closedconvex},
 \begin{align}\label{eq:prop:weakcluster:Pu:leq}
  \lim_{i \to \infty} \innp{u-  \Pro_{\zer A} u, y_{k_{i}} - \Pro_{\zer A} u} = \innp{u-  \Pro_{\zer A} u, \bar{y} - \Pro_{\zer A} u} \leq 0.
 \end{align}
 Combine \cref{eq:prop:weakcluster:Pu:limsup} and \cref{eq:prop:weakcluster:Pu:leq} to obtain the required inequality. 
\end{proof}


\section{Generalized proximal point algorithms} \label{sec:GeneralizedPPAs}

Recall that 
\begin{empheq}[box = \mybluebox]{equation*}
A: \mathcal{H} \to 2^{\mathcal{H}} \text{ is a maximally monotone operator}.
\end{empheq}
In the rest of this work, $u \in \mathcal{H}$ 
 and $x_{0} \in \mathcal{H}$ are arbitrary but fixed, and  the generalized proximal point algorithm is generated by  conforming the following recursion:
\begin{align} \label{eq:xk}
(\forall k \in \mathbb{N}) \quad x_{k+1} = \alpha_{k} u +\beta_{k} x_{k} + \gamma_{k} \J_{c_{k} A} (x_{k}) +\delta_{k} e_{k},
\end{align}
where  $(\forall k \in \mathbb{N})$ $e_{k} \in \mathcal{H}$, 
$c_{k} \in \mathbb{R}_{++}$, 
and $\{ \alpha_{k},  \beta_{k}, \gamma_{k}, \delta_{k}  \} \subseteq \mathbb{R}$. From now on, 
\begin{empheq}[box = \mybluebox]{equation*}
\Omega  \text{ is the set of all weak sequential cluster points of } (x_{k} )_{k \in \mathbb{N}}.
\end{empheq}

In this section, we investigate the boundedness and asymptotic regularity of $(x_{k} )_{k \in \mathbb{N}}$; after that, we demonstrate the equivalence of the boundedness of $(x_{k} )_{k \in \mathbb{N}}$ and $\zer A \neq \varnothing$.
\subsection*{Boundedness}

\begin{lemma} \label{leamma:xkTk}
	Set $(\forall k \in \mathbb{N})$ $T_{k}:= 2\J_{c_{k } A} -\Id  $. Then $(\forall k \in \mathbb{N})$ $T_{k}$ is nonexpansive and $\Fix T_{k} =\zer A$. Moreover, 
	\begin{align*} 
	(\forall k \in \mathbb{N})  \quad x_{k+1} =\left( \beta_{k} +\frac{\gamma_{k}}{2} \right)  x_{k} + \frac{\gamma_{k}}{2} T_{k}(x_{k}) +\alpha_{k} u+\delta_{k} e_{k}.
	\end{align*}
\end{lemma}

\begin{proof}
	Based on \cref{fact:resolvent}\cref{fact:resolvent:FN}$\&$\cref{fact:resolvent:zer} and \cite[Proposition~4.4]{BC2017},   $(\forall k \in \mathbb{N})$ $T_{k}$ is nonexpansive and $\Fix T_{k} =\Fix \J_{c_{k } A}=\zer A$. In consideration of \cref{eq:xk},
	\begin{subequations} \label{eq:prop:weakclustersXY:xy}
		\begin{align*}
		(\forall k \in \mathbb{N})  \quad x_{k+1} & = \alpha_{k} u +\beta_{k} x_{k} + \gamma_{k} \J_{c_{k} A} (x_{k}) +\delta_{k} e_{k}\\
		&=\alpha_{k} u +\beta_{k} x_{k} +  \frac{1}{2}\gamma_{k} ( x_{k} + T_{k}x_{k}) +\delta_{k} e_{k}\\
		&=\left( \beta_{k} +\frac{\gamma_{k}}{2} \right)  x_{k} + \frac{\gamma_{k}}{2} T_{k}(x_{k}) +\alpha_{k} u+\delta_{k} e_{k},
		\end{align*}
		which implies directly the desired equality. 
	\end{subequations}
\end{proof}

The following inequalities will be used frequently later. 
\begin{lemma} \label{lemma:xk+1-p}
Let $p \in \zer A$. Set $(\forall k \in \mathbb{N})$ $T_{k}:= 2\J_{c_{k } A} -\Id  $.  Then the following hold.
\begin{enumerate}
	\item \label{lemma:xk+1-p:norm} $(\forall k \in \mathbb{N})$ $	\norm{x_{k+1} -p } \leq \left( \abs{   \beta_{k} +\frac{\gamma_{k}}{2} } + \abs{\frac{\gamma_{k}}{2}  } \right) \norm{ x_{k} -p } + \norm{\alpha_{k} u+\delta_{k}e_{k} - (1-\beta_{k} -\gamma_{k}) p}$.
	\item \label{lemma:xk+1-p:square}  
	Denote by $(\forall k \in \mathbb{N})$  $\xi_{k} := \left( \abs{   \beta_{k} +\frac{\gamma_{k}}{2} } + \abs{\frac{\gamma_{k}}{2}  } \right)^{2}$,  $ \phi_{k}:= 1 -\beta_{k}- \gamma_{k}$,  $\varphi_{k}:= 1- \alpha_{k} -\beta_{k}- \gamma_{k}$, $F(k):= \norm{\delta_{k} e_{k} -\varphi_{k}u} $, and 
	$G(k):= F(k) +2\norm{ \left( \beta_{k} +\frac{\gamma_{k}}{2} \right)  (x_{k} -p )+ \frac{\gamma_{k}}{2} (T_{k}(x_{k}) -p)+\phi_{k} \left( u-p \right) } $. 	
	 Then
 $(\forall k \in \mathbb{N})$ $\norm{x_{k+1} -p }^{2}  
	\leq  \xi_{k}  \norm{ x_{k} -p }^{2}   
	+ 2 \phi_{k}  \innp{u-p, x_{k+1}-p - \delta_{k} e_{k}  + \varphi_{k} u}  + F(k)G(k)$.
 \item \label{lemma:xk+1-p:xk-Jckto}  Set $(\forall k \in \mathbb{N})$ $M(k) :=\norm{\alpha_{k} u + \delta_{k} e_{k} - (1-\beta_{k} -\gamma_{k} )p}$. Suppose $\inf_{k \in \mathbb{N}} \gamma_{k}( \beta_{k} +\gamma_{k}) \geq 0$. Then $(\forall k \in \mathbb{N})$ $ \norm{x_{k+1} -p}^{2} \leq  (  \beta_{k} +\gamma_{k} )^{2} \norm{  x_{k} -p}^{2} -\gamma_{k}	(2\beta_{k}+\gamma_{k}) \norm{x_{k} -  \J_{c_{k} A} x_{k} }^{2}+2 M (k)\norm{x_{k+1} -p}$.
 \item  \label{lemma:xk+1-p:alphak} $(\forall k \in \mathbb{N})$ $\norm{x_{k+1} -p }^{2}  
 \leq  \left( \abs{   \beta_{k} +\frac{\gamma_{k}}{2} } + \abs{\frac{\gamma_{k}}{2}  } \right)^{2} \norm{ x_{k} -p }^{2}   +2 \innp{ \alpha_{k} u+ \delta_{k} e_{k} - \left( 1-  \beta_{k} -\gamma_{k}   \right) p, x_{k+1} -p}$.
\end{enumerate}
\end{lemma}

\begin{proof}
 Let $k \in \mathbb{N}$.   

\cref{lemma:xk+1-p:norm}: In view of \cref{leamma:xkTk},  
	\begin{align*}
 	\norm{x_{k+1} -p }  
	= & \norm{ \left( \beta_{k} +\frac{\gamma_{k}}{2} \right)   x_{k}  + \frac{\gamma_{k}}{2}  T_{k}(x_{k}) +\alpha_{k} u+\delta_{k} e_{k} -p}\\
	= 	& \norm{ \left( \beta_{k} +\frac{\gamma_{k}}{2} \right)  (x_{k} -p )+ \frac{\gamma_{k}}{2} (T_{k}(x_{k}) -p)+\alpha_{k} u +  \delta_{k} e_{k}  - \left( 1  -  \beta_{k} -\gamma_{k}   \right) p}\\
	\leq & \abs{ \beta_{k} +\frac{\gamma_{k}}{2} } \norm{x_{k} -p} + \abs{ \frac{\gamma_{k}}{2}} \norm{T_{k}(x_{k}) -p} + \norm{\alpha_{k} u +  \delta_{k} e_{k}  - \left( 1  -  \beta_{k} -\gamma_{k}   \right) p}\\
	\leq &   \left( \abs{   \beta_{k} +\frac{\gamma_{k}}{2} } + \abs{\frac{\gamma_{k}}{2}  } \right) \norm{x_{k} -p} + \norm{\alpha_{k} u +  \delta_{k} e_{k}  - \left( 1  -  \beta_{k} -\gamma_{k}   \right) p},
	\end{align*}	
where in the  last inequality we used the nonexpansiveness of $T_{k}$  and the fact that $p \in \zer A =\Fix T_{k}$.

\cref{lemma:xk+1-p:square}: Applying   \cref{leamma:xkTk}  in the first equality and the last inequality, and employing \cref{fact:x+y} in the first and second inequalities, we obtain that
	\begin{align*}
	& \norm{x_{k+1} -p }^{2}  \\
	=&  \norm{ \left( \beta_{k} +\frac{\gamma_{k}}{2} \right)  (x_{k} -p )+ \frac{\gamma_{k}}{2} (T_{k}(x_{k}) -p)+(1- \beta_{k} -\gamma_{k}) \left( u-p \right) +  \delta_{k} e_{k}  - \left( 1-\alpha_{k} -  \beta_{k} -\gamma_{k}   \right) u}^{2}\\
	\leq &\norm{ \left( \beta_{k} +\frac{\gamma_{k}}{2} \right)  (x_{k} -p )+ \frac{\gamma_{k}}{2} (T_{k}(x_{k}) -p)+\phi_{k} \left( u-p \right) }^{2} + F(k)G(k)\\
		\leq & \norm{ \left( \beta_{k} +\frac{\gamma_{k}}{2} \right)  (x_{k} -p )+ \frac{\gamma_{k}}{2} (T_{k}(x_{k}) -p)}^{2}  + 2 \phi_{k}  \innp{u-p, x_{k+1}-p - \delta_{k} e_{k}  +\varphi_{k} u} + F(k)G(k)  \\
	\leq 	& \left( \abs{   \beta_{k} +\frac{\gamma_{k}}{2} } + \abs{\frac{\gamma_{k}}{2}  } \right)^{2} \norm{ x_{k} -p }^{2} + 2 \phi_{k}  \innp{u-p, x_{k+1}-p - \delta_{k} e_{k}  +\varphi_{k} u}   + F(k)G(k).
	\end{align*}

	\cref{lemma:xk+1-p:xk-Jckto}:	  According to \cref{fact:resolvent}\cref{fact:resolvent:FN}$\&$\cref{fact:resolvent:zer} and \cite[Proposition~4.4]{BC2017},  
	\begin{subequations} \label{eq:prop:xk-Jckto0:-innp}
		\begin{align}
		\innp{ x_{k} -p, \J_{c_{k} A} (x_{k})  -x_{k}} &= -\innp{ x_{k} -p,( \Id -\J_{c_{k} A}) (x_{k})  -( \Id -\J_{c_{k} A}) p } \\
		&\leq -\norm{( \Id -\J_{c_{k} A}) (x_{k})  -( \Id -\J_{c_{k} A}) p}^{2}\\
		&=-\norm{x_{k} -  \J_{c_{k} A} x_{k} }^{2}.
		\end{align}	
	\end{subequations}
Utilizing  $\inf_{k \in \mathbb{N}} \gamma_{k}( \beta_{k} +\gamma_{k}) \geq 0$ in the last inequality, we establish that
	\begin{align*}
	&	 \norm{x_{k+1} -p}^{2} \\
	\stackrel{\cref{eq:xk}}{=}& \norm{ (  \beta_{k} +\gamma_{k} )  x_{k} 
		+\gamma_{k} \left( \J_{c_{k} A} (x_{k})  -x_{k} \right) +\alpha_{k} u + \delta_{k} e_{k} - p }^{2} \\
	=\,&  \norm{ (  \beta_{k} +\gamma_{k} ) ( x_{k} -p)
		+\gamma_{k} \left( \J_{c_{k} A} (x_{k})  -x_{k} \right) +\alpha_{k} u + \delta_{k} e_{k} - (1-\beta_{k} -\gamma_{k} )p }^{2} \\
	\leq\,&  \norm{ (  \beta_{k} +\gamma_{k} ) ( x_{k} -p)	+\gamma_{k} \left( \J_{c_{k} A} (x_{k})  -x_{k} \right)}^{2} +2 M (k)\norm{x_{k+1} -p} \quad (\text{by \cref{fact:x+y}})\\
	=\,&(  \beta_{k} +\gamma_{k} )^{2} \norm{  x_{k} -p}^{2} +\gamma_{k}^{2} \norm{\J_{c_{k} A} (x_{k})  -x_{k}}^{2} +2\gamma_{k}( \beta_{k} +\gamma_{k}) \innp{ x_{k} -p, \J_{c_{k} A} (x_{k})  -x_{k}}+2 M(k)\norm{x_{k+1} -p}\\
	\stackrel{\cref{eq:prop:xk-Jckto0:-innp}}{\leq}& (  \beta_{k} +\gamma_{k} )^{2} \norm{  x_{k} -p}^{2} -\gamma_{k}	(2\beta_{k}+\gamma_{k}) \norm{x_{k} -  \J_{c_{k} A} x_{k} }^{2}+2 M (k)\norm{x_{k+1} -p}.
	\end{align*}
	
\cref{lemma:xk+1-p:alphak}:	Apply \cref{leamma:xkTk} and \cref{fact:x+y} in the following  first equality and first inequality, respectively, and employ the nonexpansiveness of  $T_{k}$ and the fact that $p= T_{k}(p)$ in the second inequality   to observe that 
	\begin{align*}
	&\norm{x_{k+1} -p }^{2}\\
	=& \norm{ \left( \beta_{k} +\frac{\gamma_{k}}{2} \right)   x_{k}  + \frac{\gamma_{k}}{2}  T_{k}(x_{k}) +\alpha_{k} u+\delta_{k} e_{k} -p}^{2}\\
	=& \norm{ \left( \beta_{k} +\frac{\gamma_{k}}{2} \right)  (x_{k} -p )+ \frac{\gamma_{k}}{2} (T_{k}(x_{k}) -p)+\alpha_{k}  u+ \delta_{k} e_{k} - \left( 1 -  \beta_{k} -\gamma_{k}   \right) p}^{2}\\
	\leq&\norm{ \left( \beta_{k} +\frac{\gamma_{k}}{2} \right)  (x_{k} -p )+ \frac{\gamma_{k}}{2} (T_{k}(x_{k}) -p) }^{2} + 2 \innp{   \alpha_{k} u + \delta_{k} e_{k} - \left( 1 -  \beta_{k} -\gamma_{k}   \right) p, x_{k+1} -p}\\
	\leq &  \left( \abs{   \beta_{k} +\frac{\gamma_{k}}{2} } + \abs{\frac{\gamma_{k}}{2}  } \right)^{2} \norm{ x_{k} -p }^{2}   +2 \innp{   \alpha_{k} u + \delta_{k} e_{k} - \left( 1 -  \beta_{k} -\gamma_{k}   \right) p, x_{k+1} -p}.
	\end{align*}
	
	Altogether, the proof is complete.
\end{proof}

We present sufficient conditions for the boundedness of $(x_{k})_{k \in \mathbb{N}}$ in the remaining subsection. 

	Note that if $(\forall k \in \mathbb{N})$  $\{ \alpha_{k}, \beta_{k} , \gamma_{k} \} \subseteq \left[0,1 \right]$ with  $\alpha_{k}+\beta_{k} +\gamma_{k}  = 1$ (which is the case in many publications on generalized proximal point algorithms), then based on \cref{prop:xkbounded}\cref{prop:xkbounded:leq1} or \cref{prop:xkbounded}\cref{prop:xkbounded:alpha}, we deduce the classical statement: $\zer A \neq \varnothing$ and $\sum_{k \in \mathbb{N}}  \norm{\delta_{k} e_{k}}< \infty$ imply the boundedness of $(x_{k})_{k \in \mathbb{N}}$.
\begin{proposition} \label{prop:xkbounded}
Suppose that $\zer A \neq \varnothing$ and that one of the following holds.
	\begin{enumerate}		
		\item \label{prop:xkbounded:suprhobounded}   $ \limsup_{k \to \infty} \left(  \abs{   \beta_{k} +\frac{\gamma_{k}}{2} } + \abs{\frac{\gamma_{k}}{2}  } \right)<1$,    $ \sup_{k\in \mathbb{N}} \abs{\alpha_{k}} < \infty$, and $ \sup_{k\in \mathbb{N}} \norm{\delta_{k}e_{k}} < \infty$. 
		\item  \label{prop:xkbounded:leq1}    $(\forall k \in \mathbb{N})$  $  \abs{   \beta_{k} +\frac{\gamma_{k}}{2} } + \abs{\frac{\gamma_{k}}{2}  } \leq 1$, and  the following hold:
			\begin{enumerate}
			\item \label{prop:xkbounded:leq1:alpha}   $(\forall k \in \mathbb{N})$    $|\alpha_{k}| +  \abs{   \beta_{k} +\frac{\gamma_{k}}{2} } + \abs{\frac{\gamma_{k}}{2}  } \leq 1$ or $\sum_{i \in \mathbb{N}}  \abs{\alpha_{i}} < \infty$;
			\item \label{prop:xkbounded:leq1:delta}   $\left[  (\forall k \in \mathbb{N}   \abs{   \beta_{k} +\frac{\gamma_{k}}{2} } + \abs{\frac{\gamma_{k}}{2}  } +|\delta_{k}| \leq 1 \text{ and } \sup_{i \in \mathbb{N}} \norm{e_{i}} < \infty \right]$ or $\sum_{i \in \mathbb{N}}  \norm{\delta_{i} e_{i}}< \infty$;
			\item \label{prop:xkbounded:leq1:p}     $ (\forall k \in \mathbb{N})  \abs{   \beta_{k} +\frac{\gamma_{k}}{2} } + \abs{\frac{\gamma_{k}}{2}  }+\abs{1 - \beta_{k}-\gamma_{k}}  \leq 1 $  or $\sum_{i \in \mathbb{N}}  \abs{1 - \beta_{i}-\gamma_{i}}< \infty$.
		\end{enumerate}	
	\item \label{prop:xkbounded:alpha} $(\forall k \in \mathbb{N})$ $\abs{\alpha_{k}} + \abs{   \beta_{k} +\frac{\gamma_{k}}{2} } + \abs{\frac{\gamma_{k}}{2}} \leq 1$,   $\sum_{k \in \mathbb{N}} \abs{1-\alpha_{k}-\beta_{k} -\gamma_{k}}< \infty$, and $\sum_{k \in \mathbb{N}} \norm{\delta_{k} e_{k}} < \infty$.
	\item \label{prop:xkbounded:delta} $(\forall k \in \mathbb{N})$ $ \abs{   \beta_{k} +\frac{\gamma_{k}}{2} } + \abs{\frac{\gamma_{k}}{2}  } +\abs{\delta_{k} } \leq 1$,  $\sup_{k \in \mathbb{N}} \norm{e_{k} } < \infty$, $\sum_{k \in \mathbb{N}} \abs{1-\beta_{k} -\gamma_{k} -\delta_{k}}< \infty$, and $\sum_{k \in \mathbb{N}} \abs{\alpha_{k} }< \infty$.
	
	\end{enumerate}

Then $(x_{k})_{k \in \mathbb{N}}$ is bounded. 
\end{proposition}

\begin{proof}
	Let $p \in \zer A$. 	In view of \cref{lemma:xk+1-p}\cref{lemma:xk+1-p:norm},  
\begin{align}\label{eq:cor:Rxkbounded}
(\forall k \in \mathbb{N}) \quad 	\norm{x_{k+1} -p }  \leq  \left( \abs{   \beta_{k} +\frac{\gamma_{k}}{2} } + \abs{\frac{\gamma_{k}}{2}  } \right) \norm{ x_{k} -p } + \norm{\alpha_{k} u+\delta_{k}e_{k} - (1-\beta_{k} -\gamma_{k}) p}.
\end{align}

\cref{prop:xkbounded:suprhobounded}: Note that $(\forall k \in \mathbb{N})$ $\abs{1-\beta_{k} -\gamma_{k}} 
\leq 1+ \abs{   \beta_{k} +\frac{\gamma_{k}}{2} } + \abs{\frac{\gamma_{k}}{2}  }$ and that $ \limsup_{i  \to \infty} \left(  \abs{   \beta_{i} +\frac{\gamma_{i}}{2} } + \abs{\frac{\gamma_{i}}{2}  } \right)<1$ implies the boundedness of $\left(\abs{   \beta_{i} +\frac{\gamma_{i}}{2} } + \abs{\frac{\gamma_{i}}{2}  }  \right)_{i \in \mathbb{N}}$ and $ (\abs{1-\beta_{i} -\gamma_{i} }  )_{i \in \mathbb{N}}$.  
The desired result is clear from \cref{eq:cor:Rxkbounded} and \cref{prop:tkleq}\cref{prop:tkleq:supbounded} with $(\forall k \in \mathbb{N})$ $t_{k} =\norm{ x_{k} -p }$, $\alpha_{k} =  \abs{   \beta_{k} +\frac{\gamma_{k}}{2} } + \abs{\frac{\gamma_{k}}{2}  } $, $\beta_{k}\equiv 0$, $\omega_{k} \equiv 0$, and $\gamma_{k} = \norm{\alpha_{k} u+\delta_{k}e_{k} - (1-\beta_{k} -\gamma_{k}) p}$.

\cref{prop:xkbounded:leq1}: Clearly, there are eight cases to prove and it suffices to show the boundedness of $\left(\norm{ x_{k} -p }\right)_{k \in \mathbb{N}}$ in each case. We prove only the following three cases and omit the similar proof of the remaining ones.

\emph{Case~1}: Suppose that   $\sum_{i \in \mathbb{N}}  \abs{\alpha_{i}} < \infty$, $\sum_{i \in \mathbb{N}} \norm{ \delta_{i} e_{i}} < \infty$, and $\sum_{i \in \mathbb{N}} \abs{1 - \beta_{i}-\gamma_{i}}< \infty$.  

Recall \cref{eq:cor:Rxkbounded} and apply  \cref{prop:tkleq}\cref{prop:tkleq:bounded} with $(\forall k \in \mathbb{N})$ $t_{k} =\norm{ x_{k} -p }$, $\alpha_{k} =  \abs{   \beta_{k} +\frac{\gamma_{k}}{2} } + \abs{\frac{\gamma_{k}}{2}  } $, $\beta_{k}\equiv 0$, $\omega_{k} \equiv 0$, and $\gamma_{k} = \norm{\alpha_{k} u+\delta_{k}e_{k} - (1-\beta_{k} -\gamma_{k}) p}$ to obtain the required boundedness of $\left(\norm{ x_{k} -p }\right)_{k \in \mathbb{N}}$.

\emph{Case~2}: Suppose that $(\forall k \in \mathbb{N})$    $|\alpha_{k}| +  \abs{   \beta_{k} +\frac{\gamma_{k}}{2} } + \abs{\frac{\gamma_{k}}{2}  }  \leq 1$, $ \abs{   \beta_{k} +\frac{\gamma_{k}}{2} } + \abs{\frac{\gamma_{k}}{2}  }  +|\delta_{k}|   \leq 1$, $\sup_{i \in \mathbb{N}} \norm{e_{i}} < \infty$, and  $ \abs{   \beta_{k} +\frac{\gamma_{k}}{2} } + \abs{\frac{\gamma_{k}}{2}  }  +\abs{1-  \beta_{k} -\gamma_{k}} \leq 1$.

Denote by $(\forall k \in \mathbb{N})$ $\xi_{k}:= \max \{ |\alpha_{k}|, |\delta_{k}|, \abs{1-  \beta_{k} -\gamma_{k}} \}$. In view of the assumption above, 
\begin{align}\label{eq:theorem:Convergence:leq1Bounded:xi}
(\forall k \in \mathbb{N}) \quad  \abs{   \beta_{k} +\frac{\gamma_{k}}{2} } + \abs{\frac{\gamma_{k}}{2}  }  + \xi_{k} \leq 1.
\end{align}
On the other hand, clearly \cref{eq:cor:Rxkbounded} forces
\begin{align*}
(\forall k \in \mathbb{N}) \quad \norm{ x_{k+1} -p} &\leq   \left(  \abs{   \beta_{k} +\frac{\gamma_{k}}{2} } + \abs{\frac{\gamma_{k}}{2}  }  \right) \norm{  x_{k} -p} + \abs{\alpha_{k}} \norm{u } +\norm{\delta_{k}  e_{k}}  + \abs{1  -\beta_{k} -\gamma_{k}}\norm{  p}\\
&\leq    \left(  \abs{   \beta_{k} +\frac{\gamma_{k}}{2} } + \abs{\frac{\gamma_{k}}{2}  }  \right) \norm{  x_{k} -p} + \xi_{k} \left( \norm{u } + \norm{p} + \sup_{i \in \mathbb{N}} \norm{e_{i}}  \right),
\end{align*}
which, connecting with \cref{eq:theorem:Convergence:leq1Bounded:xi} and applying \cref{prop:tkleq}\cref{prop:tkleq:bounded} with $(\forall k \in \mathbb{N})$ $t_{k} =\norm{ x_{k} -p }$, $\alpha_{k} =  \abs{   \beta_{k} +\frac{\gamma_{k}}{2} } + \abs{\frac{\gamma_{k}}{2}  } $,  $\beta_{k} =\xi_{k}$, $\omega_{k}=    \norm{u } + \norm{p} + \sup_{i \in \mathbb{N}} \norm{e_{i}}  $, and $\gamma_{k} \equiv 0$, guarantees the boundedness of $( \norm{  x_{k} -p})_{k \in \mathbb{N}}$.  

\emph{Case~3}: Suppose that  $(\forall k \in \mathbb{N})$  $ \abs{\alpha_{k}} +\abs{   \beta_{k} +\frac{\gamma_{k}}{2} } + \abs{\frac{\gamma_{k}}{2}  }  \leq 1$, $\sum_{i \in \mathbb{N}} \norm{ \delta_{i} e_{i}} < \infty$,  and $  \abs{   \beta_{k} +\frac{\gamma_{k}}{2} } + \abs{\frac{\gamma_{k}}{2}  }  +\abs{1 - \beta_{k}-\gamma_{k}} \leq 1$.

Denote by $(\forall k \in \mathbb{N})$ $\eta_{k}:= \max \{ |\alpha_{k}|, \abs{1  -\beta_{k} -\gamma_{k}} \}$. Similarly with the proof of Case~2, we observe that $(\forall k \in \mathbb{N})$ $ \abs{   \beta_{k} +\frac{\gamma_{k}}{2} } + \abs{\frac{\gamma_{k}}{2}  } + \eta_{k} \leq 1$, and that
\begin{align*}
(\forall k \in \mathbb{N}) \quad \norm{ x_{k+1} -p} & \leq   \left(  \abs{   \beta_{k} +\frac{\gamma_{k}}{2} } + \abs{\frac{\gamma_{k}}{2}  }  \right) \norm{  x_{k} -p} + \abs{\alpha_{k}} \norm{u } +\norm{\delta_{k}  e_{k}}  + \abs{1  -\beta_{k} -\gamma_{k}}\norm{  p}\\
&\leq    \left(  \abs{   \beta_{k} +\frac{\gamma_{k}}{2} } + \abs{\frac{\gamma_{k}}{2}  }  \right) \norm{  x_{k} -p} + \eta_{k} \left( \norm{u } + \norm{p}   \right) +\norm{\delta_{k}  e_{k}},
\end{align*}
which, applying \cref{prop:tkleq}\cref{prop:tkleq:bounded} with $(\forall k \in \mathbb{N})$ $t_{k} =\norm{ x_{k} -p }$, $\alpha_{k} =  \abs{   \beta_{k} +\frac{\gamma_{k}}{2} } + \abs{\frac{\gamma_{k}}{2}  } $,  $\beta_{k} =\eta_{k}$, $\omega_{k}=    \norm{u} + \norm{p}   $, and $\gamma_{k} = \norm{\delta_{k}  e_{k}}$, ensures the boundedness of $( \norm{  x_{k} -p})_{k \in \mathbb{N}}$.  

\cref{prop:xkbounded:alpha}$\&$\cref{prop:xkbounded:delta}: As a consequence of \cref{eq:cor:Rxkbounded}, for every $k \in \mathbb{N}$,
\begin{subequations}
\begin{align} 
\norm{x_{k+1} -p }  &\leq  \left( \abs{   \beta_{k} +\frac{\gamma_{k}}{2} } + \abs{\frac{\gamma_{k}}{2}  } \right) \norm{ x_{k} -p } + \abs{\alpha_{k}} \norm{ u -p} + \abs{1-\alpha_{k}-\beta_{k} -\gamma_{k}} \norm{p} +\norm{\delta_{k}e_{k}}; \label{eq:prop:xkbounded:alpha}\\
\norm{x_{k+1} -p }  &\leq  \left( \abs{   \beta_{k} +\frac{\gamma_{k}}{2} } + \abs{\frac{\gamma_{k}}{2}  } \right) \norm{ x_{k} -p } + \abs{\delta_{k}} \norm{e_{k} -p} + \abs{1 -\beta_{k} -\gamma_{k} -\delta_{k}} \norm{p} + \abs{\alpha_{k}} \norm{u}. \label{eq:prop:xkbounded:delta}
\end{align}
\end{subequations}
Hence, we obtain \cref{prop:xkbounded:alpha} (resp.\,\cref{prop:xkbounded:delta}) by invoking \cref{eq:prop:xkbounded:alpha} (resp.\,\cref{eq:prop:xkbounded:delta}) and applying  \cref{prop:tkleq}\cref{prop:tkleq:bounded} with $(\forall k \in \mathbb{N})$ $t_{k} =\norm{ x_{k} -p }$, $\alpha_{k} =  \abs{   \beta_{k} +\frac{\gamma_{k}}{2} } + \abs{\frac{\gamma_{k}}{2}  } $,  $\beta_{k} =\alpha_{k}$ (resp.\,$\beta_{k} =\delta_{k}$), $\omega_{k}=    \norm{u-p}   $ (resp.\,$\omega_{k}= \norm{e_{k} -p}$),  and $\gamma_{k} =  \abs{1-\alpha_{k}-\beta_{k} -\gamma_{k}} \norm{p} +\norm{\delta_{k}e_{k}}$ (resp.\,$\gamma_{k} = \abs{1 -\beta_{k} -\gamma_{k} -\delta_{k}} \norm{p} + \abs{\alpha_{k}} \norm{u}$). 

Altogether, the proof is complete.
\end{proof}

The following result is motivated by the Step~1 in the proof of \cite[Theorem~1]{BoikanyoMorosanu2010}.
 \begin{proposition} \label{prop:xkboundedalphak}
 	Suppose that $\zer A \neq \varnothing$, that  $(\forall k \in \mathbb{N})$ $\alpha_{k} \in \left]0,1\right]$ and $\alpha_{k}+ \abs{ \beta_{k} +\frac{ \gamma_{k}}{2} }+ \abs{\frac{ \gamma_{k}}{2}  }  \leq 1$,  
 	and that   $\frac{\delta_{k}e_{k} }{\alpha_{k}} \to 0$ and $\frac{1-\alpha_{k} -\beta_{k} -\gamma_{k} }{\alpha_{k}} \to 0$.
 	Then  $(x_{k})_{k \in \mathbb{N}}$ is bounded.
 \end{proposition}
  
  \begin{proof}
  	Let $p \in \zer A$.		Set $(\forall k \in \mathbb{N})$ $T_{k}:= 2\J_{c_{k } A} -\Id  $.  Because $\frac{\delta_{k}e_{k} }{\alpha_{k}} \to 0$ and $\frac{1-\alpha_{k} -\beta_{k} -\gamma_{k} }{\alpha_{k}} \to 0$, there exists $M \in \mathbb{R}_{++}$ such that 
  	\begin{align*}
(\forall k \in \mathbb{N}) \quad  	\norm{x_{0} -p} \leq M \text{ and } \norm{u-p}+\norm{ \frac{ \delta_{k} e_{k}}{\alpha_{k}}} +\abs{\frac{1-\alpha_{k} -\beta_{k} -\gamma_{k} }{\alpha_{k}} } \norm{p} \leq M.
  	\end{align*}
  	We prove 
  	\begin{align}  \label{eq:prop:alphakek:induction}
  	(\forall k \in \mathbb{N}) \quad \norm{x_{k}-p} \leq 2M.
  	\end{align} 
  	by induction below. 
  	
  	The basic case of \cref{eq:prop:alphakek:induction} follows immediately from the definition of $M$. Suppose that \cref{eq:prop:alphakek:induction} holds for some $k \in \mathbb{N}$.  Employ \cref{lemma:xk+1-p}\cref{lemma:xk+1-p:alphak} in the first inequality and use the assumption  $(\forall k \in \mathbb{N})$ $\alpha_{k}+ \abs{ \beta_{k} +\frac{ \gamma_{k}}{2} }+ \abs{\frac{ \gamma_{k}}{2}  }  \leq 1$  in the second inequality below to observe that 
  	\begin{align*}
  	&\norm{x_{k+1} -p }^{2}\\
  	\leq\,&\left( \abs{   \beta_{k} +\frac{\gamma_{k}}{2} } + \abs{\frac{\gamma_{k}}{2}  } \right)^{2} \norm{ x_{k} -p }^{2}   +2 \innp{ \alpha_{k} u+ \delta_{k} e_{k} - \left( 1-  \beta_{k} -\gamma_{k}   \right) p, x_{k+1} -p}\\
  	=\,& \left( \abs{   \beta_{k} +\frac{\gamma_{k}}{2} } + \abs{\frac{\gamma_{k}}{2}  } \right)^{2} \norm{ x_{k} -p }^{2}  + 2 \alpha_{k} \innp{   u-p+\frac{ \delta_{k} e_{k}}{\alpha_{k}}   - \frac{ 1-\alpha_{k} -  \beta_{k} -\gamma_{k} }{\alpha_{k}}  p, x_{k+1} -p}\\
  	\leq\,& \left( 1-\alpha_{k} \right)^{2} \norm{ x_{k} -p }^{2}  + 2\alpha_{k} M \norm{x_{k+1}-p},
  	\end{align*}
  	which, utilizing the induction hypothesis in the  inequality below, entails that 
  	\begin{align*}
  	\norm{x_{k+1} -p }^{2} \leq 4\left( 1-\alpha_{k} \right)^{2} M^{2} + 2\alpha_{k} M \norm{x_{k+1}-p}. 
  	\end{align*}
  	This guarantees that 
  	\begin{align*}
  	\left( 	\norm{x_{k+1} -p }-\alpha_{k} M   \right)^{2} \leq 4 \left( 1-\alpha_{k} \right)^{2} M^{2}  +( \alpha_{k}M)^{2}. 
  	\end{align*}
  	Moreover, the inequality above ensures that 
  	\begin{align*}
  	\norm{x_{k+1} -p }  \leq \alpha_{k} M + \left(  4\left( 1-\alpha_{k} \right)^{2} M^{2}  +( \alpha_{k}M)^{2} \right)^{\frac{1}{2}} 
   \leq M \left(\alpha_{k} +\left(2 \left( 1-\alpha_{k}\right) +\alpha_{k}\right)  \right) = 2M.  
  	\end{align*}
  	Therefore, \cref{eq:prop:alphakek:induction} holds, which ensures the desired boundedness of $(x_{k})_{k \in \mathbb{N}}$.
  \end{proof}

\subsection*{Asymptotic regularity}
 
In this section, we shall provide sufficient conditions for $ x_{k} -  \J_{c_{k} A}x_{k} \to 0$ or $  \Omega \subseteq   \zer A$.
 
 \begin{proposition} \label{prop:Jckboundedweakcluster}
 	Suppose that $(x_{k})_{k \in \mathbb{N}}$ is bounded.  
 	Then the following assertions hold.
 	\begin{enumerate}
 		\item \label{prop:Jckboundedweakcluster:Jckbounded} Suppose that one of the following holds.
 		\begin{enumerate}
 			\item \label{prop:Jckboundedweakcluster:Jckbounded:a} $\zer A \neq \varnothing$.
 			\item \label{prop:Jckboundedweakcluster:Jckbounded:b}  $\liminf_{k \to \infty} \abs{\gamma_{k}} >0$, $ \sup_{k\in \mathbb{N}} \abs{\alpha_{k}} < \infty$, $ \sup_{k\in \mathbb{N}} \abs{\beta_{k}} < \infty$,  and $ \sup_{k\in \mathbb{N}} \norm{\delta_{k}e_{k}} < \infty$.
 		\end{enumerate} 
 	 Then $(\J_{c_{k} A}x_{k} )_{k \in \mathbb{N}}$ is bounded. 
 		\item  \label{prop:Jckboundedweakcluster:xk+1xk}  
 		Suppose  that  $\alpha_{k} \to 0$, $\beta_{k} \to 0$, $ \gamma_{k} \to 1$,  and $\delta_{k} e_{k} \to 0$.  Then $x_{k+1} - \J_{c_{k} A} x_{k} \to 0$.
 		\item  \label{prop:Jckboundedweakcluster:contained}  
 		Suppose  that $ c_{k}  \to \infty$,  $\alpha_{k} \to 0$, $\beta_{k} \to 0$, $ \gamma_{k} \to 1$,  and $\delta_{k} e_{k} \to 0$.  Then $x_{k+1} - \J_{c_{k} A} x_{k} \to 0$ and $\varnothing \neq \Omega \subseteq   \zer A$.
 	\end{enumerate}
 	
 \end{proposition}
 
 \begin{proof}
 	\cref{prop:Jckboundedweakcluster:Jckbounded}:  	If $\zer A \neq \varnothing$, then via \cref{fact:resolvent}\cref{fact:resolvent:FN}$\&$\cref{fact:resolvent:zer}, for every $p \in \zer A$,
 	\begin{align*}
 	(\forall k \in \mathbb{N}) \quad \norm{\J_{c_{k} A}x_{k}} -\norm{p} \leq \norm{\J_{c_{k} A}x_{k} -p}= \norm{\J_{c_{k} A}x_{k} -\J_{c_{k} A}p} \leq \norm{x_{k} -p}.
 	\end{align*} 
 	Hence, the boundedness of $(x_{k})_{k \in \mathbb{N}}$ implies the boundedness of $(\J_{c_{k} A}x_{k} )_{k \in \mathbb{N}}$.

 Assume \cref{prop:Jckboundedweakcluster:Jckbounded:b}  holds. Then $ \hat{\alpha}:= \sup_{k\in \mathbb{N}} \abs{\alpha_{k}} < \infty$, $\hat{\beta}:= \sup_{k\in \mathbb{N}} \abs{\beta_{k}} < \infty$,  $L:= \sup_{k\in \mathbb{N}} \norm{\delta_{k}e_{k}} < \infty$, and $M:=\sup_{k \in \mathbb{N}}\norm{x_{k}} <\infty$. Take $\bar{\gamma} \in \mathbb{R}_{++}$ such that $0 < \bar{\gamma} <\liminf_{k \to \infty} \abs{\gamma_{k}} $. 
Then there exists $N \in \mathbb{N}$ such that 
 	\begin{align*} 
 	(\forall k \geq N) \quad \norm{\J_{c_{k} A} x_{k}} \stackrel{\cref{eq:xk} }{=} \norm{\frac{1 }{\gamma_{k}} \left(
 		x_{k+1} -\alpha_{k} u -\beta_{k} x_{k} -\delta_{k} e_{k} \right)} \leq \frac{1 }{\bar{\gamma}} \left( M+\hat{\alpha}\norm{u} +\hat{\beta}M+L\right),
 	\end{align*}
 	which shows the boundedness of $(\J_{c_{k} A}x_{k} )_{k \in \mathbb{N}}$. 
 	
 	\cref{prop:Jckboundedweakcluster:xk+1xk}:  Due to \cref{prop:Jckboundedweakcluster:Jckbounded}, $(\J_{c_{k} A}x_{k} )_{k \in \mathbb{N}}$ is bounded. Then apply  \cref{eq:xk}  to deduce that 
 	\begin{align*}
 	(\forall k \in \mathbb{N}) \quad \norm{x_{k+1} - \J_{c_{k} A} x_{k}} \leq \norm{\alpha_{k} u +\beta_{k} x_{k} +\delta_{k} e_{k}} +\abs{\gamma_{k} -1}\norm{ \J_{c_{k} A} (x_{k}) },
 	\end{align*}
 	which, by $\norm{\alpha_{k} u +\beta_{k} x_{k} +\delta_{k} e_{k}} +\abs{\gamma_{k} -1}\norm{ \J_{c_{k} A} (x_{k}) } \to 0$, necessitates that $x_{k+1} - \J_{c_{k} A} x_{k} \to 0$.
 	
 	\cref{prop:Jckboundedweakcluster:contained}:
 	In view of \cite[Lemma~2.45]{BC2017},  the boundedness of $(x_{k})_{k \in \mathbb{N}}$  ensures $\Omega  \neq \varnothing$.
 	Hence, the required result is clear from \cref{prop:Jckboundedweakcluster:xk+1xk} and \cref{prop:weakclusterCBAR}\cref{prop:weakclusterCBAR:i}.
 \end{proof}

 The idea of the following proof is motivated by the proof of \cite[Theorem~1]{Rockafellar1976}.
\begin{proposition} \label{prop:Rockafellar}
Suppose that $ (x_{k} )_{k \in \mathbb{N}}$ is bounded and that  $(\forall k \in \mathbb{N})$ 
$ \abs{   \beta_{k} +\frac{\gamma_{k}}{2} } + \abs{\frac{\gamma_{k}}{2}  }   \leq 1$, $\sum_{k \in \mathbb{N}} \abs{\alpha_{k} } <\infty$,  $\sum_{k \in \mathbb{N}} \abs{1-\beta_{k} -\gamma_{k} } <\infty$, $\sum_{k \in \mathbb{N}} \norm{\delta_{k}e_{k} } <\infty$, and $ \gamma_{k} \to 1$.  Then the following statements hold. 
	\begin{enumerate}
		\item \label{prop:Rockafellar:xkJckAxk}  $x_{k} - \J_{c_{k} A}x_{k}  \to 0$. 
 
		 	\item \label{prop:Rockafellar:zerANEQ}  If  $ \inf_{k \in \mathbb{N}}c_{k} >0$ or $c_{k} \to \infty$, then $\varnothing \neq \Omega \subseteq \zer A$. 
	\end{enumerate}	   
\end{proposition} 
 
\begin{proof}	
\cref{prop:Rockafellar:xkJckAxk}: 	Note that $\sum_{k \in \mathbb{N}} \abs{1-\beta_{k} -\gamma_{k} } <\infty$ and $ \gamma_{k} \to 1$ entail $\beta_{k} \to 0$.
In view of  \cref{prop:Jckboundedweakcluster}\cref{prop:Jckboundedweakcluster:Jckbounded}, our assumptions guarantee that  $(\J_{c_{k} A}x_{k} )_{k \in \mathbb{N}}$ is bounded. Then
by \cref{prop:tildeAzerA}, there exists a maximally monotone operator $\tilde{A}: \mathcal{H} \to 2^{\mathcal{H}}$ such that 
	\begin{align}  \label{eq:prop:Rockafellar:Atilde}
	\zer \tilde{A} \neq   \varnothing,  \quad   \left( \Omega \cap \zer \tilde{A} \right) \subseteq \zer A, \quad  \text{and} \quad  (\forall  k \in \mathbb{N}) ~ \J_{c_{k} A}x_{k} = \J_{c_{k} \tilde{A}}x_{k}.
	\end{align}
	Let $p \in \zer \tilde{A}$. Because $\tilde{A}$ is maximally monotone, via \cref{lemma:xk+1-p}\cref{lemma:xk+1-p:norm} and \cref{eq:prop:Rockafellar:Atilde},
	\begin{align*}
	(\forall k \in \mathbb{N}) \quad 	\norm{x_{k+1} -p } \leq \left( \abs{   \beta_{k} +\frac{\gamma_{k}}{2} } + \abs{\frac{\gamma_{k}}{2}  } \right) \norm{ x_{k} -p } + \norm{\alpha_{k} u+\delta_{k}e_{k} - (1-\beta_{k} -\gamma_{k}) p},
	\end{align*}  
	which, combining with the assumption and \cref{fact:SequenceConverg}, guarantees that $\lim_{k \to \infty} \norm{ x_{k} -p } $ exists in $\mathbb{R}_{+}$. 
	
	Using the maximal monotonicity of $\tilde{A}$ again and noticing  that  $(\forall k \in \mathbb{N})$ $c_{k} \in \mathbb{R}_{++}$, via \cref{fact:resolvent}\cref{fact:resolvent:FN}$\&$\cref{fact:resolvent:zer} and \cref{defn:FirmNonexpansive}\cref{defn:FirmNonexpansive:Firm}, we observe that $(\forall k \in \mathbb{N})$ $ \norm{\J_{c_{k} \tilde{A}} x_{k} -p}^{2} + \norm{  (\Id- \J_{c_{k} \tilde{A}}  ) x_{k}}^{2} \leq \norm{x_{k} -p}^{2}$,
	which implies that for every $k \in \mathbb{N}$,
	\begin{align*}
  \norm{ x_{k} -  \J_{c_{k} \tilde{A}} x_{k} }^{2} -\norm{x_{k} -p}^{2} +\norm{x_{k+1} -p}^{2} & \leq - \norm{\J_{c_{k} \tilde{A}} x_{k} -p}^{2}  + \norm{x_{k+1} -p}^{2} \\
		&=\innp{ x_{k+1} -p +p - \J_{c_{k} \tilde{A}} x_{k},  x_{k+1} -p + \J_{c_{k} \tilde{A}} x_{k} -p}\\
		&\leq \norm{x_{k+1}   -  \J_{c_{k} \tilde{A}} x_{k}}\left( \norm{x_{k+1} -p } + \norm{ \J_{c_{k} \tilde{A}}  x_{k} -p } \right).
	\end{align*}
	Hence, 
	\begin{align*}
		(\forall k \in \mathbb{N}) \quad  \norm{ x_{k} -  \J_{c_{k} \tilde{A}} x_{k}}^{2} \leq \norm{x_{k} -p}^{2} -\norm{x_{k+1} -p}^{2} + \norm{x_{k+1}   -  \J_{c_{k} \tilde{A}}  x_{k}}\left( \norm{x_{k+1} -p } + \norm{ \J_{c_{k} \tilde{A}}  x_{k} -p } \right).
	\end{align*}
	This together with \cref{prop:Jckboundedweakcluster}\cref{prop:Jckboundedweakcluster:xk+1xk}, the existence of     $\lim_{k \to \infty} \norm{ x_{k} -p } $, and the boundedness of $ (x_{k} )_{k \in \mathbb{N}}$ and $ (\J_{c_{k} \tilde{A}} x_{k} )_{k \in \mathbb{N}}$ leads to  $x_{k} - \J_{c_{k} \tilde{A}}x_{k}  \to 0$, which, due to \cref{eq:prop:Rockafellar:Atilde}, forces $x_{k} - \J_{c_{k} A}x_{k}  \to 0$. 
	
	\cref{prop:Rockafellar:zerANEQ}: 
Note that the boundedness of $ (x_{k} )_{k \in \mathbb{N}}$  forces $ \Omega \lr{ (x_{k} )_{k \in \mathbb{N}} } \neq  \varnothing $. Furthermore, 
based on \cref{prop:Rockafellar:xkJckAxk}, the required inclusion follows immediately from the assumption  and \cref{prop:weakclusterCBAR}\cref{prop:weakclusterCBAR:>0}$\&$\cref{prop:weakclusterCBAR:i}.  
\end{proof}

The following result is inspired by the proof of \cite[Theorem~4]{FHWang2011} which improves the strong convergence of the regularization method for the proximal point algorithm in  \cite[Theorem~3.3]{Xu2006}.

\begin{proposition} \label{prop:weakclustersXY}
	Suppose that $ (x_{k} )_{k \in \mathbb{N}}$ is  bounded and that    
$(\forall k \in \mathbb{N})$ $\beta_{k} +\gamma_{k} \leq 1$,  $\alpha_{k} \to 0$, $ \limsup_{k \to \infty} |\beta_{k}|  < 1$, $1- \alpha_{k} -\beta_{k} -\gamma_{k} \to 0$,  
 $0 < \liminf_{k \to \infty}  1 -\beta_{k} - \frac{\gamma_{k}}{2}  \leq \limsup_{k \to \infty}    1 -\beta_{k} - \frac{\gamma_{k}}{2}   < 1$, $ \delta_{k}  e_{k} \to 0$,  
	and	$1 -\frac{c_{k}}{c_{k+1}} \to 0$.
Then the following hold. 
\begin{enumerate}
	\item \label{prop:weakclustersXY:JckBounded} $ (\J_{c_{k} A} x_{k} )_{k \in \mathbb{N}}$  is bounded.
	\item \label{prop:weakclustersXY:Jck}  $x_{k} - \J_{c_{k} A}(x_{k} ) \to 0$. 
	\item\label{prop:weakclustersXY:OmegaZer} If $\inf_{k \in \mathbb{N}}c_{k} >0$ or $c_{k} \to \infty$, then $\varnothing \neq \Omega \subseteq \zer A$.   
\end{enumerate}
\end{proposition}

\begin{proof}
	\cref{prop:weakclustersXY:JckBounded}: According to our assumption, it is easy to see that
	\begin{align*}
\frac{1}{2}\liminf_{k \to \infty} \abs{\gamma_{k}} &=   \liminf_{k \to \infty} \abs{ 1 -\beta_{k} - \frac{\gamma_{k}}{2} - \left( 1- \alpha_{k} -\beta_{k} -\gamma_{k} \right) -\alpha_{k} } \\
&\geq  \liminf_{k \to \infty} \abs{ 1 -\beta_{k} - \frac{\gamma_{k}}{2} } - \limsup_{k \to \infty} \abs{ 1- \alpha_{k} -\beta_{k} -\gamma_{k}} -\limsup_{k \to \infty} \abs{\alpha_{k}}\\
&=  \liminf_{k \to \infty} \abs{ 1 -\beta_{k} - \frac{\gamma_{k}}{2} } >0.
	\end{align*} 
This combined with our assumptions and \cref{prop:Jckboundedweakcluster}\cref{prop:Jckboundedweakcluster:Jckbounded}  entails the boundedness of 	$ (\J_{c_{k} A} x_{k} )_{k \in \mathbb{N}}$.

\cref{prop:weakclustersXY:Jck}:			Denote by $(\forall k \in \mathbb{N})$ $\eta_{k}:= 1 -\beta_{k} - \frac{\gamma_{k}}{2} $.
 Inasmuch as $0 < \liminf_{i \to \infty} \eta_{i} \leq \limsup_{i \to \infty}  \eta_{i} < 1$ and $ \limsup_{k \to \infty} \abs{\beta_{k}}  < 1$, without loss of generality, we assume that 
	\begin{align*}
	(\forall k \in \mathbb{N}) \quad \eta_{k} \in \left]0,1\right] \quad \text{and} \quad \abs{\beta_{k}} <1,
	\end{align*}
	which, in connection with $(\forall k \in \mathbb{N})$ $\beta_{k} +\gamma_{k} \leq 1$, implies that
	\begin{align*}
(\forall k \in \mathbb{N}) \quad 	1+\frac{\gamma_{k}}{2} = \left(1 -\beta_{k} - \frac{\gamma_{k}}{2}\right)+\beta_{k} +\gamma_{k} \leq 2  \quad \text{and} \quad  \frac{\gamma_{k}}{2} +1 \geq 1-\beta_{k}\geq 1-\abs{ \beta_{k}} >0.
	\end{align*}
	Hence, $(\forall k \in \mathbb{N})$ $\gamma_{k} \in \left]-2,2\right]$ and $(\gamma_{i})_{i \in \mathbb{N}}$ is bounded.

Set $(\forall k \in \mathbb{N})$ $T_{k}:= 2\J_{c_{k } A} -\Id  $.  Bearing \cref{leamma:xkTk} in mind, we observe that 
\begin{align} \label{eq:prop:weakclustersXY:eta}
	(\forall k \in \mathbb{N})  \quad  x_{k+1} =  (1-\eta_{k})  x_{k} +\eta_{k} y_{k},
\end{align}
where $(\forall k \in \mathbb{N})$ $y_{k}:= \frac{1}{\eta_{k}} \left( \frac{\gamma_{k}}{2} T_{k}(x_{k}) +\alpha_{k} u+\delta_{k} e_{k} \right)$.  
Note that for every $k \in \mathbb{N}$,
	\begin{align}\label{eq:prop:weakclustersXY:y}
	\norm{y_{k+1} -y_{k}} \leq \norm{ \frac{\gamma_{k+1} }{2\eta_{k+1}} T_{k+1}(x_{k+1}) -\frac{\gamma_{k} }{2\eta_{k}} T_{k}(x_{k})}   + \left| \frac{\alpha_{k+1} }{\eta_{k+1}} -\frac{\alpha_{k} }{\eta_{k }}  \right| \norm{u}+\norm{ \frac{\delta_{k+1} }{\eta_{k+1}}e_{k+1} -\frac{\delta_{k} }{\eta_{k}}e_{k} }.
	\end{align}
Moreover, apply \cref{fact:ineqTIden}\cref{fact:ineqTIden:ineq} and recall that $(\forall k \in \mathbb{N})$ $T_{k}$ is nonexpansive in the following second inequality to see that  for every $k \in \mathbb{N}$,
\begin{subequations} \label{eq:prop:weakclustersXY:Tk}
	\begin{align}
	& \norm{ \frac{\gamma_{k+1} }{2\eta_{k+1}} T_{k+1}(x_{k+1}) -\frac{\gamma_{k} }{2\eta_{k}} T_{k}(x_{k})}  \\
	\leq\, &
	\frac{|\gamma_{k+1}| }{2\eta_{k+1}}   \norm{T_{k+1}(x_{k+1}) - T_{k}(x_{k+1})} +\frac{ |\gamma_{k+1}| }{2\eta_{k+1}}  \norm{T_{k}(x_{k+1}) - T_{k}(x_{k})   }
	+ \left| \frac{\gamma_{k+1} }{2\eta_{k+1}} -  \frac{\gamma_{k} }{2\eta_{k}}  \right|  \norm{T_{k}(x_{k}) }\\
	\leq\, & \frac{ |\gamma_{k+1}| }{2\eta_{k+1}} \left|  1 -\frac{c_{k}}{c_{k+1}}\right| \norm{T_{k+1}(x_{k+1}) -  x_{k+1}}+\frac{|\gamma_{k+1}| }{2\eta_{k+1}}  \norm{ x_{k+1} -  x_{k}   }  + \left| \frac{\gamma_{k+1} }{2\eta_{k+1}} -  \frac{\gamma_{k} }{2\eta_{k}}  \right|  \norm{T_{k}(x_{k}) }.
	\end{align}
\end{subequations}
Because $(\forall k \in \mathbb{N})$ $\abs{ \beta_{k}} <1$ and $\beta_{k} +\gamma_{k} \leq 1$, for every $k \in \mathbb{N}$, if $ \gamma_{k+1} \leq 0$, then  $ |\gamma_{k+1} | +\gamma_{k+1} +2 \beta_{k+1} =2 \beta_{k+1}  \leq 2$; otherwise, $ |\gamma_{k+1} | +\gamma_{k+1} +2 \beta_{k+1} =2( \gamma_{k+1} +\beta_{k+1} ) \leq 2$.  This together  with the equivalence
  $(\forall k \in \mathbb{N})$  $\frac{ |\gamma_{k+1} | }{2\eta_{k+1}} = \frac{ |\gamma_{k+1} | }{2- 2 \beta_{k+1}-\gamma_{k+1}}   \leq 1 \Leftrightarrow |\gamma_{k+1} | +\gamma_{k+1} +2 \beta_{k+1} \leq 2$  implies  that  $(\forall k \in \mathbb{N})$  $\frac{ |\gamma_{k+1} | }{2\eta_{k+1}}    \leq 1 $.
 Then 
combine \cref{eq:prop:weakclustersXY:y} and \cref{eq:prop:weakclustersXY:Tk} to obtain that  
\begin{align} \label{eq:prop:weakclustersXY:Lambda}
(\forall k \in \mathbb{N}) \quad \norm{y_{k+1} -y_{k}} \leq  \norm{ x_{k+1} -  x_{k}   } + \Lambda (k),
\end{align}
where $(\forall k \in \mathbb{N})$ $ \Lambda (k) := \frac{ |\gamma_{k+1}| }{2\eta_{k+1}} \left|  1 -\frac{c_{k}}{c_{k+1}}\right| \norm{T_{k+1}(x_{k+1}) -  x_{k+1}} + \left| \frac{\gamma_{k+1} }{2\eta_{k+1}} -  \frac{\gamma_{k} }{2\eta_{k}}  \right|  \norm{T_{k}(x_{k}) } +  \left| \frac{\alpha_{k+1} }{\eta_{k+1}} -\frac{\alpha_{k} }{\eta_{k }}  \right| \norm{u}+\norm{ \frac{\delta_{k+1} }{\eta_{k+1}}e_{k+1} -\frac{\delta_{k} }{\eta_{k}}e_{k} }$.

Note that  
 the boundedness of $ (x_{k} )_{k \in \mathbb{N}}$ and $ (\J_{c_{k} A} x_{k} )_{k \in \mathbb{N}}$ implies that $\left( \norm{T_{k+1}(x_{k+1}) -  x_{k+1}}  \right)_{k \in \mathbb{N}}$ and $(\norm{T_{k}(x_{k}) } )_{k \in \mathbb{N}}$ are bounded.   Combine this  with $0 < \liminf_{k \to \infty} \eta_{k} $ and the boundedness of $(\gamma_{k})_{k \in \mathbb{N}}$ to deduce the boundedness of   $ (y_{k} )_{k \in \mathbb{N}}$.
  
In addition, by some easy algebra, it is not difficult to verify that 
\begin{align*}
&\left| \frac{\gamma_{k+1} }{2\eta_{k+1}} -  \frac{\gamma_{k} }{2\eta_{k}}  \right| = \frac{1}{2\eta_{k+1}\eta_{k }}   \abs{ \gamma_{k+1} (1-  \beta_{k}) -   \gamma_{k}  (1- \beta_{k+1} )  };\\
&  \gamma_{k+1} (1-  \beta_{k}) -   \gamma_{k}  (1- \beta_{k+1} )    =  \gamma_{k+1} \left(1- \alpha_{k} -\beta_{k} -\gamma_{k}\right) -   \gamma_{k}  \left(1- \alpha_{k+1} -\beta_{k+1} -\gamma_{k+1}\right)  + \gamma_{k+1}  \alpha_{k}  - \gamma_{k}  \alpha_{k+1};\\
&\left| \frac{\alpha_{k+1} }{\eta_{k+1}} -\frac{\alpha_{k} }{\eta_{k }}  \right|  \leq \frac{ |\alpha_{k+1}| +   |\alpha_{k}| }{ \eta_{k+1}\eta_{k }};\\
&\norm{ \frac{\delta_{k+1} }{\eta_{k+1}}e_{k+1} -\frac{\delta_{k} }{\eta_{k}}e_{k} }  \leq \frac{  \norm{ \delta_{k+1}  e_{k+1}} +   \norm{ \delta_{k}  e_{k}}}{ \eta_{k+1}\eta_{k }},
\end{align*}
which, connecting with the assumption, yields  $\lim_{k \to \infty} \Lambda (k) =0$. This and  \cref{eq:prop:weakclustersXY:Lambda}  necessitate
\begin{align*}
\limsup_{k \to \infty} \norm{y_{k+1} -y_{k}} - \norm{ x_{k+1} -  x_{k}   } \leq 0.
\end{align*}
Employing \cref{eq:prop:weakclustersXY:eta} and applying \cref{fact:sequence:ukvk} with $(\forall k \in \mathbb{N})$ $u_{k}=x_{k}$, $\alpha_{k}=\eta_{k}$, and $v_{k}=y_{k}$, we know that the inequality above  leads to
\begin{align}\label{eq:prop:weakclustersXY:xk+1xk}
y_{k}-x_{k} \to 0 \quad \text{and} \quad \norm{x_{k+1} -x_{k}} \stackrel{\cref{eq:prop:weakclustersXY:eta}}{=} \eta_{k } \norm{ y_{k} -x_{k}} \to 0.
\end{align}
Notice that 
the assumptions $\limsup_{k \to \infty} |\beta_{k}|  < 1$  and $	(\forall k \in \mathbb{N}) $ $ \abs{\beta_{k}} <1$ ensure  the boundedness of $\left( \frac{ 1 }{ 1 - |\beta_{k}|}  \right)_{k \in \mathbb{N}}$. 
Furthermore, 
for every $ k \in \mathbb{N}$,
\begin{align*}
 \norm{x_{k} - \J_{c_{k} A} x_{k}}  
&~ \leq~  \norm{x_{k} -x_{k+1} } + \norm{ x_{k+1}  -\J_{c_{k} A} x_{k} }\\
&\stackrel{\cref{eq:xk}}{=}  \norm{x_{k} -x_{k+1} } + \norm{ \alpha_{k} u +\beta_{k} x_{k} + \gamma_{k} \J_{c_{k} A} (x_{k}) +\delta_{k} e_{k} -\J_{c_{k} A} x_{k} }\\
&~\leq~ 	 \norm{x_{k} -x_{k+1} } + |\alpha_{k}| \norm{  u-\J_{c_{k} A} x_{k}} +  |\beta_{k} | \norm{ x_{k}    -\J_{c_{k} A} x_{k}  } +\norm{ \delta_{k} e_{k}}  
 + \abs{\varphi_{k}} \norm{\J_{c_{k} A} x_{k} }.
\end{align*}
The inequalities above ensure  that 
\begin{align*}
(\forall k \in \mathbb{N}) \quad 	\norm{x_{k} - \J_{c_{k} A} x_{k}}  
\leq   \frac{ 1 }{ 1 - |\beta_{k}|} \left(  \norm{x_{k} -x_{k+1} } + |\alpha_{k}| \norm{  u-\J_{c_{k} A} x_{k}} +  \norm{ \delta_{k} e_{k}} + \abs{\varphi_{k}} \norm{\J_{c_{k} A} x_{k} } \right),
\end{align*}
which, employing \cref{eq:prop:weakclustersXY:xk+1xk} and the assumption, guarantees that $x_{k} - \J_{c_{k} A} x_{k} \to 0$.

\cref{prop:weakclustersXY:OmegaZer}: This is clear from \cref{prop:weakclustersXY:Jck} and	\cref{prop:weakclusterCBAR}.  
\end{proof}

The following \cref{prop:xk-Jckto0} is inspired by \cite[Lemma~3.2]{MarinoXu2004}. Moreover, if $(\forall k \in \mathbb{N})$ $\alpha_{k} \equiv 0$,   $\delta_{k} \equiv 1$, $\gamma_{k} \in \left]0,2\right[$, and $\beta_{k} =1-\gamma_{k}$, then \cref{prop:xk-Jckto0}\cref{prop:xk-Jckto0:<infty}$\&$\cref{prop:xk-Jckto0:liminfto}$\&$\cref{prop:xk-Jckto0:to} reduce to \cite[Lemma~3.2]{MarinoXu2004}.

\begin{proposition} \label{prop:xk-Jckto0}
Suppose that  $\zer A \neq \varnothing$, that   $(\forall k \in \mathbb{N})$   $\abs{  \beta_{k} + \frac{\gamma_{k}}{2}}  +\abs{\frac{\gamma_{k}}{2}}  \leq 1$  and $\gamma_{k}( \beta_{k} +\gamma_{k})  \geq 0$, and that $\sum_{i \in \mathbb{N}} \abs{\alpha_{i}} <\infty$, $\sum_{i \in \mathbb{N}} \abs{1-  \beta_{i}  -\gamma_{i}}< \infty$, and  $\sum_{i \in \mathbb{N}} \norm{ \delta_{i} e_{i}} < \infty$.   Then the following statements hold. 
\begin{enumerate}
		\item \label{prop:xk-Jckto0:<infty}  $\sum^{\infty}_{k=0} \gamma_{k}	(2\beta_{k}+\gamma_{k})\norm{x_{k} -  \J_{c_{k} A} x_{k} }^{2} <\infty$.
	\item \label{prop:xk-Jckto0:liminfto} If  $\sum_{k \in \mathbb{N}} \gamma_{k}	(2\beta_{k}+\gamma_{k}) =\infty$ and $\inf_{k \in \mathbb{N}} \gamma_{k}	(2\beta_{k}+\gamma_{k}) \geq 0$, then $\liminf_{k \to \infty} \norm{x_{k} - \J_{c_{k} A}(x_{k} ) }=0$.
	\item \label{prop:xk-Jckto0:to} Suppose that  $\liminf_{k \to \infty} \gamma_{k}	(2\beta_{k}+\gamma_{k}) > 0$.
		Then  $x_{k} - \J_{c_{k} A}(x_{k} ) \to 0$.
	\item  \label{prop:xk-Jckto0:OmegaZer} Suppose that  $\liminf_{k \to \infty} \gamma_{k}	(2\beta_{k}+\gamma_{k}) > 0$ and that $ \inf_{k \in \mathbb{N}}c_{k} >0$ or $c_{k} \to \infty$. Then  $\varnothing \neq \Omega \subseteq \zer A$.   
\end{enumerate}
 \end{proposition}

\begin{proof}
	The assumption and \cref{prop:xkbounded}\cref{prop:xkbounded:leq1}    imply the boundedness of $(x_{k})_{k \in \mathbb{N}}$. 
	
\cref{prop:xk-Jckto0:<infty}:	  
Let $p \in \zer A$.   Set $(\forall k \in \mathbb{N})$ $M(k) :=\norm{\alpha_{k} u + \delta_{k} e_{k} - (1-\beta_{k} -\gamma_{k} )p}$. According to \cref{lemma:xk+1-p}\cref{lemma:xk+1-p:xk-Jckto}, for every $k \in \mathbb{N}$, 
$ \gamma_{k}	(2\beta_{k}+\gamma_{k}) \norm{x_{k} -  \J_{c_{k} A} x_{k} }^{2} \leq   (  \beta_{k} +\gamma_{k} )^{2} \norm{  x_{k} -p}^{2} -  \norm{x_{k+1} -p}^{2} +2 M (k)\norm{x_{k+1} -p} \leq \norm{  x_{k} -p}^{2} -  \norm{x_{k+1} -p}^{2} +2 M (k)\norm{x_{k+1} -p}$,
which, combining with the assumption, derives
	\begin{align*}
\sum^{k}_{i=0} \gamma_{i}	(2\beta_{i}+\gamma_{i})\norm{x_{i} -  \J_{c_{i} A} x_{i} }^{2} \leq  \norm{  x_{0} -p}^{2}  +2\sum^{k}_{i=0} M(i)\norm{x_{i+1} -p} \leq \norm{  x_{0} -p}^{2} +2L_{1}L_{2}<\infty,
\end{align*}
	where  $L_{1}:=\sup_{k \in \mathbb{N}} \norm{x_{k} -p} <\infty$ and $L_{2}:=\sum_{k \in \mathbb{N}} M (k) <\infty$. This verifies	\cref{prop:xk-Jckto0:<infty}.

	\cref{prop:xk-Jckto0:liminfto}:  According to the assumption and  \cref{prop:xk-Jckto0:<infty},
	\begin{align*}
	\infty > \sum^{\infty}_{k=0} \gamma_{k}	(2\beta_{k}+\gamma_{k})\norm{x_{k} -  \J_{c_{k} A} x_{k} }^{2}  \geq \liminf_{k \to \infty} \norm{x_{k} - \J_{c_{k} A}(x_{k} ) } \sum_{k \in \mathbb{N}} \gamma_{k}	(2\beta_{k}+\gamma_{k}),
	\end{align*}  
	which, noticing $\sum_{k \in \mathbb{N}} \gamma_{k}	(2\beta_{k}+\gamma_{k}) =\infty$, forces   $\liminf_{k \to \infty} \norm{x_{k} - \J_{c_{k} A}(x_{k} ) } =0$.

 \cref{prop:xk-Jckto0:to}: As a consequence of $\eta:=\liminf_{k \to \infty} \gamma_{k}	(2\beta_{k}+\gamma_{k}) > 0$,  there exists $N \in \mathbb{N}$ such that $(\forall k \geq N)$ $ \gamma_{k}	(2\beta_{k}+\gamma_{k}) \geq \frac{\eta}{2} >0$. Hence, for every $k \geq N$
\begin{align*} 
\sum^{k}_{i=0} \gamma_{i}	(2\beta_{i}+\gamma_{i})\norm{x_{i} -  \J_{c_{i} A} x_{i} }^{2} \geq \sum^{N-1}_{i=0} \gamma_{i}	(2\beta_{i}+\gamma_{i})\norm{x_{i} -  \J_{c_{i} A} x_{i} }^{2}  + \frac{\eta}{2} \sum^{k}_{i=N} \norm{x_{i} -  \J_{c_{i} A} x_{i} }^{2}. 
\end{align*} 
Combine this with \cref{prop:xk-Jckto0:<infty} to obtain that $ \sum^{\infty}_{i=N} \norm{x_{i} -  \J_{c_{i} A} x_{i} }^{2} <\infty$,
which yields $x_{k} - \J_{c_{k} A}(x_{k} ) \to 0$.

\cref{prop:xk-Jckto0:OmegaZer}: This is immediate from \cref{prop:xk-Jckto0:to} and 	\cref{prop:weakclusterCBAR}. 
\end{proof}

The following proof is motivated by \cite[Theorem~3.6]{MarinoXu2004}.
 \begin{proposition}\label{prop:weakclustersxkJck}
	Suppose that  $\zer A \neq \varnothing$,   that  $(\forall k \in \mathbb{N})$  $\abs{  \beta_{k} + \frac{\gamma_{k}}{2}}  +\abs{\frac{\gamma_{k}}{2}}  \leq 1$,  $\gamma_{k}( \beta_{k} +\gamma_{k}) \geq 0$, $\gamma_{k}	(2\beta_{k}+\gamma_{k}) \geq 0$, and 
	$\abs{  \beta_{k} + \gamma_{k}} \abs{\gamma_{k}}    \leq \max \{  1-\abs{\beta_{k}}  , 2- 2 \abs{\beta_{k} +\frac{\gamma_{k}}{2} } \}$,
	 that $\sum_{k \in \mathbb{N}} \abs{\alpha_{k}} <\infty$,  $\sum_{k \in \mathbb{N}} \abs{1-  \beta_{k}  -\gamma_{k}}< \infty$,   $\sum_{k \in \mathbb{N}} \gamma_{k}	(2\beta_{k}+\gamma_{k}) =\infty$, and $\sum_{k \in \mathbb{N}} \norm{ \delta_{k} e_{k}} < \infty$,   and that $\bar{c}:=\inf_{k \in \mathbb{N}}c_{k} >0$ and $\sum_{k \in \mathbb{N}} \abs{c_{k+1} -c_{k}} <  \infty$.
	Then  $x_{k} - \J_{c_{k} A}(x_{k} ) \to 0$ and  $ \varnothing \neq \Omega \subseteq \zer A$.    
\end{proposition}

\begin{proof}
		Note that, via  \cref{prop:xkbounded}\cref{prop:xkbounded:leq1}  and \cref{prop:xk-Jckto0}\cref{prop:xk-Jckto0:liminfto}, our assumptions force that    $ (x_{k} )_{k \in \mathbb{N}}$ is bounded and that 
	\begin{align}\label{eq:prop:weakclustersxkJck:liminf}
	\liminf_{k \to \infty} \norm{x_{k} - \J_{c_{k} A}x_{k} }=0.
	\end{align}
	Combining this with \cref{prop:weakclusterCBAR}\cref{prop:weakclusterCBAR:>0} and the assumption  $\bar{c}:=\inf_{k \in \mathbb{N}}c_{k} >0$, we know that
	it suffices to show that  $\lim_{k \to \infty} \norm{x_{k} - \J_{c_{k} A}x_{k} } =0$. 

		Set $(\forall k \in \mathbb{N})$ $T_{k}:= 2\J_{c_{k } A} -\Id  $.  Denote by $(\forall k \in \mathbb{N})$   $s_{k}:=\norm{x_{k} - \J_{c_{k } A} x_{k}}$. Then 
		\begin{align}\label{eq:prop:weakclustersxkJck:JT}
	(\forall k \in \mathbb{N}) \quad \norm{x_{k} -  T_{k }x_{k}}  =2  \norm{x_{k} - \J_{c_{k} A}(x_{k} ) } =2s_{k}.
	\end{align}
	 Due to \cref{prop:Jckboundedweakcluster}\cref{prop:Jckboundedweakcluster:Jckbounded},    $(\J_{c_{k} A}x_{k} )_{k \in \mathbb{N}}$ and $(T_{k} x_{k})_{k \in \mathbb{N}}$ are bounded. Hence, $\hat{s}:= \sup_{k \in \mathbb{N}}s_{k} <\infty$.
		 In view of \cref{leamma:xkTk}, $(\forall k \in \mathbb{N})$ $T_{k}$ is nonexpansive and $x_{k+1} =\left( \beta_{k} +\frac{\gamma_{k}}{2} \right)  x_{k} + \frac{\gamma_{k}}{2} T_{k}(x_{k}) +\alpha_{k} u+\delta_{k} e_{k}$,
		which ensures  that for every $k \in \mathbb{N}$,
			\begin{align*}
		&	\norm{x_{k+1}  -  T_{k+1}x_{k+1}} \\
		= \,& \norm{ \left( \beta_{k} +\frac{\gamma_{k}}{2} \right) ( x_{k} -  T_{k }x_{k} ) + ( \beta_{k} + \gamma_{k}) (T_{k}(x_{k}) -T_{k+1}x_{k+1} )+\alpha_{k} u+\delta_{k} e_{k} - (1-\beta_{k} -\gamma_{k}) T_{k+1}x_{k+1}} \\
		\leq \,& \abs{\beta_{k} +\frac{\gamma_{k}}{2} }  \norm{x_{k} -  T_{k }x_{k}} + \abs{  \beta_{k} + \gamma_{k}} \norm{T_{k}(x_{k}) -T_{k+1}x_{k+1} } +\norm{\alpha_{k} u+\delta_{k} e_{k} - (1-\beta_{k} -\gamma_{k}) T_{k+1}x_{k+1}}.
		\end{align*}
Set $(\forall k \in \mathbb{N})$  $F_{1}(k):= \norm{\alpha_{k} u+\delta_{k} e_{k} - (1-\beta_{k} -\gamma_{k}) T_{k+1}x_{k+1} }$. Then we establish  that  for every $k \in \mathbb{N}$,
\begin{align}\label{eq:prop:weakclustersxkJck:xk+1Tk+1}
\norm{x_{k+1}  -  T_{k+1}x_{k+1}} \leq \abs{\beta_{k} +\frac{\gamma_{k}}{2} }  \norm{x_{k} -  T_{k }x_{k}} + \abs{  \beta_{k} + \gamma_{k}} \norm{T_{k}(x_{k}) -T_{k+1}x_{k+1} } +F_{1}(k).
\end{align}
Similarly, via \cref{eq:xk},  we get that for every $k \in \mathbb{N}$,
\begin{align}\label{eq:prop:weakclustersxkJck:Jck+1}
\norm{x_{k+1} - \J_{c_{k+1} A}x_{k+1}} \leq \abs{\beta_{k}} \norm{x_{k} - \J_{c_{k} A} x_{k}} + \abs{\beta_{k} +\gamma_{k}} \norm{  \J_{c_{k} A} x_{k} - \J_{c_{k+1} A}x_{k+1}} +F_{2}(k),
\end{align}
where $(\forall k \in \mathbb{N})$  $F_{2}(k):= \norm{\alpha_{k} u+\delta_{k} e_{k} - (1-\beta_{k} -\gamma_{k})  \J_{c_{k+1} A}x_{k+1}}$.

Furthermore, using \cref{eq:xk} again, we observe that for every $ k \in \mathbb{N}$,
\begin{subequations}
\begin{align} 
\norm{ x_{k+1} -x_{k}} &=\norm{ \alpha_{k} u +\beta_{k} x_{k} + \gamma_{k} \J_{c_{k} A} (x_{k}) +\delta_{k} e_{k} -x_{k}} \leq \abs{\gamma_{k}} \norm{ x_{k} - \J_{c_{k} A} x_{k}} +G_{1}(k), \label{eq:prop:weakclustersxkJck:xkxk+1}\\
\norm{ x_{k+1} -  \J_{c_{k} A} x_{k}} &=\norm{ \alpha_{k} u +\beta_{k} x_{k} + \gamma_{k} \J_{c_{k} A} (x_{k}) +\delta_{k} e_{k} -  \J_{c_{k} A} x_{k}} \leq \abs{\beta_{k}} \norm{ x_{k} - \J_{c_{k} A} x_{k}} +G_{2}(k), \label{eq:prop:weakclustersxkJck:Jck}
\end{align}	
\end{subequations}
where  $G_{1}(k) :=  \norm{\alpha_{k} u+\delta_{k} e_{k} - (1-\beta_{k} -\gamma_{k})  x_{k} }$ and $G_{2}(k) :=  \norm{\alpha_{k} u+\delta_{k} e_{k} - (1-\beta_{k} -\gamma_{k})  \J_{c_{k} A} x_{k} }$. 
		
Let $k \in \mathbb{N}$. We have exactly the following two cases.
 
 	\emph{Case~1}: $c_{k} \leq c_{k+1}$. Then invoking \cref{fact:ineqTIden}\cref{fact:ineqTIden:eq} in the following equality, using the nonexpansiveness of $T_{k}$ in the first inequality, and employing $\norm{  x_{k} -\J_{c_{k +1} A} x_{k+1}} \leq \norm{  x_{k} - x_{k+1}} +\norm{  x_{k+1} -\J_{c_{k +1} A} x_{k+1}}$ in the second inequality below,  we get that  
\begin{align*}
& \norm{T_{k}(x_{k}) -T_{k+1}(x_{k+1}) } \\
=\,&  \norm{T_{k}(x_{k}) -T_{k} \left(  \frac{c_{k}}{c_{k+1}} x_{k+1}  + \left( 1-   \frac{c_{k}}{c_{k+1}} \right) \J_{c_{k +1} A} x_{k+1}  \right)}  +\left( 1-   \frac{c_{k}}{c_{k+1}} \right) \norm{ x_{k+1}  -  \J_{c_{k +1} A} x_{k+1} }\\
\leq\, & \frac{c_{k}}{c_{k+1}}  \norm{x_{k} -x_{k+1} } +  \left( 1-   \frac{c_{k}}{c_{k+1}} \right)  \norm{   x_{k} -\J_{c_{k +1} A} x_{k+1}} + \left( 1-   \frac{c_{k}}{c_{k+1}} \right) \norm{ x_{k+1}  -  \J_{c_{k +1} A} x_{k+1} }\\
\leq\, &  \norm{x_{k} -x_{k+1} } + 2 \left( 1-   \frac{c_{k}}{c_{k+1}} \right)  \norm{  x_{k+1} -\J_{c_{k +1} A} x_{k+1}}\\
\stackrel{\cref{eq:prop:weakclustersxkJck:xkxk+1}}{\leq}&  \abs{\gamma_{k}} \norm{ x_{k} - \J_{c_{k} A} x_{k}} +G_{1}(k) +2 \left( 1-   \frac{c_{k}}{c_{k+1}} \right)  \norm{ x_{k+1} -\J_{c_{k +1} A} x_{k+1}}.
\end{align*}
Applying this result, \cref{eq:prop:weakclustersxkJck:JT}, and \cref{eq:prop:weakclustersxkJck:xk+1Tk+1} in the first inequality below and employing $\abs{  \beta_{k} + \gamma_{k}} \leq \abs{  \beta_{k} + \frac{\gamma_{k}}{2}}  +\abs{\frac{\gamma_{k}}{2}}  \leq 1$ and $2 \abs{\beta_{k} +\frac{\gamma_{k}}{2} }  + \abs{  \beta_{k} + \gamma_{k}}\abs{\gamma_{k}}\leq  2\lr{\abs{  \beta_{k} + \frac{\gamma_{k}}{2}}  +\abs{\frac{\gamma_{k}}{2}}  } \leq 2$ in the second inequality below, we get that
\begin{align*}
2 s_{k+1} &\stackrel{\cref{eq:prop:weakclustersxkJck:JT}}{=} \norm{x_{k+1} -  T_{k+1 }x_{k+1}} \\
 &~\leq~ 2 \abs{\beta_{k} +\frac{\gamma_{k}}{2} }  s_{k} + \abs{  \beta_{k} + \gamma_{k}} \left( \abs{\gamma_{k}}s_{k} +G_{1}(k) +2 \left( 1-   \frac{c_{k}}{c_{k+1}} \right) s_{k+1}  \right)   +F_{1}(k)\\
&~\leq~  2s_{k} +2   \left( 1-   \frac{c_{k}}{c_{k+1}} \right) s_{k+1} +  F_{1}(k)  +G_{1}(k),
\end{align*}
which implies that
\begin{align}\label{eq:prop:weakclustersxkJck:sckleqck+1}
\frac{s_{k+1}}{c_{k+1}} \leq \frac{s_{k }}{c_{k }}  +\frac{  F_{1}(k)  +G_{1}(k)}{2 \bar{c}}.
\end{align}

\emph{Case~2}: $c_{k+1} < c_{k}$. We claim that 
\begin{align} \label{eq:prop:weakclustersxkJck:claim}
  s_{k+1} \leq s_{k}  + \left( 1- \frac{c_{k+1}}{c_{k}} \right) s_{k}  + F_{1} (k) +F_{2} (k)+ G_{1}(k) +G_{2}(k).
\end{align}

\emph{Case~2.1}:  Assume that   $\abs{\beta_{k} +\gamma_{k}}  \abs{\gamma_{k}}  \leq 1-\abs{\beta_{k}}  $, i.e., $\abs{\beta_{k}} + \abs{\beta_{k} +\gamma_{k}}  \abs{\gamma_{k}} \leq 1$. Applying  \cref{fact:ineqJIden} in the first equality, utilizing   the nonexpansiveness of $\J_{ c_{k} A}$ in the first inequality, and employing \cref{eq:prop:weakclustersxkJck:xkxk+1} and \cref{eq:prop:weakclustersxkJck:Jck} in the second inequality, we deduce that
\begin{align*}
&\norm{\J_{ c_{k} A} x_{k} - \J_{c_{k+1}A}x_{k+1}} \\
=  &\norm{\J_{ c_{k+1} A}  \left( \frac{c_{k+1}}{c_{k}} x_{k} + \left( 1- \frac{c_{k+1}}{c_{k}} \right) \J_{ c_{k} A} x_{k}    \right) - \J_{c_{k+1}A}x_{k+1}} \\
\leq &  \frac{c_{k+1}}{c_{k}} \norm{x_{k} -x_{k+1} } +  \left( 1- \frac{c_{k+1}}{c_{k}} \right) \norm{\J_{ c_{k} A} x_{k}  - x_{k+1} }\\
\leq & \abs{\gamma_{k}} \frac{c_{k+1}}{c_{k}} \norm{ x_{k} - \J_{c_{k} A} x_{k}} + \abs{\beta_{k}} \left( 1- \frac{c_{k+1}}{c_{k}} \right) \norm{ x_{k} - \J_{c_{k} A} x_{k}} +\frac{c_{k+1}}{c_{k}} G_{1}(k)  + \left( 1- \frac{c_{k+1}}{c_{k}} \right) G_{2}(k).
\end{align*}
	Combine this with \cref{eq:prop:weakclustersxkJck:Jck+1} and some easy algebra to get that
\begin{align*}
s_{k+1} &\leq \abs{\beta_{k}} s_{k} + \abs{\beta_{k} +\gamma_{k}} \left(   
\left( \abs{\gamma_{k}} \frac{c_{k+1}}{c_{k}}  + \abs{\beta_{k}} \left( 1- \frac{c_{k+1}}{c_{k}} \right) \right) s_{k} +\frac{c_{k+1}}{c_{k}} G_{1}(k)  + \left( 1- \frac{c_{k+1}}{c_{k}} \right) G_{2}(k)
\right)+F_{2}(k)\\
&\leq \left( \abs{\beta_{k}} + \abs{\beta_{k} +\gamma_{k}}  \abs{\gamma_{k}} \right) s_{k} +  \abs{\beta_{k} +\gamma_{k}}  \left( 1- \frac{c_{k+1}}{c_{k}} \right) \left( \abs{\beta_{k}}-\abs{\gamma_{k}}  \right)s_{k}+ G_{1}(k) +G_{2}(k) +F_{2} (k)\\
&\leq  s_{k} + \left( 1- \frac{c_{k+1}}{c_{k}} \right) s_{k}  +  G_{1}(k) +G_{2}(k) +F_{2} (k),
\end{align*}
which, using $\abs{\beta_{k}}-\abs{\gamma_{k}} \leq \abs{\beta_{k} +\gamma_{k} } \leq \abs{  \beta_{k} + \frac{\gamma_{k}}{2}}  +\abs{\frac{\gamma_{k}}{2}}  \leq 1$ and   the assumption $\abs{\beta_{k}} + \abs{\beta_{k} +\gamma_{k}}  \abs{\gamma_{k}} \leq 1$ in the last inequality,  verifies \cref{eq:prop:weakclustersxkJck:claim}.

\emph{Case~2.2}: Assume that $(\forall k \in \mathbb{N})$ $\abs{  \beta_{k} + \gamma_{k}} \abs{\gamma_{k}}   \leq 2- 2 \abs{\beta_{k} +\frac{\gamma_{k}}{2} }$, i.e., $ 2 \abs{\beta_{k} +\frac{\gamma_{k}}{2} } +\abs{  \beta_{k} + \gamma_{k}}   \abs{\gamma_{k}}    \leq 2$.  Similarly with the proof of Case 1 above, we have that 
\begin{align}  \label{eq:prop:weakclustersxkJck:TkTk+1}
\norm{T_{k}(x_{k}) -T_{k+1}(x_{k+1}) } \leq \abs{\gamma_{k}} \norm{ x_{k} - \J_{c_{k} A} x_{k}} +G_{1}(k) +2 \left( 1-   \frac{c_{k+1}}{c_{k}} \right)  \norm{  x_{k} -\J_{c_{k} A} x_{k}}.
\end{align}
Invoking \cref{eq:prop:weakclustersxkJck:xk+1Tk+1}, \cref{eq:prop:weakclustersxkJck:JT}, and \cref{eq:prop:weakclustersxkJck:TkTk+1} in the first inequality and using
the assumptions $\abs{  \beta_{k} + \gamma_{k}} \leq \abs{  \beta_{k} + \frac{\gamma_{k}}{2}}  +\abs{\frac{\gamma_{k}}{2}}  \leq 1$ and $ 2 \abs{\beta_{k} +\frac{\gamma_{k}}{2} } +\abs{  \beta_{k} + \gamma_{k}}   \abs{\gamma_{k}}    \leq 2$  in the last inequality below, we observe that
\begin{align*}
  2 s_{k+1} & \stackrel{\cref{eq:prop:weakclustersxkJck:JT}}{=} \norm{x_{k+1} -  T_{k+1 }x_{k+1}} \\
&~\leq`  2 \abs{\beta_{k} +\frac{\gamma_{k}}{2} }  s_{k} + \abs{  \beta_{k} + \gamma_{k}} \left( \left( \abs{\gamma_{k}} +2 \bigg( 1-   \frac{c_{k+1}}{c_{k}} \bigg) \right) s_{k}  +G_{1}(k)     \right)   +F_{1}(k)\\
&~\leq~  \left( 2 \abs{\beta_{k} +\frac{\gamma_{k}}{2} } +\abs{  \beta_{k} + \gamma_{k}}   \abs{\gamma_{k}}    \right) s_{k} + 2\abs{  \beta_{k} + \gamma_{k}}   \left( 1-   \frac{c_{k+1}}{c_{k}} \right) s_{k} +  F_{1}(k)  +G_{1}(k)\\
&~\leq~ 2 s_{k} + 2  \left( 1-   \frac{c_{k+1}}{c_{k}} \right) s_{k} +  F_{1}(k)  +G_{1}(k),
\end{align*}
which confirming  \cref{eq:prop:weakclustersxkJck:claim} as well. 

Therefore, in both subcases, the claim is true and we have that  
\begin{subequations} \label{eq:prop:weakclustersxkJck:sck+1<ck}
	\begin{align}
	\frac{s_{k+1}}{c_{k+1}} &\leq \frac{s_{k}}{c_{k+1}} +  \frac{c_{k} -c_{k+1}}{c_{k}c_{k+1}} s_{k} +\frac{  F_{1}(k) +F_{2}(k)  +G_{1}(k) +G_{2}(k)}{2 \bar{c}}\\
	&= \frac{s_{k}}{c_{k}}  +s_{k} \left( \frac{1}{c_{k+1}} -\frac{1}{c_{k}}  +  \frac{c_{k} -c_{k+1}}{c_{k}c_{k+1}} \right) +\frac{  F_{1}(k) +F_{2}(k)  +G_{1}(k) +G_{2}(k) }{2 \bar{c}}\\
	&=  \frac{s_{k}}{c_{k}}  + \frac{c_{k} -c_{k+1}+ c_{k} - c_{k+1}}{c_{k}c_{k+1}}s_{k} +\frac{ F_{1}(k) +F_{2}(k)  +G_{1}(k) +G_{2}(k) }{2 \bar{c}}\\
	&\leq  \frac{s_{k}}{c_{k}}  + 2 \frac{ c_{k} -c_{k+1}}{\bar{c}^{2} }\hat{s} +\frac{ F_{1}(k) +F_{2}(k)  +G_{1}(k) +G_{2}(k) }{2 \bar{c}}.
	\end{align}
\end{subequations}
	Furthermore, by  assumptions,  $\sum_{k \in \mathbb{N}} \abs{c_{k+1} -c_{k}} <\infty$, $\sum_{k \in \mathbb{N}}   F_{1}(k) <\infty$, $\sum_{k \in \mathbb{N}}   F_{2}(k) <\infty$, $\sum_{k \in \mathbb{N}}   G_{1}(k) <\infty$,	and $\sum_{k \in \mathbb{N}}   G_{2}(k) <\infty$. Hence, applying \cref{fact:SequenceConverg}, \cref{eq:prop:weakclustersxkJck:sckleqck+1}, and \cref{eq:prop:weakclustersxkJck:sck+1<ck}, we obtain that in both cases, $\lim_{k \to \infty} \frac{s_{k}}{c_{k}} $ exists in $\mathbb{R}_{+}$. Clearly, \cref{eq:prop:weakclustersxkJck:liminf} and \cref{eq:prop:weakclustersxkJck:JT}  necessitate $\liminf_{k \to \infty} s_{k} =0$. These results imply that $ \lim_{k \to \infty} \frac{s_{k}}{c_{k}} = \liminf_{k \to \infty} \frac{s_{k}}{c_{k}} \leq \liminf_{k \to \infty}  \frac{s_{k}}{\bar{c}} =0$ and  that $\limsup_{k \to \infty}   s_{k}   =  \limsup_{k \to \infty} \frac{s_{k}}{c_{k}} c_{k}\leq  \lim_{k \to \infty} \frac{s_{k}}{c_{k}} \sup_{k \in \mathbb{N}} c_{k}  =0$ since  
\begin{align*}
\sum_{i \in \mathbb{N}} \abs{c_{i+1} -c_{i}} <\infty \Rightarrow   \sup_{k \in \mathbb{N}} c_{k}   <\infty.
\end{align*} 
Recall that, via \cref{eq:prop:weakclustersxkJck:liminf} and \cref{eq:prop:weakclustersxkJck:JT},   $\liminf_{k \to \infty} s_{k} = \liminf_{k \to \infty}  \norm{x_{k} - \J_{c_{k} A}x_{k} } =0$.

Altogether,  $\lim_{k \to \infty} \norm{x_{k} - \J_{c_{k} A}x_{k} } =\lim_{k \to \infty} s_{k} =0$. 
%
\end{proof}

The following result is inspired by the proof of \cite[Theorem~4.1]{MarinoXu2004}.
\begin{proposition} \label{prop:xk+1xkJckto0}
		Suppose that  $\zer A \neq \varnothing$ and $ (x_{k} )_{k \in \mathbb{N}}$ is bounded, that   $\sum_{i \in \mathbb{N}} \abs{\alpha_{i+1}  -\alpha_{i} }< \infty$ or $(\forall k \in \mathbb{N})$ $ \abs{\gamma_{k} } \neq 1$ with $\lim_{k \to \infty} \frac{\abs{\alpha_{k+1} -\alpha_{k}}}{1-\abs{\gamma_{k+1}}} =0$, that $(\forall k \in \mathbb{N})$ $\abs{\gamma_{k}} \in \left[0,1\right]$ with  $\alpha_{i}+\gamma_{i} \to 1$, $\sum_{i \in \mathbb{N}} (1-\abs{\gamma_{i}}) =\infty$, and $\sum_{i \in \mathbb{N}} \abs{(\alpha_{i+1} +\gamma_{i+1} ) -(\alpha_{i} +\gamma_{i})} < \infty$, that  $\alpha_{k} \to 0$, $\sum_{k \in \mathbb{N}} \abs{\beta_{k}} <\infty$, and $\sum_{k \in \mathbb{N}} \norm{ \delta_{k} e_{k}} < \infty$, and that $\bar{c}:=\inf_{k \in \mathbb{N}}c_{k} >0$ and $\sum_{k \in \mathbb{N}} \abs{c_{k+1}  -c_{k}} < \infty$.			Then  $x_{k} - \J_{c_{k} A}(x_{k} ) \to 0$ and $\varnothing \neq \Omega \subseteq \zer A$. 
\end{proposition}

\begin{proof}
	Because $\bar{c} =\inf_{k \in \mathbb{N}}c_{k} >0$, via \cref{prop:weakclusterCBAR}\cref{prop:weakclusterCBAR:>0}, it suffices to prove $x_{k} - \J_{c_{k} A}(x_{k} ) \to 0$.
	 
Because $ (x_{k} )_{k \in \mathbb{N}}$ is bounded and $\zer A \neq \varnothing$, due to \cref{prop:Jckboundedweakcluster}\cref{prop:Jckboundedweakcluster:Jckbounded}, $ ( \J_{c_{k} A}x_{k} )_{k \in \mathbb{N}}$ is bounded.
	Notice that for every $k \in \mathbb{N}$, $\norm{ x_{k} - \J_{c_{k} A}x_{k}} \leq \norm{ x_{k} -x_{k+1}} +\norm{x_{k+1} -  \J_{c_{k} A}x_{k}}$
	and that 
	\begin{align*}
	\norm{x_{k+1} - \J_{c_{k} A}x_{k}} \stackrel{\cref{eq:xk}}{\leq} \norm{\alpha_{k} (u -  \J_{c_{k} A}x_{k}  )} +\abs{\alpha_{k} +\gamma_{k} -1 }\norm{  \J_{c_{k} A}x_{k} } +\norm{\beta_{k}x_{k}} +\norm{\delta_{k}e_{k}} \to 0.
	\end{align*}
	Hence, it remains to show that $ x_{k} -x_{k+1} \to 0$. 
	
Clearly, $\sum_{k \in \mathbb{N}} |c_{k+1}  -c_{k}| < \infty$ leads to $\hat{c} := \sup_{k \in \mathbb{N}}c_{k} <\infty$.			In view of \cref{fact:ineqJIden}, for every $k \in \mathbb{N}$, 
	\begin{subequations}  \label{eq:prop:xk+1xkJckto0:J}
			\begin{align}
		\norm{   \J_{c_{k+1} A}x_{k+1} - \J_{c_{k} A}x_{k}} &= \norm{\J_{c_{k} A} \left( \frac{c_{k}}{c_{k+1}} x_{k+1}   + \left( 1- \frac{c_{k}}{c_{k+1}}   \right) \J_{c_{k+1} A}x_{k+1} \right) - \J_{c_{k} A}x_{k}}\\
		&\leq \frac{c_{k}}{c_{k+1}}  \norm{x_{k+1}-x_{k} } + \abs{ 1- \frac{c_{k}}{c_{k+1}} }\norm{ \J_{c_{k+1} A}x_{k+1}  - x_{k} }.
		\end{align}
	\end{subequations}
Set  $(\forall k \in \mathbb{N} \smallsetminus \{0\})$  $M(k):=\abs{\gamma_{k}} \abs{c_{k}-c_{k-1}} \norm{ \J_{c_{k} A}x_{k}  - x_{k-1} }   + \hat{c} \abs{\alpha_{k} +\gamma_{k}-\alpha_{k-1} -\gamma_{k-1}} \norm{\J_{c_{k-1} A}x_{k-1}} +\hat{c} \norm{\beta_{k}x_{k}- \beta_{k-1}x_{k-1} +\delta_{k}e_{k} -\delta_{k-1}e_{k-1} }$. Based on the assumption, it is clear that $\sum_{k =1}^{\infty} M(k) <\infty$. 
Due to \cref{eq:xk},   $(\forall k \in \mathbb{N} \smallsetminus \{0\})$, $	\norm{ x_{k+1} -x_{k}}  
	\leq  \abs{\alpha_{k} -\alpha_{k-1}} \norm{ u- \J_{c_{k-1} A}x_{k-1}} + \abs{\gamma_{k}}\norm{   \J_{c_{k} A}x_{k} - \J_{c_{k-1} A}x_{k-1}} + \abs{\alpha_{k} +\gamma_{k}-\alpha_{k-1}  -\gamma_{k-1}} \norm{\J_{c_{k-1} A}x_{k-1}} 
 +\norm{\beta_{k}x_{k}- \beta_{k-1}x_{k-1} +\delta_{k}e_{k} -\delta_{k-1}e_{k-1} }$, which, via \cref{eq:prop:xk+1xkJckto0:J}, implies
 \begin{align}\label{eq:prop:xk+1xkJckto0}
 c_{k} \norm{ x_{k+1} -x_{k}}  \leq \abs{\gamma_{k}}c_{k-1} \norm{ x_{k} -x_{k-1}} +\hat{c} \abs{\alpha_{k} -\alpha_{k-1}} \norm{ u- \J_{c_{k-1} A}x_{k-1}} + M(k).
 \end{align}
 If $\sum_{i \in \mathbb{N}} \abs{\alpha_{i+1}  -\alpha_{i} }< \infty$  (resp.\,$(\forall k \in \mathbb{N})$ $ \abs{\gamma_{k}} \neq 1$ with $\lim_{k \to \infty} \frac{\abs{\alpha_{k+1} -\alpha_{k}}}{1-\abs{\gamma_{k+1}}} =0$) is satisfied, then, by \cref{eq:prop:xk+1xkJckto0}, 
 applying \cref{prop:tkleq}\cref{prop:tkleq:conve0}  with $(\forall k \in \mathbb{N} \smallsetminus \{0\} )$ $t_{k} = c_{k-1} \norm{ x_{k} -x_{k-1}} $, $\alpha_{k}= \abs{\gamma_{i}} $, $\beta_{k} =\omega_{k} \equiv0$, and $\gamma_{k} =\hat{c} \abs{\alpha_{k} -\alpha_{k-1}} \norm{ u- \J_{c_{k-1} A}x_{k-1}} + M(k)$  (resp.\,$\beta_{k} = 1- \abs{\gamma_{k}}$, $\omega_{k} = \frac{\hat{c} \abs{\alpha_{k} -\alpha_{k-1}}}{1-\abs{\gamma_{k}}}\norm{ u- \J_{c_{k-1} A}x_{k-1}} $,  and $\gamma_{k} =  M(k)$),
 we obtain that  $\lim_{k \to \infty}  c_{k} \norm{ x_{k+1} -x_{k}} =0 $, which, in connection with $\inf_{k \in \mathbb{N}}c_{k} >0$, guarantees that $ x_{k+1} -x_{k} \to 0 $. Therefore, the proof is complete.
\end{proof}

\subsection*{Equivalence of boundedness and non-emptiness of sets of zeroes}

\cref{Theorem:ZerBounded}\cref{Theorem:ZerBounded:Rockafellar} reduces to the equivalence proved in \cite[Theorem~1]{Rockafellar1976} when $(\forall k \in \mathbb{N})$ $\gamma_{k} =\delta_{k} \equiv 1$ and $\alpha_{k}=\beta_{k} \equiv 0$ and $\inf_{k \in \mathbb{N}}c_{k} >0$.
\begin{theorem}\label{Theorem:ZerBounded} 
	 Suppose  that  $(\forall k \in \mathbb{N})$ 
	 $ \abs{   \beta_{k} +\frac{\gamma_{k}}{2} } + \abs{\frac{\gamma_{k}}{2}  }   \leq 1$, that  $\sum_{k \in \mathbb{N}} \abs{\alpha_{k} } <\infty$, $\sum_{k \in \mathbb{N}} \abs{1-\beta_{k} -\gamma_{k} } <\infty$, and $\sum_{k \in \mathbb{N}} \norm{\delta_{k}e_{k} } <\infty$,  and that one of the following statements holds. 
	\begin{enumerate}
		\item \label{Theorem:ZerBounded:Rockafellar} Suppose  that   $ \gamma_{k} \to 1$ and that $ \inf_{k \in \mathbb{N}}c_{k} >0$ or $c_{k} \to \infty$. 
		
		\item \label{Theorem:ZerBounded:inf} Suppose that    
   $(\forall k \in \mathbb{N})$   $\gamma_{k}( \beta_{k} +\gamma_{k})  \geq 0$,    that  $\liminf_{k \to \infty} \gamma_{k}	(2\beta_{k}+\gamma_{k}) > 0$, and that $ \inf_{k \in \mathbb{N}}c_{k} >0$ or $c_{k} \to \infty$. Assume that $\liminf_{k \to \infty} \abs{\gamma_{k}} >0$ and $ \sup_{k\in \mathbb{N}} \abs{\beta_{k}} < \infty$. 
	 		
		\item  \label{Theorem:ZerBounded:infcbar} Suppose that    $(\forall k \in \mathbb{N})$   $\gamma_{k}( \beta_{k} +\gamma_{k}) \geq 0$, $\gamma_{k}	(2\beta_{k}+\gamma_{k}) \geq 0$, and $\abs{  \beta_{k} + \gamma_{k}} \abs{\gamma_{k}}    \leq \max \{  1-\abs{\beta_{k}}  , 2- 2 \abs{\beta_{k} +\frac{\gamma_{k}}{2} } \}$, that  $\sum_{k \in \mathbb{N}} \gamma_{k}	(2\beta_{k}+\gamma_{k}) =\infty$,   and that $ \inf_{k \in \mathbb{N}}c_{k} >0$ and $\sum_{k \in \mathbb{N}} \abs{c_{k+1} -c_{k}} <  \infty$. Assume that $\liminf_{k \to \infty} \abs{\gamma_{k}} >0$ and $ \sup_{k\in \mathbb{N}} \abs{\beta_{k}} < \infty$. 
	\end{enumerate}
	Then $\zer A \neq \varnothing $ if and only if $(x_{k})_{k \in \mathbb{N}}$ is bounded. 
\end{theorem}

\begin{proof}
	If $\zer A \neq \varnothing $, then combine \cref{prop:xkbounded}\cref{prop:xkbounded:leq1}   with the global assumptions to deduce the boundedness of $(x_{k})_{k \in \mathbb{N}}$.
	
	Suppose that $(x_{k})_{k \in \mathbb{N}}$ is bounded. Due to \cref{prop:Rockafellar}, \cref{Theorem:ZerBounded:Rockafellar} necessitates  $\zer A \neq \varnothing $.

Suppose that \cref{Theorem:ZerBounded:inf} or \cref{Theorem:ZerBounded:infcbar} holds. Then,	 via 
	\cref{prop:Jckboundedweakcluster}\cref{prop:Jckboundedweakcluster:Jckbounded},  $(\J_{c_{k} A}x_{k} )_{k \in \mathbb{N}}$ is bounded.  Moreover, apply	\cref{prop:tildeAzerA} to ensure that there exists  $r \in \mathbb{R}_{++}$ such that  $\tilde{A} := A +\partial \iota_{B[0;r]}$  is a maximally monotone operator and that
	\begin{align}\label{eq:zerAneqvarnothing}
	\zer \tilde{A}  \neq \varnothing, \quad 	 \left( \Omega \cap \zer \tilde{A} \right) \subseteq \zer A, \quad  \text{and} \quad  (\forall  k \in \mathbb{N}) ~ \J_{c_{k} A}x_{k} = \J_{c_{k} \tilde{A}}x_{k}.
	\end{align}
	If \cref{Theorem:ZerBounded:inf} (resp.\,\cref{Theorem:ZerBounded:infcbar})  holds, then apply 	\cref{prop:xk-Jckto0}\cref{prop:xk-Jckto0:OmegaZer} (resp.\,\cref{prop:weakclustersxkJck})  with $ A $ replaced by $\tilde{A}$ to obtain that $\varnothing \neq \Omega \subseteq \zer \tilde{A}$, which, in connection with \cref{eq:zerAneqvarnothing}, establishes that $\varnothing \neq \Omega =( \Omega \cap \zer \tilde{A}) \subseteq \zer A$.	
\end{proof}

\begin{theorem} \label{theorem:ZerBoundedIFF}
	Assume that one of the following is satisfied.
	\begin{enumerate}[label=(\Roman*)]
		\item \label{theorem:ZerBoundedIFF:limsup} $ \limsup_{k \to \infty} \left(  \abs{   \beta_{k} +\frac{\gamma_{k}}{2} } + \abs{\frac{\gamma_{k}}{2}  } \right)<1$,    $ \sup_{k\in \mathbb{N}} \abs{\alpha_{k}} < \infty$, and $ \sup_{k\in \mathbb{N}} \norm{\delta_{k}e_{k}} < \infty$. 
		\item  \label{theorem:ZerBoundedIFF:eitheror} $(\forall k \in \mathbb{N})$  $  \abs{   \beta_{k} +\frac{\gamma_{k}}{2} } + \abs{\frac{\gamma_{k}}{2}  } \leq 1$, and  the following hold:
		\begin{enumerate}
			\item    $(\forall k \in \mathbb{N})$    $|\alpha_{k}| +  \abs{   \beta_{k} +\frac{\gamma_{k}}{2} } + \abs{\frac{\gamma_{k}}{2}  } \leq 1$ or $\sum_{i \in \mathbb{N}}  \abs{\alpha_{i}} < \infty$;
			\item    $\left[  (\forall k \in \mathbb{N}   \abs{   \beta_{k} +\frac{\gamma_{k}}{2} } + \abs{\frac{\gamma_{k}}{2}  } +|\delta_{k}| \leq 1 \text{ and } \sup_{i \in \mathbb{N}} \norm{e_{i}} < \infty \right]$ or $\sum_{i \in \mathbb{N}}  \norm{\delta_{i} e_{i}}< \infty$;
			\item    $ (\forall k \in \mathbb{N})  \abs{   \beta_{k} +\frac{\gamma_{k}}{2} } + \abs{\frac{\gamma_{k}}{2}  }+\abs{1 - \beta_{k}-\gamma_{k}}  \leq 1 $  or $\sum_{i \in \mathbb{N}}  \abs{1 - \beta_{i}-\gamma_{i}}< \infty$.
		\end{enumerate}	
		\item \label{theorem:ZerBoundedIFF:alpha} $(\forall k \in \mathbb{N})$ $\abs{\alpha_{k}} + \abs{   \beta_{k} +\frac{\gamma_{k}}{2} } + \abs{\frac{\gamma_{k}}{2}} \leq 1$,   $\sum_{k \in \mathbb{N}} \abs{1-\alpha_{k}-\beta_{k} -\gamma_{k}}< \infty$, and $\sum_{k \in \mathbb{N}} \norm{\delta_{k} e_{k}} < \infty$.
	\item \label{theorem:ZerBoundedIFF:delta} $\sup_{k \in \mathbb{N}} \norm{e_{k}  } < \infty$,    $\sum_{k \in \mathbb{N}} \abs{\alpha_{k} }< \infty$,  $\sum_{k \in \mathbb{N}} \abs{1-\beta_{k} -\gamma_{k} -\delta_{k}}< \infty$, and $(\forall k \in \mathbb{N})$ $ \abs{   \beta_{k} +\frac{\gamma_{k}}{2} } + \abs{\frac{\gamma_{k}}{2}  } +\abs{\delta_{k} } \leq 1$.
		\item \label{theorem:ZerBoundedIFF:to} $(\forall k \in \mathbb{N})$ $\alpha_{k} \in \left]0,1\right]$ and $\alpha_{k}+ \abs{ \beta_{k} +\frac{ \gamma_{k}}{2} }+ \abs{\frac{ \gamma_{k}}{2}  }  \leq 1$,  
	  $\frac{\delta_{k}e_{k} }{\alpha_{k}} \to 0$, and $\frac{1-\alpha_{k} -\beta_{k} -\gamma_{k} }{\alpha_{k}} \to 0$.
	\end{enumerate}
	Then the following statements hold. 
	\begin{enumerate}
		\item  \label{theorem:ZerBoundedIFF:bounded} $\zer A \neq \varnothing $ implies the boundedness of $(x_{k})_{k \in \mathbb{N}}$.
		\item  \label{theorem:ZerBoundedIFF:nonempty} Suppose additionally that  one of the following  holds.
		\begin{enumerate}
			\item  \label{theorem:ZerBoundedIFF:nonempty:ck} $ c_{k}  \to \infty$,  $\alpha_{k} \to 0$, $\beta_{k} \to 0$, $ \gamma_{k} \to 1$,  and $\delta_{k} e_{k} \to 0$.   
			\item  \label{theorem:ZerBoundedIFF:nonempty:etak}	Suppose that    
			$(\forall k \in \mathbb{N})$ $\beta_{k} +\gamma_{k} \leq 1$,  $\alpha_{k} \to 0$, $ \limsup_{k \to \infty} |\beta_{k}|  < 1$, $1- \alpha_{k} -\beta_{k} -\gamma_{k} \to 0$,  
			$0 < \liminf_{k \to \infty}  1 -\beta_{k} - \frac{\gamma_{k}}{2}  \leq \limsup_{k \to \infty}    1 -\beta_{k} - \frac{\gamma_{k}}{2}   < 1$, $ \delta_{k}  e_{k} \to 0$,  
			and	$1 -\frac{c_{k}}{c_{k+1}} \to 0$, and that $\inf_{k \in \mathbb{N}}c_{k} >0$ or $c_{k} \to \infty$.
 	
			\item  \label{theorem:ZerBoundedIFF:sum} Suppose that     $\sum_{i \in \mathbb{N}} \abs{\alpha_{i+1}  -\alpha_{i} }< \infty$ or $(\forall k \in \mathbb{N})$ $ \abs{\gamma_{k} } \neq 1$ with $\lim_{k \to \infty} \frac{\abs{\alpha_{k+1} -\alpha_{k}}}{1-\abs{\gamma_{k+1}}} =0$, that $(\forall k \in \mathbb{N})$ $\abs{\gamma_{k}} \in \left[0,1\right]$, that  $\alpha_{k} \to 0$ and
			$\alpha_{k}+\gamma_{k} \to 1$, that $\sum_{k \in \mathbb{N}} \abs{(\alpha_{k+1} +\gamma_{k+1} ) -(\alpha_{k} +\gamma_{k})} < \infty$,  $\sum_{k \in \mathbb{N}} \abs{\beta_{k}} <\infty$, $\sum_{k \in \mathbb{N}} (1-\abs{\gamma_{k}}) =\infty$, and $\sum_{k \in \mathbb{N}} \norm{ \delta_{k} e_{k}} < \infty$, and that $\inf_{k \in \mathbb{N}}c_{k} >0$ and $\sum_{k \in \mathbb{N}} \abs{c_{k+1}  -c_{k}} < \infty$.		
		\end{enumerate}
		
		Then  $\zer A \neq \varnothing $ if and only if $(x_{k})_{k \in \mathbb{N}}$ is bounded. 
	\end{enumerate}	
	
\end{theorem}

\begin{proof}
	\cref{theorem:ZerBoundedIFF:bounded}: This is clear from  \cref{prop:xkbounded} and \cref{prop:xkboundedalphak}.

	\cref{theorem:ZerBoundedIFF:nonempty}:  In view of \cref{theorem:ZerBoundedIFF:bounded}, it remains to prove that the boundedness of  $(x_{k})_{k \in \mathbb{N}}$ leads to $\zer A \neq \varnothing $.
	
In the rest of the proof we assume that $(x_{k})_{k \in \mathbb{N}}$ is bounded. Then via \cite[Lemma~2.45]{BC2017}, $ \Omega \neq \varnothing$.  If \cref{theorem:ZerBoundedIFF:nonempty:ck} (resp.\,\cref{theorem:ZerBoundedIFF:nonempty:etak}) is satisfied, then $\varnothing \neq \Omega \subseteq   \zer A$ follows immediately from 
	\cref{prop:Jckboundedweakcluster}\cref{prop:Jckboundedweakcluster:contained}  
	(resp.\,\cref{prop:weakclustersXY}\cref{prop:weakclustersXY:OmegaZer}).	
	
Suppose that  \cref{theorem:ZerBoundedIFF:sum}  is satisfied. Due to \cref{prop:Jckboundedweakcluster}\cref{prop:Jckboundedweakcluster:Jckbounded},   $(\J_{c_{k} A}x_{k} )_{k \in \mathbb{N}}$ is bounded. 
	Then apply	\cref{prop:tildeAzerA} to ensure that there exists  $r \in \mathbb{R}_{++}$ such that  $\tilde{A} := A +\partial \iota_{B[0;r]}$  is a maximally monotone operator and that
	\begin{align}\label{eq:theorem:ZerBoundedIFF:zerAneqvarnothing}
	\zer \tilde{A}  \neq \varnothing, \quad 	 \left( \Omega \cap \zer \tilde{A} \right) \subseteq \zer A, \quad  \text{and} \quad  (\forall  k \in \mathbb{N}) ~ \J_{c_{k} A}x_{k} = \J_{c_{k} \tilde{A}}x_{k}.
	\end{align}
Furthermore,  apply   \cref{prop:xk+1xkJckto0} with $ A $ replaced by $\tilde{A}$ to obtain that $\varnothing \neq \Omega \subseteq \zer \tilde{A}$, which, connecting with \cref{eq:theorem:ZerBoundedIFF:zerAneqvarnothing}, establishes that $\varnothing \neq \Omega =( \Omega \cap \zer \tilde{A}) \subseteq \zer A$.	
\end{proof}


\section{Convergence of generalized proximal point algorithms} \label{sec:Convergence}
Because convergence implies boundedness, based on the equivalence of the boundedness of $(x_{k})_{k \in \mathbb{N}}$  and $\zer A \neq \varnothing$ shown in \Cref{Theorem:ZerBounded,theorem:ZerBoundedIFF},  to study the convergence of $(x_{k})_{k \in \mathbb{N}}$, we always assume  $\zer A \neq \varnothing$.
 
We first uphold our general assumptions and notations that
\begin{empheq}[box = \mybluebox]{equation*}
A: \mathcal{H} \to 2^{\mathcal{H}} \text{ is a maximally monotone operator with } \zer A \neq \varnothing,
\end{empheq}
that $u \in \mathcal{H}$ and $x_{0} \in \mathcal{H}$ are arbitrary but fixed,  and that
\begin{align} \label{eq:xk:sequence}
(\forall k \in \mathbb{N}) \quad x_{k+1} = \alpha_{k} u +\beta_{k} x_{k} + \gamma_{k} \J_{c_{k} A} (x_{k}) +\delta_{k} e_{k},
\end{align}
where  $(\forall k \in \mathbb{N})$ $e_{k} \in \mathcal{H}$, 
$c_{k} \in \mathbb{R}_{++}$, 
and $\{ \alpha_{k},  \beta_{k}, \gamma_{k}, \delta_{k}  \} \subseteq \mathbb{R}$. Recall that
\begin{empheq}[box = \mybluebox]{equation*}
\Omega  \text{ is the set of all weak sequential cluster points of } (x_{k} )_{k \in \mathbb{N}}.
\end{empheq}


\subsection*{Weak convergence}

\begin{theorem} \label{theorem:WeakConverg}
Suppose that  $(\forall k \in \mathbb{N})$ $\abs{   \beta_{k} +\frac{\gamma_{k}}{2} } + \abs{\frac{\gamma_{k}}{2}  }  \leq 1$, that $\sum_{k \in \mathbb{N}} \abs{\alpha_{k}} <\infty$,   $\sum_{k \in \mathbb{N}} \abs{1- \beta_{k} -\gamma_{k}} <\infty$, and $\sum_{k \in \mathbb{N}} \norm{\delta_{k}e_{k}} <\infty$, and that one of the following statements holds.
	\begin{enumerate}
		\item  \label{theorem:WeakConverg:Rockafellar} Assume  that  $ \gamma_{k} \to 1$ and that  $ \inf_{k \in \mathbb{N}}c_{k} >0$ or $c_{k} \to \infty$.    
 
		\item \label{theorem:WeakConverg:sum}  
		Assume that    
		 $ \limsup_{k \to \infty} \abs{\beta_{k}}  < 1$ and
		$0 < \liminf_{k \to \infty}  1 -\beta_{k} - \frac{\gamma_{k}}{2}  \leq \limsup_{k \to \infty}    1 -\beta_{k} - \frac{\gamma_{k}}{2}   < 1$,   that 
 $1 -\frac{c_{k}}{c_{k+1}} \to 0$, 
		and that  $ \inf_{k \in \mathbb{N}}c_{k} >0$ or $c_{k} \to \infty$.

\item  \label{theorem:WeakConverg:liminf} Assume  that $\inf_{k \in \mathbb{N}}  \gamma_{k}( \beta_{k} +\gamma_{k})  \geq 0$  and $\liminf_{k \to \infty} \gamma_{k}	(2\beta_{k}+\gamma_{k}) > 0$, and that $ \inf_{k \in \mathbb{N}}c_{k} >0$ or $c_{k} \to \infty$.

\item \label{theorem:WeakConverg:absck+1ck} Assume that  $\inf_{k \in \mathbb{N}} \gamma_{k}( \beta_{k} +\gamma_{k}) \geq 0$ and $ \inf_{k \in \mathbb{N}} \gamma_{k} (2\beta_{k}+\gamma_{k}) \geq 0$,  that  $\sum_{k \in \mathbb{N}} \gamma_{k}	(2\beta_{k}+\gamma_{k}) =\infty$,   that $(\forall k \in \mathbb{N})$ $\abs{  \beta_{k} + \gamma_{k}} \abs{\gamma_{k}}    \leq \max \{  1-\abs{\beta_{k}}  , 2- 2 \abs{\beta_{k} +\frac{\gamma_{k}}{2} } \}$,   and that $ \inf_{k \in \mathbb{N}}c_{k} >0$ and $\sum_{k \in \mathbb{N}} \abs{c_{k+1} -c_{k}} <  \infty$.					
	\end{enumerate}
	Then $ (x_{k})_{k \in \mathbb{N}}$ converges weakly to a point in $\zer A$.
\end{theorem}

\begin{proof}
According to	\cref{lemma:xk+1-p}\cref{lemma:xk+1-p:norm},
	\begin{align*}
(\forall p \in \zer A) (\forall k \in \mathbb{N}) \quad	\norm{x_{k+1} -p } \leq \left( \abs{   \beta_{k} +\frac{\gamma_{k}}{2} } + \abs{\frac{\gamma_{k}}{2}  } \right) \norm{ x_{k} -p } + \norm{\alpha_{k} u+\delta_{k}e_{k} - (1-\beta_{k} -\gamma_{k}) p},
	\end{align*}
	which, combining with the global assumptions and \cref{fact:SequenceConverg},   ensures that  $(\forall p \in \zer A)$ $\lim_{k \to \infty} \norm{x_{k} -p} $ exists in $\mathbb{R}_{+}$ and that $(x_{k})_{k \in \mathbb{N}}$ is bounded. 

Therefore, via \cref{fact:weakconveRockEcks}, it suffices to prove $\Omega \subseteq \zer A$ under the assumption \cref{theorem:WeakConverg:Rockafellar},   \cref{theorem:WeakConverg:sum},  \cref{theorem:WeakConverg:liminf}, or \cref{theorem:WeakConverg:absck+1ck}.

If \cref{theorem:WeakConverg:Rockafellar} is true, the desired inclusion is immediate from \cref{prop:Rockafellar}. 

Note that 	$(\forall k \in \mathbb{N})$ $\beta_{k} +\gamma_{k} \leq \abs{\beta_{k} +\gamma_{k} } \leq \abs{   \beta_{k} +\frac{\gamma_{k}}{2} } + \abs{\frac{\gamma_{k}}{2}  } $;  $\sum_{k \in \mathbb{N}} \abs{\alpha_{k}} <\infty$ and   $\sum_{k \in \mathbb{N}} \abs{1- \beta_{k} -\gamma_{k}} <\infty$ imply that $\alpha_{k} \to 0$, $1- \beta_{k} -\gamma_{k} \to 0$, and $ 1- \alpha_{k} -\beta_{k} -\gamma_{k} \to 0$.
As a consequence of \cref{prop:weakclustersXY}\cref{prop:weakclustersXY:OmegaZer}, \cref{theorem:WeakConverg:sum} implies $\Omega \subseteq \zer A$.

In addition, it is easy to see that the required inclusion is also immediate from \cref{theorem:WeakConverg:liminf} and \cref{prop:xk-Jckto0}\cref{prop:xk-Jckto0:OmegaZer} (or from \cref{theorem:WeakConverg:absck+1ck} and \cref{prop:weakclustersxkJck}).
\end{proof}

\begin{remark} \label{remark:WeakConvergence}
	We compare  \cref{theorem:WeakConverg} with related existed results on the weak convergence of generalized proximal point algorithms below. 
	\begin{enumerate}
		\item Suppose that  $(\forall k \in \mathbb{N})$  $\alpha_{k}=\beta_{k} \equiv 0$ and $\gamma_{k} =\delta_{k} \equiv 1$, and that $\inf_{k \in \mathbb{N}}c_{k} >0$ and $\sum_{k\in \mathbb{N}} \norm{e_{k}} <\infty$. Then \cref{theorem:WeakConverg}\cref{theorem:WeakConverg:Rockafellar} reduces to the weak convergence proved in \cite[Theorem~1]{Rockafellar1976}. 
		
	\item The relaxed proximal point algorithm presented in \cite[Algorithm~5.2]{Xu2002} is a special case of the scheme \cref{eq:xk:sequence} with $(\forall k \in \mathbb{N})$ $\alpha_{k}\equiv 0$,  $\beta_{k} \in \left[0,1\right[$, and $\gamma_{k} =\delta_{k} = 1-\beta_{k}$.
		Because in the Step~2 of the proof of \cite[Theorem~5.2]{Xu2002}, the author requires \enquote{repeating the proof of the second part of step~2 of Theorem~5.1} and in the Step~2 of  \cite[Theorem~5.1]{Xu2002}, $\beta_{k} \to 0$ is a critical assumption, we assume $\beta_{k} \to 0$ is a necessary  assumption of \cite[Theorem~5.2]{Xu2002}. Therefore, \cref{theorem:WeakConverg}\cref{theorem:WeakConverg:Rockafellar} is also a generalized result of \cite[Theorem~5.2]{Xu2002} which requires that $(\forall k \in \mathbb{N})$ $\alpha_{k}\equiv 0$,  $\beta_{k} \in \left[0,1-\delta\right]$ for some $\delta \in \left]0,1\right[$, $\beta_{k}\to 0$, and  $\gamma_{k} =\delta_{k} = 1-\beta_{k}$, and that $c_{k} \to \infty$ and $\sum_{k\in \mathbb{N}} \norm{e_{k}} <\infty$.
	
		\item  Consider \cref{eq:xk:sequence} with $(\forall k \in \mathbb{N})$ $\alpha_{k} \equiv 0$,  $\beta_{k} \in \left[0,1\right]$, $\gamma_{k}= 1 -\beta_{k}$, and $\delta_{k} \equiv 1$.
	In this case $(\forall k \in \mathbb{N})$ $\eta_{k} := 1- \beta_{k} -\frac{\gamma_{k}}{2}=\frac{\gamma_{k}}{2}$, so $0 < \liminf_{i \to \infty} \eta_{i} \leq \limsup_{i \to \infty}  \eta_{i} < 1$ follows immediately from $0 < \liminf_{i \to \infty} \gamma_{i} \leq \limsup_{i \to \infty}  \gamma_{i} < 2$. Moreover,  
	it is easy to see that $\bar{c} :=\inf_{k \in \mathbb{N}}c_{k} >0$ and $c_{k+1}-c_{k} \to 0$ imply that $1 -\frac{c_{k}}{c_{k+1}} \to 0$, since $\abs{1 -\frac{c_{k}}{c_{k+1}} } =\abs{\frac{c_{k+1} -c_{k}}{c_{k+1}}} \leq \frac{ \abs{c_{k+1} -c_{k}} }{\bar{c}} $.    	
	Hence, we know that 
	\cref{theorem:WeakConverg}\cref{theorem:WeakConverg:sum}  improves \cite[Theorem~3.2]{YaoNoor2008}.

	\item 	Note that if $(\forall k \in \mathbb{N})$ $\gamma_{k} \in \mathbb{R}_{+}$ and $\beta_{k} =1-\gamma_{k}$ such that $0 < \bar{\gamma} :=\inf_{k \in \mathbb{N}} \gamma_{k} \leq \hat{\gamma} :=\sup_{k \in \mathbb{N}} \gamma_{k}  <2$, then $ \inf_{k \in \mathbb{N}} \gamma_{k}( \beta_{k} +\gamma_{k}) =\inf_{k \in \mathbb{N}} \gamma_{k} \geq 0$ and $\liminf_{k \to \infty} \gamma_{k}	(2\beta_{k}+\gamma_{k}) = \liminf_{k \to \infty} \gamma_{k}	(2- \gamma_{k}) \geq  \liminf_{k \to \infty} \gamma_{k} \liminf_{k \to \infty}  2- \gamma_{k} \geq   \bar{\gamma} (2- \hat{\gamma} ) > 0$. Therefore, we see that 
		\cref{theorem:WeakConverg}\cref{theorem:WeakConverg:liminf} covers \cite[Theorem~3]{EcksteinBertsekas1992} in which the assumptions $(\forall k \in \mathbb{N})$ $\alpha_{k} \equiv 0$,  $\gamma_{k} \in \mathbb{R}_{+}$ with   $0 < \bar{\gamma} =\inf_{k \in \mathbb{N}} \gamma_{k} \leq \hat{\gamma}= \sup_{k \in \mathbb{N}} \gamma_{k}   <2$, $\beta_{k} =1-\gamma_{k}$, and $\delta_{k} \equiv \gamma_{k}$, $\inf_{k \in \mathbb{N}}c_{k} >0$, and $\sum_{k \in \mathbb{N}} \norm{e_{k}} < \infty$ are required.

	\item Note that the generalized proximal point algorithms studied in \cite[Section~3]{MarinoXu2004} are \cref{eq:xk:sequence} satisfying that $0 < \inf_{k \in \mathbb{N}} c_{k} \leq \sup_{k \in \mathbb{N}} c_{k} < \infty$ and that $(\forall k \in \mathbb{N})$ $\alpha_{k} \equiv 0$,  $\gamma_{k} \in \left]0,2\right[$, $\beta_{k} =1-\gamma_{k}$, and $\delta_{k} \equiv 1$. In this case, the conditions $\inf_{k \in \mathbb{N}} \gamma_{k}( \beta_{k} +\gamma_{k}) \geq 0$,  $\inf_{k \in \mathbb{N}} \gamma_{k}	(2\beta_{k}+\gamma_{k}) \geq 0$, $\abs{   \beta_{k} +\frac{\gamma_{k}}{2} } + \abs{\frac{\gamma_{k}}{2}  }  \leq 1$, and
	$\abs{  \beta_{k} + \gamma_{k}} \abs{\gamma_{k}}    \leq \max \{  1-\abs{\beta_{k}}  , 2- 2 \abs{\beta_{k} +\frac{\gamma_{k}}{2} } \}$ hold trivially. Moreover, since $(\forall k \in \mathbb{N})$  $\gamma_{k} (2-\gamma_{k})  = \gamma_{k}	(2\beta_{k}+\gamma_{k})$, thus 
	$\sum_{k} \gamma_{k} (2-\gamma_{k}) =\infty$ is exactly 	$\sum_{i \in \mathbb{N}} \gamma_{i}	(2\beta_{k}+\gamma_{k}) =\infty$.
	Therefore, we know that 	\cref{theorem:WeakConverg}\cref{theorem:WeakConverg:absck+1ck} 	covers \cite[Theorems~3.5 and 3.6]{MarinoXu2004}.
 	\end{enumerate}
\end{remark}


\subsection*{Strong convergence} \label{sec:StrongConvergence}

Note that based on \cref{fact:resolvent}\cref{fact:resolvent:closedconvex}, the maximal monotoneness of $A$ and the assumption $\zer A \neq \varnothing$ above guarantee that the projection $\Pro_{\zer A} u$ is well-defined. 
In this subsection, we shall specify sufficient conditions on coefficients of \cref{eq:xk:sequence} for $x_{k} \to \Pro_{\zer A} u $.
Notice that when $u=0$ (resp.\,$u=x_{0}$) in \cref{eq:xk:sequence},  the strong limit $ \Pro_{\zer A} u$   of $(x_{k} )_{k \in \mathbb{N}}$   is the minimum-norm  solution in  $\zer A$ (resp.\,the closest point to the initial point $x_{0}$ onto $\zer A$).

\begin{proposition} \label{prop:StrongConver}
	Suppose that $(x_{k})_{k \in \mathbb{N}}$ is bounded and that   $\Omega \subseteq \zer A$. 
	Suppose that $(\forall k \in \mathbb{N})$ $ \abs{   \beta_{k} +\frac{\gamma_{k}}{2} } + \abs{\frac{\gamma_{k}}{2}  }  \leq 1$ with  $\sum_{i \in \mathbb{N}} 1- \left(  \abs{   \beta_{k} +\frac{\gamma_{k}}{2} } + \abs{\frac{\gamma_{k}}{2}  }  \right)^{2} =\infty$, and
 that one of the following holds. 
	\begin{enumerate}
		\item  \label{prop:StrongConver:Alpha:sum} $(\forall k \in \mathbb{N})$ $\beta_{k} +\gamma_{k} \geq 0$ and $\left(\abs{   \beta_{k} +\frac{\gamma_{k}}{2} } + \abs{\frac{\gamma_{k}}{2}  } \right)^{2}  -\beta_{k}-\gamma_{k} \leq 0$, $\sum_{k \in \mathbb{N}} \abs{1- \alpha_{k}-\beta_{k} -\gamma_{k}} <\infty$, and $\sum_{k \in \mathbb{N}} \norm{\delta_{k}e_{k}} <\infty$.
		\item  \label{prop:StrongConver:Alpha:frac} $(\forall k \in \mathbb{N})$  $\abs{   \beta_{k} +\frac{\gamma_{k}}{2} } + \abs{\frac{\gamma_{k}}{2}  } <1 $, $\sup_{k \in \mathbb{N}} \frac{ 1-\beta_{k} -\gamma_{k}}{ 1- \left(\abs{   \beta_{k} +\frac{\gamma_{k}}{2} } + \abs{\frac{\gamma_{k}}{2}  } \right)^{2} } < \infty$, $\sum_{k \in \mathbb{N}} \abs{1- \alpha_{k}-\beta_{k} -\gamma_{k}} <\infty$, and $\sum_{k \in \mathbb{N}} \norm{\delta_{k}e_{k}} <\infty$.

			\item \label{prop:StrongConver:alpha} $(\forall k \in \mathbb{N})$  $\alpha_{k} \in \left[ 0,1 \right]$ and $\alpha_{k}+  \left(\abs{ \beta_{k} +\frac{ \gamma_{k}}{2} }+ \abs{\frac{ \gamma_{k}}{2}  }  \right)^{2} \leq 1$,  $\sum_{k \in \mathbb{N}} \abs{1- \alpha_{k}-\beta_{k} -\gamma_{k}} <\infty$,   $\sum_{k \in \mathbb{N}} \norm{\delta_{k}e_{k}} <\infty$.
		
		\item \label{prop:StrongConver:beta} $(\forall k \in \mathbb{N})$ $ \abs{   \beta_{k} +\frac{\gamma_{k}}{2} } + \abs{\frac{\gamma_{k}}{2}  }  <1$ and $\alpha_{k} \geq 0$,  $\sup_{k \in \mathbb{N}} \frac{ \alpha_{k} }{  1-  \left(\abs{ \beta_{k} +\frac{ \gamma_{k}}{2} }+ \abs{\frac{ \gamma_{k}}{2}  } \right)^{2}  } < \infty$, $\sum_{k \in \mathbb{N}} \abs{1- \alpha_{k}-\beta_{k} -\gamma_{k}} <\infty$, and $\sum_{k \in \mathbb{N}} \norm{\delta_{k}e_{k}} <\infty$.
		
		\item \label{prop:StrongConver:alphaNEQ0} $(\forall k \in \mathbb{N})$  $\alpha_{k} \in \left] 0,1 \right]$ and $\alpha_{k}+  \left( \abs{ \beta_{k} +\frac{ \gamma_{k}}{2} }+ \abs{\frac{ \gamma_{k}}{2}  }  \right)^{2} \leq 1$,  $\frac{\delta_{k}e_{k} }{\alpha_{k}} \to 0$, and $\frac{1-\alpha_{k} -\beta_{k} -\gamma_{k} }{\alpha_{k}} \to 0$.
		
		\item \label{prop:StrongConver:NEQ0}  $(\forall k \in \mathbb{N})$  $ \abs{   \beta_{k} +\frac{\gamma_{k}}{2} } + \abs{\frac{\gamma_{k}}{2}  }  <1$ and $\alpha_{k} \in \mathbb{R}_{++}$,   $\sup_{k \in \mathbb{N}} \frac{ \alpha_{k} }{  1-  \left(\abs{ \beta_{k} +\frac{ \gamma_{k}}{2} }+ \abs{\frac{ \gamma_{k}}{2}  } \right)^{2}  } < \infty$,  $\frac{\delta_{k}e_{k} }{\alpha_{k}} \to 0$, and $\frac{1-\alpha_{k} -\beta_{k} -\gamma_{k} }{\alpha_{k}} \to 0$.
	\end{enumerate}
Then $x_{k} \to \Pro_{\zer A} u$.
\end{proposition}
\begin{proof}
		Because $\Omega \subseteq \zer A$ and $(x_{k})_{k \in \mathbb{N}}$ is bounded, due to \cref{prop:weakcluster:Pu}, we get that 
	\begin{align} \label{eq:prop:StrongConver:leq0}
	\limsup_{k \to \infty} \innp{u-  \Pro_{\zer A} u, x_{k} - \Pro_{\zer A} u} \leq 0.
	\end{align}	
	Set $p:=\Pro_{\zer A} u$ and $(\forall k \in \mathbb{N})$ $T_{k}:= 2\J_{c_{k } A} -\Id  $. 	
		Since $\zer A \neq \varnothing$, due to \cref{prop:Jckboundedweakcluster}\cref{prop:Jckboundedweakcluster:Jckbounded},   we know that  $ (\J_{c_{k} A} x_{k} )_{k \in \mathbb{N}}$ and $(T_{k}x_{k})_{k \in \mathbb{N}}$ are bounded. 
	
	Denote by $(\forall k \in \mathbb{N})$  $\xi_{k} := \left( \abs{   \beta_{k} +\frac{\gamma_{k}}{2} } + \abs{\frac{\gamma_{k}}{2}  } \right)^{2}$,  $ \phi_{k}:= 1 -\beta_{k}- \gamma_{k}$,  $\varphi_{k}:= 1- \alpha_{k} -\beta_{k}- \gamma_{k}$, $F(k):= \norm{\delta_{k} e_{k} -\varphi_{k}u} $, and 
	$G(k):= F(k) +2\norm{ \left( \beta_{k} +\frac{\gamma_{k}}{2} \right)  (x_{k} -p )+ \frac{\gamma_{k}}{2} (T_{k}(x_{k}) -p)+\phi_{k} \left( u-p \right) } $. 	
	Because $p= \Pro_{\zer A} u \in \zer A$, via \cref{lemma:xk+1-p}\cref{lemma:xk+1-p:square}$\&$\cref{lemma:xk+1-p:alphak}, we have that  for every $ k \in \mathbb{N}$,
	\begin{subequations}
			\begin{align}  
		\norm{x_{k+1} -p }^{2}  
	&	\leq  \xi_{k}  \norm{ x_{k} -p }^{2}   
		+ 2  \phi_{k}   \innp{u-p, x_{k+1}-p - \delta_{k} e_{k}  + \varphi_{k} u}  + F(k)G(k); \label{eq:theorem:WeakStrong}\\
		\norm{x_{k+1} - p }^{2} & \leq  \xi_{k} \norm{ x_{k} -p}^{2}   +2 \alpha_{k}  \innp{ u-p, x_{k+1} -p} + 2 \innp{   \delta_{k} e_{k} - \varphi_{k} p, x_{k+1} -p}. \label{eq:prop:StrongConver:alpha}
		\end{align}
	\end{subequations}

	We separate the remaining proof into the following three cases.
	
	\emph{Case~1}: Assume \cref{prop:StrongConver:Alpha:sum}  or \cref{prop:StrongConver:Alpha:frac} is satisfied.  	Note that $(\forall k \in \mathbb{N})$ $0 \leq 1 -\beta_{k}- \gamma_{k} \leq 1 \Leftrightarrow 0 \leq \beta_{k} +  \gamma_{k} \leq 1$ and that $(\forall k \in \mathbb{N})$ $\beta_{k} +  \gamma_{k} \leq \abs{   \beta_{k} +\frac{\gamma_{k}}{2} } + \abs{\frac{\gamma_{k}}{2}  }  \leq 1 $.  Exploiting $\sum_{k \in \mathbb{N}} \abs{1- \alpha_{k}-\beta_{k} -\gamma_{k}} <\infty$ and  $\sum_{k \in \mathbb{N}} \norm{\delta_{k}e_{k}} <\infty$, we know that  $\varphi_{k} \to 0$, $\delta_{k}e_{k} \to 0$, and $\innp{u-p, \delta_{k}e_{k}-\varphi_{k}u} \to 0$, that $\sup_{k \in \mathbb{N}} G(k) <\infty$, and that $\sum_{k \in \mathbb{N}} F(k)G(k) <\infty$.  Moreover, combining $\varphi_{k} \to 0$ and $\delta_{k}e_{k} \to 0$ with \cref{eq:prop:StrongConver:leq0}, we observe that 
	\begin{align} \label{eq:prop:StrongConver:limsup}
	\limsup_{k \to \infty} 2  \innp{u-p, x_{k+1}-p - \delta_{k} e_{k}  + \varphi_{k} u} \leq 0.
	\end{align} 
	
If \cref{prop:StrongConver:Alpha:sum} (resp.\,\cref{prop:StrongConver:Alpha:frac}) holds, then using \cref{eq:theorem:WeakStrong} and applying \cref{prop:tkleq}\cref{prop:tkleq:conve0} (resp.\,\cref{prop:tkleq}\cref{prop:tkleq:betakalphak}) with $(\forall k \in \mathbb{N})$ $t_{k} = \norm{ x_{k} -p }^{2} $, $\alpha_{k}= \xi_{k}$, $\beta_{k} =  \phi_{k} $, $\omega_{k}=   2  \innp{u-p, x_{k+1}-p - \delta_{k} e_{k}  + \varphi_{k} u} $, and $\gamma_{k} = F(k)G(k)$,  we obtain the required convergence.

\emph{Case~2}: Assume \cref{prop:StrongConver:alpha} or \cref{prop:StrongConver:beta} is true. 
	If \cref{prop:StrongConver:alpha} (resp.\,\cref{prop:StrongConver:beta}) holds,  employing \cref{eq:prop:StrongConver:alpha} and  applying 	\cref{prop:tkleq}\cref{prop:tkleq:conve0}  (resp.\,\cref{prop:tkleq}\cref{prop:tkleq:betakalphak}) with $(\forall k \in \mathbb{N})$ $t_{k}=  \norm{ x_{k} -p }^{2} $, $\alpha_{k} =\xi_{k}$, $\beta_{k} =\alpha_{k}$, $\omega_{k} =  2\innp{ u-p, x_{k+1} -p}$, and $\gamma_{k} = 2 \innp{   \delta_{k} e_{k} - \varphi_{k} p, x_{k+1} -p}$, we get the required convergence. 
	
\emph{Case~3}: Assume \cref{prop:StrongConver:alphaNEQ0} or \cref{prop:StrongConver:NEQ0} is true. In view of \cref{eq:prop:StrongConver:alpha},  for every $k \in \mathbb{N}$,  
\begin{align} \label{eq:prop:StrongConver:beta} 
\norm{x_{k+1} - p }^{2}  	\leq   \xi_{k}  \norm{ x_{k} -p }^{2} + 2\alpha_{k} \left( \innp{u-p, x_{k+1} -p} + \innp{  \frac{\delta_{k} e_{k} }{\alpha_{k}}- \frac{ \varphi_{k}}{\alpha_{k}} p, x_{k+1} -p}  \right).
\end{align}
Because $(x_{k})_{k \in \mathbb{N}}$ is bounded,   $\frac{\delta_{k}e_{k} }{\alpha_{k}} \to 0$ and $\frac{1-\alpha_{k} -\beta_{k} -\gamma_{k} }{\alpha_{k}} \to 0$  yield   $\innp{  \frac{\delta_{k} e_{k} }{\alpha_{k}}- \frac{1-\alpha_{k} -  \beta_{k} -\gamma_{k}}{\alpha_{k}}  p, x_{k+1} -p} \to 0$.
	This in connection with \cref{eq:prop:StrongConver:leq0} leads to
	\begin{align*}
	\limsup_{k \to \infty}2 \left( \innp{u-p, x_{k+1} -p} + \innp{  \frac{\delta_{k} e_{k} }{\alpha_{k}}- \frac{ \varphi_{k}}{\alpha_{k}} p, x_{k+1} -p}  \right) \leq 0.
	\end{align*}
	If \cref{prop:StrongConver:alphaNEQ0} (resp.\,\cref{prop:StrongConver:NEQ0})  is true, then utilizing \cref{eq:prop:StrongConver:beta}  and   applying 	\cref{prop:tkleq}\cref{prop:tkleq:conve0}  (resp.\,\cref{prop:tkleq}\cref{prop:tkleq:betakalphak})  with $(\forall k \in \mathbb{N})$ $t_{k}=  \norm{ x_{k} -p }^{2} $, $\alpha_{k} =\xi_{k}$, $\beta_{k} =\alpha_{k}$, $\omega_{k} =  2 \left( \innp{u-p, x_{k+1} -p} + \innp{  \frac{\delta_{k} e_{k} }{\alpha_{k}}- \frac{ \varphi_{k}}{\alpha_{k}} p, x_{k+1} -p}  \right)$, and $\gamma_{k} \equiv 0$, we get the desired convergence. 

 Altogether, the proof is complete.
\end{proof}
 
Note that \Cref{prop:xkbounded,prop:xkboundedalphak} provide sufficient conditions for the boundedness of $(x_{k})_{k \in \mathbb{N}}$, 
that
 \Cref{prop:Jckboundedweakcluster,prop:Rockafellar,prop:weakclustersXY,prop:xk-Jckto0,prop:weakclustersxkJck,prop:xk+1xkJckto0} specify conditions for $\Omega \subseteq \zer A$,
  and that  \cref{prop:StrongConver} present conditions for the strong convergence under the assumptions of the boundedness of $(x_{k})_{k \in \mathbb{N}}$  and  $\Omega \subseteq \zer A$. Clearly, combining these results, we are able to deduce many sufficient conditions for the strong convergence of the generalized proximal point algorithm generated by the scheme \cref{eq:xk:sequence}.  For simplicity, we present only some easy sufficient conditions for the strong convergence below.

\begin{theorem} \label{theorem:StrongConvergence}
  Suppose that   $\sum_{i \in \mathbb{N}} 1- \left(  \abs{   \beta_{k} +\frac{\gamma_{k}}{2} } + \abs{\frac{\gamma_{k}}{2}  }  \right)^{2} =\infty$ and that one of the following holds.
	\begin{enumerate}
		\item  \label{theorem:StrongConvergence:Xu2002} 
		Assume that $(\forall k \in \mathbb{N})$ $\abs{\alpha_{k}} + \abs{   \beta_{k} +\frac{\gamma_{k}}{2} } + \abs{\frac{\gamma_{k}}{2}} \leq 1$, $\beta_{k} +\gamma_{k} \geq 0$, and $\left(\abs{   \beta_{k} +\frac{\gamma_{k}}{2} } + \abs{\frac{\gamma_{k}}{2}  } \right)^{2}  -\beta_{k}-\gamma_{k} \leq 0$,
		that  
		$\sum_{k \in \mathbb{N}} \abs{1-\alpha_{k}-\beta_{k} -\gamma_{k}}< \infty$ and $\sum_{k \in \mathbb{N}} \norm{\delta_{k} e_{k}} < \infty$, 
	and that  $\alpha_{k} \to 0$, $\beta_{k} \to 0$, $ \gamma_{k} \to 1$,  and $ c_{k}  \to \infty$.

		\item     \label{theorem:StrongConvergence:Wang}	
	Assume that	$(\forall k \in \mathbb{N})$ $\abs{\alpha_{k}} + \abs{   \beta_{k} +\frac{\gamma_{k}}{2} } + \abs{\frac{\gamma_{k}}{2}} \leq 1$,   $\beta_{k} +\gamma_{k} \geq 0$, and $\left(\abs{   \beta_{k} +\frac{\gamma_{k}}{2} } + \abs{\frac{\gamma_{k}}{2}  } \right)^{2}  -\beta_{k}-\gamma_{k} \leq 0$,	
		that $\sum_{k \in \mathbb{N}} \abs{1-\alpha_{k}-\beta_{k} -\gamma_{k}}< \infty$ and $\sum_{k \in \mathbb{N}} \norm{\delta_{k} e_{k}} < \infty$,
  that $\alpha_{k} \to 0$, $ \limsup_{k \to \infty} \abs{\beta_{k}}  < 1$,  
		$0 < \liminf_{k \to \infty}  1 -\beta_{k} - \frac{\gamma_{k}}{2}  \leq \limsup_{k \to \infty}    1 -\beta_{k} - \frac{\gamma_{k}}{2}   < 1$,   
		and	$1 -\frac{c_{k}}{c_{k+1}} \to 0$, 
		and that $\inf_{k \in \mathbb{N}}c_{k} >0$ or $c_{k} \to \infty$.

		\item \label{theorem:StrongConvergence:MarinoXu} Assume that $(\forall k \in \mathbb{N})$ $\abs{\alpha_{k}} + \abs{   \beta_{k} +\frac{\gamma_{k}}{2} } + \abs{\frac{\gamma_{k}}{2}} \leq 1$,  $\beta_{k} +\gamma_{k} \geq 0$, $\left(\abs{   \beta_{k} +\frac{\gamma_{k}}{2} } + \abs{\frac{\gamma_{k}}{2}  } \right)^{2}  -\beta_{k}-\gamma_{k} \leq 0$, $\abs{\gamma_{k}} \in \left[0,1\right]$, 
		that $\sum_{k \in \mathbb{N}} \abs{(\alpha_{k+1} +\gamma_{k+1} ) -(\alpha_{k} +\gamma_{k})} < \infty$, $\sum_{k \in \mathbb{N}} \abs{1-\alpha_{k}-\beta_{k} -\gamma_{k}}< \infty$, $\sum_{k \in \mathbb{N}} \abs{\beta_{k}} <\infty$, $\sum_{k \in \mathbb{N}} (1-\abs{\gamma_{k}}) =\infty$,  and $\sum_{k \in \mathbb{N}} \norm{\delta_{k} e_{k}} < \infty$,
		that  $\sum_{k \in \mathbb{N}} \abs{\alpha_{k+1}  -\alpha_{k} }< \infty$ or $(\forall k \in \mathbb{N})$ $ \abs{\gamma_{k} } \neq 1$ with $\lim_{k \to \infty} \frac{\abs{\alpha_{k+1} -\alpha_{k}}}{1-\abs{\gamma_{k+1}}} =0$, 
	that $\alpha_{k} \to 0$ and $\alpha_{k}+\gamma_{k} \to 1$,  and that $\inf_{k \in \mathbb{N}}c_{k} >0$ and $\sum_{i \in \mathbb{N}} \abs{c_{k+1}  -c_{k}} < \infty$.

	\item  \label{theorem:StrongConvergence:BM} Assume that $(\forall k \in \mathbb{N})$ $\alpha_{k} \in \left]0,1\right]$ and   $\alpha_{k}+ \abs{ \beta_{k} +\frac{ \gamma_{k}}{2} }+ \abs{\frac{ \gamma_{k}}{2}  }  \leq 1$,  
and that $\alpha_{k} \to 0$, $\beta_{k} \to 0$, $ \gamma_{k} \to 1$,  $\frac{\delta_{k}e_{k} }{\alpha_{k}} \to 0$,   $\frac{1-\alpha_{k} -\beta_{k} -\gamma_{k} }{\alpha_{k}} \to 0$, and $ c_{k}  \to \infty$.

	\end{enumerate}
	Then $x_{k} \to \Pro_{\zer A} u$.
\end{theorem}

\begin{proof}
If \cref{theorem:StrongConvergence:Xu2002} (resp.\,\cref{theorem:StrongConvergence:Wang} or \cref{theorem:StrongConvergence:MarinoXu}) holds, the boundedness of $(x_{k})_{k \in \mathbb{N}}$ comes from 
\cref{prop:xkbounded}\cref{prop:xkbounded:alpha}; $\Omega \subseteq \zer A$ follows from  
\cref{prop:Jckboundedweakcluster}\cref{prop:Jckboundedweakcluster:contained} (resp.\,\cref{prop:weakclustersXY}\cref{prop:weakclustersXY:OmegaZer} or \cref{prop:xk+1xkJckto0});  and the strong convergence  is  clear from \cref{prop:StrongConver}\cref{prop:StrongConver:Alpha:sum}. 

If  \cref{theorem:StrongConvergence:BM}  is satisfied, employing \cref{prop:xkboundedalphak} 
and \cref{prop:Jckboundedweakcluster}\cref{prop:Jckboundedweakcluster:contained}, we establish the boundedness of $(x_{k})_{k \in \mathbb{N}}$ and  $\Omega \subseteq \zer A$, respectively. Then the strong convergence follows immediately from \cref{prop:StrongConver}\cref{prop:StrongConver:alphaNEQ0}.
\end{proof}

\cref{theorem:StrongConvergence:YS}\cref{theorem:StrongConvergence:YS:u0} is  inspired by the proof of  \cite[Theorem~3.1]{YaoShahzad2012} and illustrates that comparing with the closest point to the initial point $x_{0}$ onto $\zer A$,  the minimum-norm  solution in  $\zer A$ is easier to find in some circumstances.   

\begin{theorem} \label{theorem:StrongConvergence:YS} 
		Denote by $(\forall k \in \mathbb{N})$  $\xi_{k} :=  \abs{   \beta_{k} +\frac{\gamma_{k}}{2} } + \abs{\frac{\gamma_{k}}{2}  }  $.
	Suppose that $(\forall k \in \mathbb{N})$ $\xi_{k} <1$,   
	that $\sum_{i \in \mathbb{N}} 1- \xi_{k}^{2} =\infty$, 
	that   $ \limsup_{k \to \infty} \abs{\beta_{k}}  < 1$, 
	$0 < \liminf_{k \to \infty}  1 -\beta_{k} - \frac{\gamma_{k}}{2}  \leq \limsup_{k \to \infty}    1 -\beta_{k} - \frac{\gamma_{k}}{2}   < 1$, $ \frac{\delta_{k}e_{k}}{ 1-\beta_{k} -\gamma_{k} } \to 0$, and  $1 -\frac{c_{k}}{c_{k+1}} \to 0$, that  $\sup_{k \in \mathbb{N}} \frac{ 1-\beta_{k} -\gamma_{k}}{ 1- \xi_{k}^{2} } < \infty$, and   that  $\inf_{k \in \mathbb{N}}c_{k} >0$ or $c_{k} \to \infty$.
	Suppose further that one of the following holds.
	\begin{enumerate}
		\item \label{theorem:StrongConvergence:YS:u0} $u=0$, $\sum_{k \in \mathbb{N}} \abs{1-\beta_{k} -\gamma_{k} -\delta_{k}}< \infty$,  $\sup_{k \in \mathbb{N}} \norm{e_{k} } < \infty$, $(\forall k \in \mathbb{N})$  $ \xi_{k}+\abs{\delta_{k} } \leq 1$ and  $\alpha_{k} \equiv 0$, and $1  -\beta_{k} -\gamma_{k} \to 0$.
		\item \label{theorem:StrongConvergence:YS:frac} 
		$\sup_{k \in \mathbb{N}} \norm{e_{k} } < \infty$,	$\sum_{k \in \mathbb{N}} \abs{\alpha_{k} }< \infty$, $\sum_{k \in \mathbb{N}} \abs{1-\beta_{k} -\gamma_{k} -\delta_{k}}< \infty$,		$ \frac{ 1- \alpha_{k} -\beta_{k}- \gamma_{k} }{1-\beta_{k} -\gamma_{k}} \to 0$, and $(\forall k \in \mathbb{N})$  $ \xi_{k}+\abs{\delta_{k} } \leq 1$.
		\item   \label{theorem:StrongConvergence:YS:alphak}  $(\forall k \in \mathbb{N})$  $\abs{\alpha_{k}} +\xi_{k} \leq 1$, $\alpha_{k} \to 0$, $\frac{1-\alpha_{k} -\beta_{k} -\gamma_{k}}{1 -\beta_{k} -\gamma_{k}} \to 0$, $\sum_{k \in \mathbb{N} } \norm{\delta_{k} e_{k}} < \infty$, and $\sum_{k \in \mathbb{N}} \abs{1-\alpha_{k} -\beta_{k} -\gamma_{k} } < \infty$.
	\end{enumerate}
	Then $x_{k} \to \Pro_{\zer A} u$.
\end{theorem}

\begin{proof}
We separate the proof into the following three steps. 

\emph{Step~1}: If the assumption \cref{theorem:StrongConvergence:YS:u0}  or \cref{theorem:StrongConvergence:YS:frac} (resp.\,\cref{theorem:StrongConvergence:YS:alphak}) holds, then combining the assumptions with \cref{prop:xkbounded}\cref{prop:xkbounded:delta} (resp.\,\cref{prop:xkbounded}\cref{prop:xkbounded:alpha}) and \cref{prop:weakclustersXY}\cref{prop:weakclustersXY:OmegaZer}, we deduce that   $(x_{k} )_{k \in \mathbb{N}}$ is bounded and that $\Omega \subseteq \zer A$.

\emph{Step~2}: 	Denote by $(\forall k \in \mathbb{N})$  
$ \phi_{k}:= 1 -\beta_{k}- \gamma_{k}$,  $\varphi_{k}:= 1- \alpha_{k} -\beta_{k}- \gamma_{k}$, 
\begin{align*}
 F(k):= \norm{\delta_{k} e_{k} -\varphi_{k}u} \text{ and } 
 G(k):= F(k) +2\norm{ \left( \beta_{k} +\frac{\gamma_{k}}{2} \right)  (x_{k} -p )+ \frac{\gamma_{k}}{2} (T_{k}(x_{k}) -p)+\phi_{k} \left( u-p \right) }. 
\end{align*}
Note that $(\forall k \in \mathbb{N})$ $\beta_{k} +\gamma_{k} \leq \abs{   \beta_{k} +\frac{\gamma_{k}}{2} } + \abs{\frac{\gamma_{k}}{2}  }   <1$.	Because $p := \Pro_{\zer A} u \in \zer A$, due to \cref{lemma:xk+1-p}\cref{lemma:xk+1-p:square}, we have that  for every $ k \in \mathbb{N}$,
	\begin{align}  \label{eq:theorem:StrongConvergence:YS} 
 	\norm{x_{k+1} -p }^{2}  
	\leq   \xi_{k}^{2}  \norm{ x_{k} -p }^{2}   
	+ 2  \phi_{k}   \innp{u-p, x_{k+1}-p - \delta_{k} e_{k}  + \varphi_{k} u}  + F(k)G(k). 
	\end{align}
Because 
$ \frac{\delta_{k}e_{k}}{ 1-\beta_{k} -\gamma_{k}} \to 0$ and $\sup_{k \in \mathbb{N}} G(k) < \infty$, any one of the assumptions  \cref{theorem:StrongConvergence:YS:u0}, \cref{theorem:StrongConvergence:YS:frac}, and \cref{theorem:StrongConvergence:YS:alphak} ensures that  for every $ k \in \mathbb{N}$,
\begin{align*}
 \frac{F(k) G(k)}{ \phi_{k} } 
=   \norm{ \frac{\delta_{k}e_{k} }{\phi_{k}}  - \frac{\varphi_{k} }{\phi_{k}} u} G(k) 
\leq   \left( \frac{ \norm{ \delta_{k} e_{k}} }{1-\beta_{k} -\gamma_{k}} + \frac{\abs{\varphi_{k}} }{1-\beta_{k} -\gamma_{k}} \norm{u}  \right) \sup_{k \in \mathbb{N}} G(k)  \to 0,
\end{align*}
which, connecting with \cref{prop:weakcluster:Pu},  entails that 
\begin{align*}
\limsup_{k \to \infty}   2  \innp{u-p, x_{k+1}-p - \delta_{k} e_{k}  + \varphi_{k} u} + \frac{ F(k)G(k) }{\phi_{k}}  
\leq  0.
\end{align*}

\emph{Step~3}: 
Considering \cref{eq:theorem:StrongConvergence:YS}  and applying \cref{prop:tkleq}\cref{prop:tkleq:betakalphak} with $(\forall k \in \mathbb{N})$ $t_{k}=\norm{ x_{k} -p }^{2}$, $\alpha_{k} =\xi_{k}^{2}$,  $\beta_{k}= \phi_{k}$,  $\omega_{k} =2  \innp{u-p, x_{k+1}-p - \delta_{k} e_{k}  + \varphi_{k} u}+ \frac{ F(k)G(k) }{\phi_{k}} $, and $\gamma_{k} \equiv 0$, we obtain the required results. 	 
\end{proof}

Notice that none of the  sums of
\begin{align*}
\alpha_{k} +\beta_{k} +\gamma_{k}, \quad \beta_{k} +\gamma_{k}, \quad \text{or} \quad \beta_{k} +\gamma_{k} +\delta_{k}  
\end{align*}
in the particular examples of \cref{example:SpecificInstances} is 1. Therefore, none of the examples in \cref{example:SpecificInstances} is covered in previous papers in the literature. 

\begin{example} \label{example:SpecificInstances}
	Suppose that $\zer A \neq \varnothing$ and that  one of the following holds. 
	\begin{enumerate}
		\item $(\forall k \in \mathbb{N})$ $\alpha_{k} = \frac{1}{k+3}$, $\beta_{k} = \frac{1}{k+2}$, $\gamma_{k}=\frac{k}{k+2}$, $\delta_{k} \equiv 1$, $c_{k}=k$, and $\norm{e_{k}} \leq \frac{1}{(k+2)^2}$. (Employ \cref{theorem:StrongConvergence}\cref{theorem:StrongConvergence:Xu2002}.)
		\item $u=0$, $(\forall k \in \mathbb{N})$ $\alpha_{k}\equiv 0$, $\beta_{k} = \gamma_{k} \equiv \frac{k+1}{2(k+2)} $, $\delta_{k} =\frac{1}{k+3}$, $c_{k} \equiv c \in \mathbb{R}_{++}$, and $e_{k} \to 0$. (Apply \cref{theorem:StrongConvergence:YS}\cref{theorem:StrongConvergence:YS:u0}.)
		\item $(\forall k \in \mathbb{N})$ $\alpha_{k} = \frac{1}{k+3}$, $\beta_{k} = \gamma_{k}=\frac{k}{2(k+2)}$, $\delta_{k} \equiv 1$, $c_{k} \equiv c \in \mathbb{R}_{++}$,  and $\norm{e_{k}} \leq \frac{1}{(k+2)^2}$. (Adopt \cref{theorem:StrongConvergence:YS}\cref{theorem:StrongConvergence:YS:alphak}.)
	\end{enumerate}
	Then $x_{k} \to \Pro_{\zer A}u$.
\end{example}

To end this paper, we compare our conditions for the strong convergence of generalized proximal point algorithm with that of related references in the remark below. 
\begin{remark}\label{reamark:StrongConvergence}
	\begin{enumerate}
		\item \label{reamark:StrongConvergence:Xu} The modified proximal point algorithm in \cite[Algorithm~5.1]{Xu2002} is \cref{eq:xk:sequence} with $u=x_{0}$ and $(\forall k \in \mathbb{N})$ $\alpha_{k} \in \left[0,1\right]$, $\beta_{k}=0$, and $\gamma_{k} =\delta_{k}=1-\alpha_{k}$. In this case, the conditions  $(\forall k \in \mathbb{N})$ $\abs{\alpha_{k}} + \abs{   \beta_{k} +\frac{\gamma_{k}}{2} } + \abs{\frac{\gamma_{k}}{2}} \leq 1$, $\beta_{k} +\gamma_{k} \geq 0$, and $\left(\abs{   \beta_{k} +\frac{\gamma_{k}}{2} } + \abs{\frac{\gamma_{k}}{2}  } \right)^{2}  -\beta_{k}-\gamma_{k} \leq 0$,
		$\sum_{k \in \mathbb{N}} \abs{1-\alpha_{k}-\beta_{k} -\gamma_{k}}< \infty$, 
		and $\beta_{k} \to 0$ hold trivially;  moreover, $\alpha_{k} \to 0$ implies that $\gamma_{k} \to 1$; furthermore, since  $(\forall k \in \mathbb{N})$  $ 1- \left(  \abs{   \beta_{k} +\frac{\gamma_{k}}{2} } + \abs{\frac{\gamma_{k}}{2}  }  \right)^{2} =1-\left(  1-\alpha_{k}\right)^{2} \geq \alpha_{k}$, thus $\sum_{k \in \mathbb{N}} \alpha_{k} =\infty$ necessitates  $\sum_{i \in \mathbb{N}} 1- \left(  \abs{   \beta_{k} +\frac{\gamma_{k}}{2} } + \abs{\frac{\gamma_{k}}{2}  }  \right)^{2} =\infty$. 
	
		Hence, it is clear that  \cref{theorem:StrongConvergence}\cref{theorem:StrongConvergence:Xu2002}  covers \cite[Theorem~5.1]{Xu2002}.
	
		 	\item The contraction proximal point algorithm in \cite{YaoNoor2008} is \cref{eq:xk:sequence} satisfying that  $(\forall k \in \mathbb{N})$ $\{  \alpha_{k},  \beta_{k}, \gamma_{k} \} \subseteq \left]0,1\right[$  with $\alpha_{k} +\beta_{k} +\gamma_{k} =1$ and $\delta_{k} \equiv 1$. In this case,  it is trivial that 	$(\forall k \in \mathbb{N})$ $\abs{\alpha_{k}} + \abs{   \beta_{k} +\frac{\gamma_{k}}{2} } + \abs{\frac{\gamma_{k}}{2}} \leq 1$,   $\beta_{k} +\gamma_{k} \geq 0$,   $\left(\abs{   \beta_{k} +\frac{\gamma_{k}}{2} } + \abs{\frac{\gamma_{k}}{2}  } \right)^{2}  -\beta_{k}-\gamma_{k} \leq 0$, and
		 $  \abs{1-\alpha_{k}-\beta_{k} -\gamma_{k}} =0$.  Since  $(\forall k \in \mathbb{N})$  $1- \left(  \abs{   \beta_{k} +\frac{\gamma_{k}}{2} } + \abs{\frac{\gamma_{k}}{2}  }  \right)^{2} \geq \alpha_{k}$, thus  $\sum_{k \in \mathbb{N}} \alpha_{k} =\infty$ leads to $\sum_{i \in \mathbb{N}} 1- \left(  \abs{   \beta_{k} +\frac{\gamma_{k}}{2} } + \abs{\frac{\gamma_{k}}{2}  }  \right)^{2} =\infty$. Moreover, because $(\forall k \in \mathbb{N})$  $ 1 -\beta_{k} - \frac{\gamma_{k}}{2}  =\frac{1}{2} +\frac{\alpha_{k}}{2} -\frac{\beta_{k}}{2}$, the assumptions $\alpha_{k} \to 0$ and $0 < \liminf_{k \to \infty} \beta_{k} \leq \limsup_{k \to \infty}  \beta_{k} < 1$ in \cite[Theorem~3.3]{YaoNoor2008} necessitate $0 < \liminf_{i \to \infty}1 -\beta_{k} - \frac{\gamma_{k}}{2}  \leq \limsup_{i \to \infty}  1 -\beta_{k} - \frac{\gamma_{k}}{2} < 1$. In addition,  it is easy to see that $c_{k+1}-c_{k} \to 0$ and $\inf_{k \in \mathbb{N}} c_{k} >0$ guarantee  $1 -\frac{c_{k}}{c_{k+1}} \to 0$.  
		 
		 	Hence, \cref{theorem:StrongConvergence}\cref{theorem:StrongConvergence:Wang} generalizes \cite[Theorem~3.3]{YaoNoor2008}.

 \item \label{reamark:StrongConvergence:Wang} As stated in \cite[Page~637]{BoikanyoMorosanu2010},   the expression $(\forall k \in \mathbb{N})$ $y_{k+1} =\J_{c_{k} A} \left( (1-\alpha_{k}) y_{k} +\alpha_{k} u +e_{k} \right)$ can be rewrite as $(\forall k \in \mathbb{N})$ $x_{k+1} = \alpha_{k+1} u+ (1- \alpha_{k+1})  \J_{ c_{k} A}(x_{k})   +e_{k+1}$, where $(\forall k \in \mathbb{N})$  $x_{k}:= (1-\alpha_{k})y_{k} +\alpha_{k} u +e_{k}$. Hence, \cref{eq:xk:sequence} with $(\forall k \in \mathbb{N})$   $\alpha_{k} \in \left]0,1\right[$, $\beta_{k}=0$, $\gamma_{k} =1-\alpha_{k}$, and $\delta_{k} =1$ is the generalized proximal point algorithm studied in \cite{BoikanyoMorosanu2010}, \cite{Xu2006}, and \cite{FHWang2011}. Note that in this case the conditions 
 $(\forall k \in \mathbb{N})$ $\abs{\alpha_{k}} + \abs{   \beta_{k} +\frac{\gamma_{k}}{2} } + \abs{\frac{\gamma_{k}}{2}} \leq 1$,   $\beta_{k} +\gamma_{k} \geq 0$,  and  $\left(\abs{   \beta_{k} +\frac{\gamma_{k}}{2} } + \abs{\frac{\gamma_{k}}{2}  } \right)^{2}  -\beta_{k}-\gamma_{k} \leq 0$,	
 $\sum_{k \in \mathbb{N}} \abs{1-\alpha_{k}-\beta_{k} -\gamma_{k}}< \infty$, and
 $ \limsup_{k \to \infty} \abs{\beta_{k}}  < 1$ hold trivially;   because $(\forall k \in \mathbb{N})$  $  1 -\beta_{k} - \frac{\gamma_{k}}{2} = 1-\frac{1}{2} \left( 1-\alpha_{k} \right) =\frac{1 -\alpha_{k}}{2} $, $\alpha_{k} \to 0$ implies directly $\lim_{k \to \infty}  1 -\beta_{k} - \frac{\gamma_{k}}{2}  =\frac{1}{2}  \in \left]0,1\right[$; furthermore, similarly with our statement in \cref{reamark:StrongConvergence:Xu} above, $\sum_{k \in \mathbb{N}} \alpha_{k} =\infty$ implies $\sum_{i \in \mathbb{N}} 1- \left(  \abs{   \beta_{k} +\frac{\gamma_{k}}{2} } + \abs{\frac{\gamma_{k}}{2}  }  \right)^{2} =\infty$.  
 
 Therefore, we know that \cref{theorem:StrongConvergence}\cref{theorem:StrongConvergence:Wang} and   \cref{theorem:StrongConvergence}\cref{theorem:StrongConvergence:BM}  improve \cite[Theorem~4]{FHWang2011} and \cite[Theorem~1]{BoikanyoMorosanu2010}, respectively. Moreover, because actually \cite[Theorem~4]{FHWang2011} refines \cite[Theorem~3.3]{Xu2006}, \cref{theorem:StrongConvergence}\cref{theorem:StrongConvergence:Wang}  naturally improves \cite[Theorem~3.3]{Xu2006} as well.  
	
		\item The contraction-proximal point algorithm is  \cref{eq:xk:sequence} with $(\forall k \in \mathbb{N})$   $\alpha_{k} \in \left]0, 1\right]$, $\beta_{k}=0$, $\gamma_{k} =1-\alpha_{k}$, and $\delta_{k} =1$. Repeating some analysis presented in \cref{reamark:StrongConvergence:Wang} and noticing that now $(\forall k \in \mathbb{N})$ 
	$  \abs{(\alpha_{k+1} +\gamma_{k+1} ) -(\alpha_{k} +\gamma_{k})} =0 $,   $  \abs{\beta_{k}} =0$, and  $ \frac{\abs{\alpha_{k+1} -\alpha_{k}}}{1-\abs{\gamma_{k+1}}}  =\abs{ \frac{ \alpha_{k+1} -\alpha_{k} }{ \alpha_{k+1}}  } =\abs{1- \frac{ \alpha_{k}}{\alpha_{k+1}}}$, we observe easily that 
	\cref{theorem:StrongConvergence}\cref{theorem:StrongConvergence:MarinoXu} covers \cite[Theorem~4.1]{MarinoXu2004}.
	
	\item The proximal point algorithm with error terms studied in \cite{YaoShahzad2012} is   \cref{eq:xk:sequence} such that  $u=0$,  $(\forall k \in \mathbb{N})$ $\alpha_{k} \equiv 0$ and $\{ \beta_{k}, \gamma_{k}, \delta_{k} \} \subseteq \left]0,1\right[$ with $\beta_{k}+\gamma_{k} +\delta_{k} =1$. In this case, it is trivial that  
		$(\forall k \in \mathbb{N})$  $\abs{   \beta_{k} +\frac{\gamma_{k}}{2} } + \abs{\frac{\gamma_{k}}{2}  } <1 $ and  $ \abs{   \beta_{k} +\frac{\gamma_{k}}{2} } + \abs{\frac{\gamma_{k}}{2}  } +\abs{\delta_{k} } \leq 1$,  and  that   $\sum_{k \in \mathbb{N}} \abs{1-\beta_{k} -\gamma_{k} -\delta_{k}}< \infty$. 
 	In this case, $\delta_{k}  \to 0$ and $e_{k} \to 0$ imply that $\sup_{k \in \mathbb{N}} \frac{ 1-\beta_{k} -\gamma_{k}}{ 1- \left(\abs{   \beta_{k} +\frac{\gamma_{k}}{2} } + \abs{\frac{\gamma_{k}}{2}  } \right)^{2} } < \infty$,     $\sup_{k \in \mathbb{N}} \norm{e_{k} -p} < \infty$, and $ \frac{\delta_{k}e_{k}}{ 1-\beta_{k} -\gamma_{k} } \to 0$; moreover, since $(\forall k \in \mathbb{N})$ $1-\beta_{k} - \frac{\gamma_{k}}{2} =\frac{1}{2} - \frac{\beta_{k}}{2}+\frac{\delta_{k}}{2}$, thus $\delta_{k} \to 0$ and 	$0 < \liminf_{i \to \infty} \beta_{i} \leq \limsup_{i \to \infty}  \beta_{i} < 1$ necessitate that    $0 < \liminf_{k \to \infty} 1-\beta_{k} - \frac{\gamma_{k}}{2} \leq \limsup_{k \to \infty}  1-\beta_{k} - \frac{\gamma_{k}}{2} < 1$.

	Therefore, \cref{theorem:StrongConvergence:YS}  generalizes \cite[Theorem~3.1]{YaoShahzad2012}

	\end{enumerate}
\end{remark}


\addcontentsline{toc}{section}{References}

\bibliographystyle{abbrv}

\end{document}